\documentclass[11pt]{article}

\usepackage[a4paper,margin=3cm]{geometry}
\usepackage[T1]{fontenc}
\usepackage[utf8]{inputenc}
\usepackage{microtype}

\usepackage{amsmath,amsfonts,amssymb,amsthm}

\newtheorem{theorem}{Theorem}
\newtheorem{lemma}[theorem]{Lemma}

\newtheorem{remark}[theorem]{Remark}
\newtheorem{corollary}[theorem]{Corollary}

\newtheorem{assumption}{Assumption}
\newtheorem{condition}{Condition}

\usepackage{graphicx}
\usepackage{booktabs}
\usepackage{tabularx}
\usepackage{multirow}
\usepackage{makecell}
\usepackage{threeparttable}
\usepackage{algorithm}
\usepackage{algorithmic}
\usepackage{float}
\usepackage{subfigure}
\usepackage{lscape}

\usepackage{xcolor}
\usepackage{pifont}
\usepackage{utfsym}
\usepackage{url}
\usepackage{enumitem}
\usepackage{lastpage}
\usepackage[numbers]{natbib}
\usepackage[hidelinks]{hyperref}

\newenvironment{keywords}{%
  \par\noindent\textbf{Keywords: }%
}{\par}

\newcommand{\BlackBox}{\rule{1.5ex}{1.5ex}}
\ifdefined\proof
    \renewenvironment{proof}{\par\noindent{\bf Proof\ }}{\hfill\BlackBox\\[2mm]}
\else
    \newenvironment{proof}{\par\noindent{\bf Proof\ }}{\hfill\BlackBox\\[2mm]}
\fi

\title{A Provably Convergent Plug-and-Play Framework for Stochastic Bilevel Optimization\thanks{%
This work was partially presented at International Conference on Machine Learning (ICML) 2024 (\protect\cite{chuspaba}).%
}}

\author{%
Tianshu Chu\thanks{Institute of Operations Research and Information Engineering, Beijing University of Technology; National Center for Applied Mathematics Shenzhen; \texttt{chuts@emails.bjut.edu.cn}}
\and
Dachuan Xu\thanks{Institute of Operations Research and Information Engineering, Beijing University of Technology; \texttt{xudc@bjut.edu.cn}}
\and
Wei Yao\thanks{National Center for Applied Mathematics Shenzhen, and Department of Mathematics, Southern University of Science and Technology; \texttt{yaow@sustech.edu.cn}}
\and
Chengming Yu\thanks{School of Science, Beijing University of Posts and Telecommunications; National Center for Applied Mathematics Shenzhen; \texttt{yucm@bupt.edu.cn}}
\and
Jin Zhang\thanks{Department of Mathematics, and National Center for Applied Mathematics Shenzhen, Southern University of Science and Technology; \texttt{zhangj9@sustech.edu.cn}}
}

\date{}

\begin{document}
\maketitle

\begin{abstract}

Bilevel optimization has recently attracted significant attention in machine learning due to its wide range of applications and advanced hierarchical optimization capabilities. In this paper, we propose a plug-and-play framework, named PnPBO, for developing and analyzing stochastic bilevel optimization methods.
This framework integrates both modern unbiased and biased stochastic estimators into the single-loop bilevel optimization framework introduced in \cite{dagreou2022framework}, with several improvements.
In the implementation of PnPBO, all stochastic estimators for different variables can be independently incorporated, and an additional moving average technique is applied when using an unbiased estimator for the upper-level variable.
In the theoretical analysis, we provide a unified convergence and complexity analysis for PnPBO, demonstrating that the adaptation of various stochastic estimators (including PAGE, ZeroSARAH, and mixed strategies) within the PnPBO framework achieves optimal sample complexity. Specifically, in the finite-sum setting, the resulting complexity matches the lower bound in \cite{dagreou2023lower} and is comparable to that of single-level optimization \citep{pmlr-v97-zhou19b}.
This resolves the open question of whether the optimal complexity bounds for solving bilevel optimization are identical to those for single-level optimization. Finally, we empirically validate our framework, demonstrating its effectiveness on several benchmark problems and confirming our theoretical findings.

\end{abstract}

\begin{keywords}
bilevel optimization, stochastic optimization, plug-and-play, sample complexity
\end{keywords}

\section{Introduction}

Bilevel optimization (BLO) effectively addresses challenges arising from hierarchical optimization, where the decision variables in the upper level are also involved in the lower level. In recent years, BLO has gained increasing attention due to its extensive and effective applications, including hyperparameter optimization \citep{franceschi2018bilevel}, meta-learning \citep{ji2020convergence}, continual learning \citep{NEURIPS2023_a0251e49}, and reinforcement learning \citep{shen2024principled}.

In this paper, we focus on the nonconvex-strongly-convex BLO problem under classical assumptions, formulated as follows
\begin{align}
\min_{x\in \mathbb{R}^{d_x}}\
 H(x):=f(x,y^*(x)) \label{UL}\quad
\text{s.t.}\ \ \,
y^*(x):=\arg \min_ {y\in \mathbb{R}^{d_y}}g(x,y) ,
\end{align}
where $H(x)$ denotes the total objective function, also referred to as the value function \citep{dagreou2022framework,dagreou2023lower,chen2024optimal}. The upper-level (UL) objective \( f(x,y) \) is \( L^f \)-smooth and possibly nonconvex.
The lower-level (LL) objective \( g(x,y) \) is \( L^g_1 \)-smooth and strongly convex in $y$, and its gradient \( \nabla g \) is \( L^g_2 \)-smooth.
As is the case in many applications of interest in machine learning \citep{shalev2014understanding}, the UL and LL objective functions take the finite-sum form
\begin{align}\label{pro_fini}
f(x, y) =
\frac{1}{n} \sum_{i=1}^n F_i\left(x, y \right),
\quad
g(x, y) =
\frac{1}{m} \sum_{j=1}^m G_j\left(x, y \right).
\end{align}
A widely used and effective strategy for solving the BLO problem in \eqref{UL} involves using implicit differentiation \citep{pedregosa2016hyperparameter, liao2018reviving, lorraine2020optimizing, li2022fully}, which leads to the following expression for $\nabla H(x)$ (known as the hypergradient)
\begin{equation*}
\nabla H(x) = \nabla_1 f(x, y^*(x)) - \nabla_{12}^2 g(x, y^*(x)) [\nabla_{22}^2 g(x, y^*(x))]^{-1} \nabla_2 f(x, y^*(x)).
\end{equation*}
The bilevel structure typically makes the objective $H(x)$ nonconvex, except in a few special cases. Therefore, our goal is to develop efficient gradient-based algorithms to find an $\epsilon$-stationary point of $H(x)$, specifically, a point $x$ that satisfies the inequality $\mathbb{E}\|\nabla H(x)\|^2\leq \epsilon$ in the stochastic setting \citep{Ghadimi-lan}.
To achieve this, a promising approach is to adapt stochastic methods from single-level optimization, which we refer to as stochastic estimators, to the bilevel optimization context. However, estimating hypergradients requires solving LL problems and computing inverse Hessian-vector products.

To tackle these challenges, \cite{ghadimi2018approximation} use multi-step gradient descent to approximately solve the LL problem and incorporate a truncated Neumann series to approximate the Hessian inversion.
Subsequent studies \citep{yang2021provably, khanduri2021near} leverage variance reduction techniques, such as STORM \citep{cutkosky2019momentum} or SARAH \citep{pham2020proxsarah}, improving the sample complexity to \( \tilde{\mathcal{O}}(\epsilon^{-1.5}) \), where the \( \tilde{\mathcal{O}} \) notation hides a factor of \( \log(\epsilon^{-1}) \).
Notably, methods using Neumann series approximations introduce an additional $ \tilde{\mathcal{O}} (1)$ factor in both sample complexity\footnote{The sample complexity refers to the number of calls made to stochastic gradients and Hessian (Jacobian)-vector products required to obtain an $\epsilon$-stationary point.} and batch size.

Another approach to computing the inverse Hessian-vector product is to solve the corresponding linear system.
For example,
by employing a warm-start strategy and performing one or more SGD steps to solve LL problems and the linear systems, stochastic algorithms such as AmIGO \citep{arbel2022amortized} and SOBA \citep{dagreou2022framework} achieve a sample complexity of \( \mathcal{O}(\epsilon^{-2}) \).
However, AmIGO requires a batch size that scales with $\epsilon^{-1}$, and SOBA depends on a stronger smoothness condition.
Thus, there is a gap in the complexity analysis between stochastic bilevel and single-level optimization when implementing the SGD gradient estimator.
Recently, this gap has been effectively addressed by MA-SOBA \citep{chen2024optimal} through the incorporation of an additional moving average (MA) technique for the UL variable in SOBA.

For other stochastic gradient estimators, SABA \citep{dagreou2022framework} integrates the stochastic estimator SAGA \citep{saga}, and SRBA \citep{dagreou2023lower} employs SARAH to propose algorithms with lower sample complexity. As indicated by the results of SABA in Appendix D of \cite{dagreou2022framework}, the sample complexity of SABA, under classical smoothness assumptions, differs from the performance of its stochastic estimator SAGA in single-level optimization by a $\mathcal{O}((n+m)^{1/3})$ factor. Additionally, the theoretical analysis of SRBA requires stronger smoothness conditions to achieve a better sample complexity of $\mathcal{O}((n+m)^{1/2}\epsilon^{-1})$.

In this paper, we address several open and important questions in stochastic bilevel optimization. First, we observe that existing algorithms and their theoretical analyses typically rely on a specific stochastic estimator, with case-by-case theoretical analysis. Therefore, the following question arises:\\
\textit{Q1: Can we propose a unified algorithm design and theoretical analysis framework in which stochastic estimators can be integrated flexibly and independently?}\\
Furthermore, observing that there remains a gap between stochastic bilevel and single-level optimization when using SAGA, the following question needs to be addressed:\\
\textit{Q2: Is it possible to bridge the gap between stochastic bilevel and single-level optimization when using SAGA?}\\
In addition, noting that SRBA \citep{dagreou2023lower} utilizes a double-loop structure and requires stronger smoothness conditions, one fundamental question arises:\\
\textit{Q3: How can we develop a fully single-loop algorithm for solving stochastic bilevel optimization problems that achieves optimal sample complexity of $\mathcal{O}((n+m)^{1/2}\epsilon^{-1})$ under standard smoothness assumptions in the finite-sum setting?}

\subsection{Contribution}
We summarize the major contributions of this work as follows.
\begin{itemize}
    \item \textbf{A unified plug-and-play framework.}
    We propose a flexible and unified framework for stochastic BLO, referred to as PnPBO. At its core, PnPBO consists of three key features:
    (1) It integrates both unbiased and biased stochastic estimators into the single-loop BLO algorithm framework presented in \cite{dagreou2022framework}, allowing independent incorporation of stochastic estimators for different variables;
    (2) An additional MA technique is suggested when an unbiased estimator is used for the UL variable;
    (3) Clipping is applied to the implicit variable, following standard techniques in deep learning practices.
\item \textbf{New algorithms with optimal sample complexity.}
The PnPBO framework is versatile, accommodating all existing single-loop BLO algorithms.
Within this framework, we employ distinct designs of stochastic estimator combinations, resulting in some specific instances of PnPBO:
(1) SPABA, SFFBA, and MSEBA, which use the stochastic estimators PAGE \citep{li2021page}, ZeroSARAH \citep{li2021zerosarah}, and a mixed strategy, {achieving optimal sample complexity that matches the finite-sum lower bound in \cite{dagreou2023lower}}; (2) MA-SABA, which bridges the gap between single-level and BLO when using SAGA; (3) Generalized SPABA and SRMBA: These are instances of the extended PnPBO framework applied to the expectation setting, {where Generalized SPABA matches the lower bound in \cite{arjevani2023lower} and SRMBA is near-optimal under existing bounds.
See Tables~\ref{tab:finite} and~\ref{table:exp} for details.}
\item \textbf{A unified and sharp convergence and complexity analysis.}
We provide a unified convergence and complexity analysis for PnPBO, applicable to all variants of PnPBO with distinct combinations of stochastic estimators.
Although the stochastic estimators for different variables can be independently plugged in, our theoretical analysis shows that the step size (learning rate) strategies for these variables must be coupled to ensure convergence.
For further details regarding the explicit coupling step size conditions, refer to Theorems \ref{theorem_biased}, \ref{theorem_unbiased} and Remark \ref{section_coupling}. Under this unified framework, we can derive convergence results for all proposed algorithms.

\item \textbf{Empirical validation.}
We validate our framework empirically, demonstrating its effectiveness on several benchmark problems and confirming the theoretical findings.
\end{itemize}
\begin{table}[H]
\centering
\renewcommand{\arraystretch}{1.7}
\begin{center}
\begin{footnotesize}
\begin{tabular}{cccccc}
\hline
Algorithm & Stochastic Estimator &{Sample Complexity} &\makecell{UL, LL}& Gap & Reference \\ \hline
\multicolumn{1}{c}{\multirow{2}{*}{SABA}}
& \multicolumn{1}{c}{\multirow{2}{*}{SAGA}}
&  \makecell{ $\mathcal{O}((n+m)\epsilon^{-1})$}
& $C_L^{1,1}$, $C_L^{2,2}$
& \ding{51}
&\multicolumn{1}{c}{\multirow{2}{*}{\cite{dagreou2022framework}}}
\\
\multicolumn{1}{c}{}
&  \multicolumn{1}{c}{}
&  \makecell{ $\mathcal{O}((n+m)^{2/3}\epsilon^{-1})$}
&  $C_L^{2,2}$,  $C_L^{3,3}$
& \ding{51}
&\multicolumn{1}{c}{}
\\
MA-SABA
&  \makecell{SAGA+MA \& SAGA}
& \makecell{ $\mathcal{O}((n+m)^{{2}/{3}}\epsilon^{-1})$}
& $C_L^{1,1}$, $C_L^{2,2}$
& \ding{55}
& Theorem 3.5$^*$
\\
SRBA
& SARAH
&\makecell{$\mathcal{O}((n+m)^{1/2}\epsilon^{-1})$}
& $C_L^{2,2}$, $C_L^{3,3}$
& \ding{51}
& \cite{dagreou2023lower}
\\
SPABA
& PAGE
&\makecell{$\mathcal{O}((n+m)^{1/2}\epsilon^{-1})$}
& $C_L^{1,1}$, $C_L^{2,2}$
& \ding{55}
& Theorem 3.7$^{*}$\\
SFFBA
& ZeroSARAH
&\makecell{$\mathcal{O}((n+m)^{1/2}\epsilon^{-1})$}
& $C_L^{1,1}$, $C_L^{2,2}$
& \ding{55}
& Theorem \ref{thsffba}\\
MSEBA
& ZeroSARAH \& PAGE
&\makecell{ $\mathcal{O}((n+m)^{1/2}\epsilon^{-1})$}
& $C_L^{1,1}$, $C_L^{2,2}$
& \ding{55}
& Theorem \ref{thmseba}\\
\hline
\multicolumn{6}{c}{ Lower Bound: $\Omega((n+m)^{1/2}\epsilon^{-1})$ \citep{dagreou2023lower}} \\ \hline
\end{tabular}
\end{footnotesize}
\end{center}

\caption{Comparison of stochastic bilevel optimization algorithms in the finite-sum setting. \\
{\footnotesize
	Following \citet{dagreou2023lower}, the finite-sum complexity comparison
	is made in the regime $n+m\lesssim \epsilon^{-2}$. In this regime, the one-time $O(n+m)$ initialization cost is subsumed by
	$O((n+m)^{1/2}\epsilon^{-1})$.
	The symbol $^*$ indicates that this result has been published in our conference paper \citep{chuspaba}.
 \( C_L^{p,p} \) means that for all \( i \in [n] \), \( F_i \in C^p \), and \( \nabla^k F_i \) is Lipschitz continuous for all \( 0 < k \leq p \). The same holds for \( G_j \) for all \( j \in [m] \).
 The symbol \ding{51}  indicates a gap in either sample complexity or required conditions when using this stochastic estimator for single-level versus BLO.
 We omit algorithms such as SOBA\citep{dagreou2022framework}, SUSTAIN \citep{khanduri2021near}, MRBO and VRBO \citep{yang2021provably}, which have a sample complexity of at best \( \tilde{\mathcal{O}}(\epsilon^{-1.5}) \). These algorithms, designed for objective functions in the expectation setting, are further discussed in Section \ref{subsection_expectation}.}
}
\label{tab:finite}
\end{table}

\subsection{Related Work}

\textbf{Stochastic Estimators and Lower Bounds.}
Single-level stochastic optimization has a variety of algorithms, which we refer to as stochastic estimators. A classic method is SGD \citep{4308316,SGD}, which requires \( \mathcal{O}(\epsilon^{-2}) \) sample complexity to reach an \( \epsilon \)-stationary point in non-convex optimization. Variants of SGD, such as reshuffled versions \citep{NIPS2016_c74d97b0} and other modifications \citep{lan2020first}, have been proposed to improve its performance. While SGD is widely used due to its simplicity, its variance remains constant.
To address this issue, \cite{NIPS2013_ac1dd209} introduced variance reduction techniques and proposed SVRG, which inspired more advanced algorithms like SAG \citep{NIPS2012_905056c1}, SAGA \citep{saga}, and STORM \citep{cutkosky2019momentum}.

The finite-sum structure is widely applied in single-level optimization to develop first-order algorithms that converge faster.
Several methods achieve optimal sample complexity, including SARAH \citep{pham2020proxsarah},
SPIDER \citep{fang2018spider}, SNVRG \citep{JMLR:v21:18-447},
PAGE \citep{li2021page}, ZeroSARAH \citep{li2021zerosarah},
 SILVER \citep{pmlr-v235-oko24a}, and SpiderBoost \citep{NEURIPS2019_512c5cad}. The sample complexity lower bound \( \Omega(\sqrt{n}\epsilon^{-1}) \) is established in \citep{fang2018spider,pmlr-v97-zhou19b}.

\vspace{5pt}
\noindent\textbf{Bilevel Optimization Methods.}
Other advances in stochastic bilevel optimization include: \cite{li2024provably} incorporated a without-replacement sampling strategy into SOBA. Recently, several algorithms using only first-order information have been proposed \citep{kwon2023fully, yang2023achieving, chen2023near}. However, these algorithms either lack non-asymptotic convergence guarantees, have less competitive convergence rates compared to our results,
or are limited to deterministic settings. \cite{NEURIPS2022_1a82986c, gu2021optimizing} proposed zeroth-order algorithms, but they do  not have a convergence rate guarantee.
\cite{shen2025penalty, huang2023momentum, kwon2024on} proposed stochastic algorithms for bilevel optimization with a non-strongly convex LL objective.
\section{Preliminaries}
\subsection{Notations}
Let $N=n+m$.
To simplify the notation,
we introduce the following notation
$D^x_{k;I,J}: = \frac{1}{|I|}\sum_{i \in I} \nabla_1 F_i(x_k, y_k) - \frac{1}{|J|}\sum_{j \in J} \nabla_{12}^2 G_j(x_k, y_k) z_k,$
where \( I \) and \( J \) are the sets of indices sampled for functions \( f \) and \( g \) at the $k$-th iteration.
Similarly, \(D^y_{k;J} \) and \(D^z_{k;I,J} \) are defined.
Additionally, let $y^*_k:=y^*(x_k)$ and $z^*_k:=z^*(x_k)$.
The clipping function on $z$ with radius $R>0$ is defined as
$\mathrm{Clip}(z; R)=\min \{1, {R}/{\|z\|}\} \cdot z.$

\subsection{Assumptions}
We present the assumptions for the BLO problem under study.
\begin{assumption}\label{assump UL}
\begin{itemize}
\item[(a)] For any $x$, there exists $C^f>0$ such that $\|\nabla_2 f(x,y^*(x))\|\leq C^f$;
\item[(b)]
For all $i\in [n]$, the function $\nabla F_i(x,y)$ is $L^f$-Lipschitz continuous in $(x,y)$;
\item[(c)] The function $H(x)$ is bounded from below, i.e.,
	\(
	H_*
	:=
	\inf_{x\in\mathbb R^{d_x}} H(x)
	>
	-\infty
	\).
\end{itemize}
\end{assumption}

\begin{assumption}\label{assump LL}
\begin{itemize}
    \item[(a)] For any $x$, $g(x,\cdot)$ is $\mu$-strongly convex;
    \item[(b)] For all $j\in [m]$, the functions $\nabla G_j(x,y)$ and $\nabla^2G_j(x,y)$ are $L^g_1$ and $L^g_2$ Lipschitz continuous in $(x,y)$, respectively.
\end{itemize}
\end{assumption}
\begin{remark}
Assumptions \ref{assump UL}(b) and \ref{assump LL}(b)
are stochastic conditions that many algorithms rely on to achieve enhanced complexity results \citep{fang2018spider, li2021page, dagreou2022framework}. They imply that \( \nabla f(x, y) \), \( \nabla g(x, y) \), and \( \nabla^2 g(x, y) \) are Lipschitz continuous with \( L^f \), \( L^g_1 \), and \( L^g_2 \), respectively.
\end{remark}

\subsection{Hypergradient Decoupling and Single-Loop Algorithm}

{We revisit the decoupling approach
to address the challenge of computing the hypergradient.
For a given $x$, we compute an approximation $y$ of the LL minimizer
\(
y^*(x)\in\arg\min_{y} g(x,y).
\)
Let the  inverse Hessian-vector product
\(
\bigl[\nabla_{22}^2 g\bigl(x, y^*(x)\bigr)\bigr]^{-1}\nabla_2 f\bigl(x, y^*(x)\bigr)
\)
be denoted by $z^*(x)$.
To obtain a computable approximation of $z^*(x)$ in a single-loop scheme, we introduce an implicit variable $z$, which is computed as an inexact solution to the following quadratic minimization problem evaluated at the available pair $(x,y)$:
{%
\setlength{\abovedisplayskip}{4pt}
\setlength{\belowdisplayskip}{4pt}
\setlength{\abovedisplayshortskip}{3pt}
\setlength{\belowdisplayshortskip}{1pt}
{\[
\min_{z\in\mathbb{R}^{d_y}} \ \frac{1}{2}\left\langle \nabla_{22}^2 g\bigl(x,y\bigr) z,\, z \right\rangle
-\left\langle \nabla_2 f\bigl(x,y\bigr),\, z \right\rangle .
\]}
}}

The process of solving the BLO problem can be viewed as consisting of three update modules: one for the UL variable \( x \), one for the LL variable \( y \), and one for the implicit variable \( z \).
Based on a warm-start strategy, the single-loop decoupling algorithms \citep{arbel2022amortized, dagreou2022framework, chen2024optimal} are structured around these update directions and can be summarized as follows.
\begin{algorithm}[h]
\begin{algorithmic}[1]
   \FOR{$k=0$ {\bfseries to} $K-1$}
    \STATE Update $x_{k+1} = x_{k}-\alpha_k \tilde{D}^x_k$,

    \quad where $\tilde{D}^x_k$ is an approximation of $D^x_k:= \nabla_1 f(x_k, y_k) - \nabla_{12}^2 g(x_k, y_k) z_k$;

   \STATE Update $ y_{k+1} = y_{k}-\beta_k \tilde{D}^y_k$,

   \quad where  $\tilde{D}^y_k$ is an approximation of $D^y_k:=\nabla_2 g(x_k, y_k)$;

   \STATE Update $z_{k+1} = z_{k}-\gamma_k \tilde{D}^z_k,$

   \quad where $\tilde{D}^z_k$ is an approximation of $D^z_k:=\nabla_{22}^2 g(x_k, y_k) z_k - \nabla_2 f(x_k, y_k)$.

   \ENDFOR
\end{algorithmic}
\end{algorithm}
\section{Method}\label{section_framework}
In this section, we introduce a unified plug-and-play framework, PnPBO, enabling the independent integration of modern stochastic estimators and offering guidance on their usage.
Furthermore, based on this framework, we propose two new algorithms as illustrative examples.
\subsection{PnPBO: A Unified Plug-and-Play Framework}\label{section_pnpbo}

We now introduce PnPBO, a flexible and unified framework for stochastic bilevel optimization, as outlined in Algorithm \ref{alg_framework}. PnPBO consists of three modules:
Lines 4-6 for updating the UL variable \( x \),
Lines 7-8 for updating the LL variable \( y \), and
Lines 9-10 for updating the implicit variable \( z \).
These modules operate in parallel and are independent.
Unlike existing algorithms that rely on specific variance reduction techniques, this framework supports the independent integration of different stochastic estimators $\mathcal{A}$, $\mathcal{B}$ and $\mathcal{C}$. They can be either unbiased or biased.

 For the implicit variable \( z \), we apply clipping with a radius \( R \) in Line 10, where {\( R:=C^f/\mu \)} serves as the upper bound for \(\| z^*(x) \|\) in Lemma \ref{Ly*}. 
 {This technique has been introduced and used in bilevel optimization \citep{pmlr-v202-hu23d,dagreou2023lower}  to control the implicit iterates \(z_k\), which allows the analysis to proceed without an a priori boundedness assumption on \(z\).}
From a computational perspective, clipping is widely adopted in deep learning practice \citep{Zhang2020Why,10614379}, and it does not introduce significant additional computation.

When \( \mathcal{A} \) is an unbiased estimator, we recommend employing the MA technique.
{Importantly, unbiasedness of \( \mathcal{A} \) (as an estimator of \(D_k^x\)) does not in general imply an unbiased hypergradient estimate, since \(D_k^x\) itself is an approximation  of \(\nabla H(x_k)\), so that in general \(\mathbb{E}[\hat{v}_k^x]\neq \nabla H(x_k)\).
We therefore introduce MA by incorporating the historical direction \(v_{k-1}^x\) into the current update to reduce the hypergradient estimation error and stabilize the \(x\)-update, which allows a larger step size \(\alpha_k\) and leads to improved sample complexity.}

\begin{algorithm}[H]
  \caption{
   PnPBO: A Unified Plug-and-Play Framework}
  \label{alg_framework}
\begin{algorithmic}[1]
   \STATE {\bfseries Input:} Initializations $v^x_{-1}$, $\rho_{-1}$, $(x_{-1},y_{-1},z_{-1})$ and $(x_{0},y_{0},z_{0})$, number of total iterations $K$,
   stochastic estimators $\mathcal{A}$, $\mathcal{B}$ and $\mathcal{C}$, step sizes $\{\alpha_k\}$, $\{\beta_k\}$, $\{\gamma_k\}$,
   MA weight $\{\rho_k\}$,  clipping radius $R$.

   \FOR{$k=0$ {\bfseries to} $K-1$}
   \STATE Sample $\mathcal{S}_k^f$ for $f$ and $\mathcal{S}_k^g$ for $g$;

    \STATE Compute an estimate \( \hat{v}_k^x \) of \( D^{x}_k\) using the stochastic estimator \( \mathcal{A} \) with \( \mathcal{S}_k^f \) and \( \mathcal{S}_k^g \);
    \STATE \begin{eqnarray*}
v_k^x=
\begin{cases}
\hat{v}_k^x, & \text{if \( \mathcal{A} \) is a biased estimator of \( D^{x}_k\)} ; \\
(1-\rho_{k-1})v_{k-1}^x + \rho_{k-1} \hat{v}_k^x, & \text{if \( \mathcal{A} \) is an unbiased estimator of \( D^{x}_k\)};
\end{cases}
\end{eqnarray*}
   \STATE Update
   \begin{equation*}
       x_{k+1} \leftarrow x_{k}-\alpha_k v_k^x;
   \end{equation*}

   \STATE Compute an estimate \( {v}_k^y \) of \( D^{y}_k\) using the stochastic estimator \( \mathcal{B} \) with \( \mathcal{S}_k^g \);
   \STATE Update
   \begin{equation*}
       y_{k+1} \leftarrow y_{k}-\beta_k v_k^y;
   \end{equation*}

   \STATE Compute an estimate \( {v}_k^z \) of \( D^{z}_k\) using the stochastic estimator \( \mathcal{C} \) with \( \mathcal{S}_k^f \) and \( \mathcal{S}_k^g \);
   \STATE Update
   \begin{equation*}
          z_{k+1} \leftarrow {\rm{Clip}} (z_{k}-\gamma_k v_k^z;R).
   \end{equation*}
   \ENDFOR
\end{algorithmic}
\end{algorithm}

\subsection{Instantiation of PnPBO}
We present two algorithms as specific implementations of PnPBO. Below, we describe the stochastic estimators used and the resulting estimates of the three update directions after independent embedding.

\subsubsection{SFFBA: Stochastic Full gradient Free Bilevel Algorithm}\label{section_alg_sffba}

We introduce a new algorithm, called SFFBA, which achieves optimal sample complexity in the finite-sum setting.
SFFBA is an adaptation of the ZeroSARAH algorithm \citep{li2021zerosarah}, originally designed for non-convex optimization, and tailored for the bilevel setting.
The key advantage of this stochastic estimator is that it eliminates the need for multiple full gradient evaluations. We apply it to all three variable update modules.

Let \( {w}_{k,i}=(w_{k,i}^x, w_{k,i}^y) \) for \( i \in [n] \) and \(\tilde{w}_{k,j} =(\tilde{w}_{k,j}^x, \tilde{w}_{k,j}^y, \tilde{w}_{k,j}^z) \) for \( j \in [m] \) represent two memory variables, corresponding to the calls to functions \( f \) and \( g \).
For each iteration $k\geq1$, we sample sets \( I \subset [n] \) and \( J \subset [m] \) for \( f \) and \( g \) with a minibatch size of \( b \).
The update rules for \(w_{k,i} \) are as follows: if \( i \in I \), then \( w_{k,i} \) is updated to \( (x_k, y_k) \); if \( i \notin I \), \( w_{k,i} \) retains its previous value \( w_{k-1,i} \).
Similarly, for \( \tilde{w}_{k,j}\), the update rule follows the same structure.

The update direction for the variable \( x \) is $v_k^x=\hat{v}_k^x$ with
 {
     \setlength{\abovedisplayskip}{1pt} 
     \setlength{\belowdisplayskip}{3pt}
\begin{align*}
    \hat{v}_k^x=&(1-\bar{\rho}_k)\left(v_{k-1}^x-D^x_{k-1;I,J}\right)
    +D^x_{{k};I,J}
    +\bar{\rho}_k\left(\hat{D}^x_{k-1;[n],[m]}-\hat{D}^x_{k-1;I,J}
    \right),
\end{align*}}
where $\bar{\rho}_k$ is the momentum parameter, and we define the notation as
{%
\setlength{\abovedisplayskip}{4pt}
\setlength{\belowdisplayskip}{3pt}
\setlength{\abovedisplayshortskip}{3pt}
\setlength{\belowdisplayshortskip}{1pt}
\begin{align*}
 \hat{D}^x_{k;I,J} :=&
    \frac{1}{|I|}\sum_{i\in I}\nabla_1 F_i(w_{k,i}^x, w_{k,i}^y) - \frac{1}{|J|}\sum_{j\in J}\nabla_{12}^2 G_j(\tilde{w}_{k,j}^x, \tilde{w}_{k,j}^y) \tilde{w}_{k,j}^z.
\end{align*}}
Similarly for \( \hat{D}^y_{k;J} \) and \( \hat{D}^z_{k;I,J} \).
The update directions for $y$ and $z$ follow a similar approach
\begin{align}
    v_k^y=&(1-\bar{\rho}_k)\left(v_{k-1}^y-D^y_{k-1;J}\right)
    +D^y_{k;J}
    +\bar{\rho}_k\left(\hat{D}^y_{k-1;[m]}-\hat{D}^y_{k-1;J}
    \right),\label{zero_y}
    \\
    v_k^z=&(1-\bar{\rho}_k)\left(v_{k-1}^z-D^z_{k-1;I,J}\right)
    +D^z_{k;I,J}
    +\bar{\rho}_k\left(\hat{D}^z_{k-1;[n],[m]}-\hat{D}^z_{k-1;I,J}
    \right).\nonumber
\end{align}

Following ZeroSARAH \citep[Algorithm~2 and Corollaries~1--2]
{li2021zerosarah}, two initialization settings can be considered.
The first uses a full batch at $k=0$. Specifically, for SFFBA, we use
the one-time full initialization, for all $i\in[n]$ and $j\in[m]$,
\begin{equation}\label{e-SFFBA-full-initialization}
w_{0,i}=(x_0,y_0),\quad
\widetilde w_{0,j}=(x_0,y_0,z_0),\quad
v_0^x=\hat v_0^x=D_0^x,\quad
v_0^y=D_0^y,\quad
v_0^z=D_0^z.
\end{equation}
No full-gradient evaluation is required thereafter.
The second uses a minibatch at $k=0$ and avoids full initialization,
as in ZeroSARAH \citep[Corollary~2]{li2021zerosarah}. In this setting, the initial update
directions are computed from the sampled minibatch rather than the full
data set.

\subsubsection{
MSEBA: Multiple Stochastic Estimators Bilevel Algorithm}

Next, we introduce MSEBA, an algorithm that integrates different stochastic estimators, including PAGE and ZeroSARAH.
This algorithm demonstrates the flexibility of PnPBO and provides an example for its plug-and-play implementation.

We choose the stochastic estimators \( \mathcal{A} \) and \( \mathcal{C} \) for the UL variable \( x \) and implicit variable \( z \) update modules, respectively, as PAGE. Specifically, for $k\geq1$,  the updates are as follows
\begin{align*}
v_k^x=&
\hat{v}_k^x =
\begin{cases}
D^x_{k;[n],[m]}, & \text{w.p. } p, \\
v_{k-1}^{x} + D^x_{k;I,J} - D^x_{k-1;I,J}
, & \text{w.p. } 1-p.
\end{cases}
\end{align*}
That is, with probability (w.p.) \( p \), we compute the full gradient \( D^x_k \), and w.p. \( 1-p \), we update the previous iteration’s direction using the samples.
The update direction for $z$ is constructed analogously: with probability
$p$, $v_k^z=D_k^z$, and with probability $1-p$,
$v_k^z=v_{k-1}^z+D_{k;I,J}^z-D_{k-1;I,J}^z$.
The stochastic estimator \( \mathcal{B} \) is chosen as ZeroSARAH, and the construction of its update direction is given by \eqref{zero_y}.

One initialization setting for MSEBA is the following one-time full
initialization, where $\widetilde w_{0,j}$ is initialized for all
$j\in[m]$:
\begin{equation}\label{e-MSEBA-full-initialization}
v_0^x=\hat v_0^x=D_0^x,\quad
v_0^y=D_0^y,\quad
v_0^z=D_0^z,\quad
\widetilde w_{0,j}=(x_0,y_0,z_0).
\end{equation}
For the ZeroSARAH estimator $\mathcal B$, the minibatch initialization
described in Section~3.2.1 can alternatively be used at $k=0$, thereby
avoiding full initialization of $\mathcal B$.

\begin{remark}
    As established in Section \ref{section_mseba}, MSEBA achieves optimal sample complexity in the finite-sum setting. This algorithm serves as an example to demonstrate the theoretical feasibility of integrating different stochastic estimators into the three update modules, highlighting the flexibility of PnPBO.
   It encourages users to select estimators based on problem-specific characteristics or practical outcomes.
\end{remark}

{\section{Theoretical Analysis Framework}\label{section_Convergence_Analysis_Framework}}

In this section, we present a unified and sharp convergence and complexity analysis framework.
Such a framework would help clarify how the variance properties of stochastic estimators, combined with step size conditions, contribute to the final convergence rates and sample complexities.
The algorithm framework PnPBO and its convergence analysis framework are illustrated in Figure \ref{fig:main}.
\begin{figure}[H]
    \centering
    \includegraphics[width=1\linewidth]{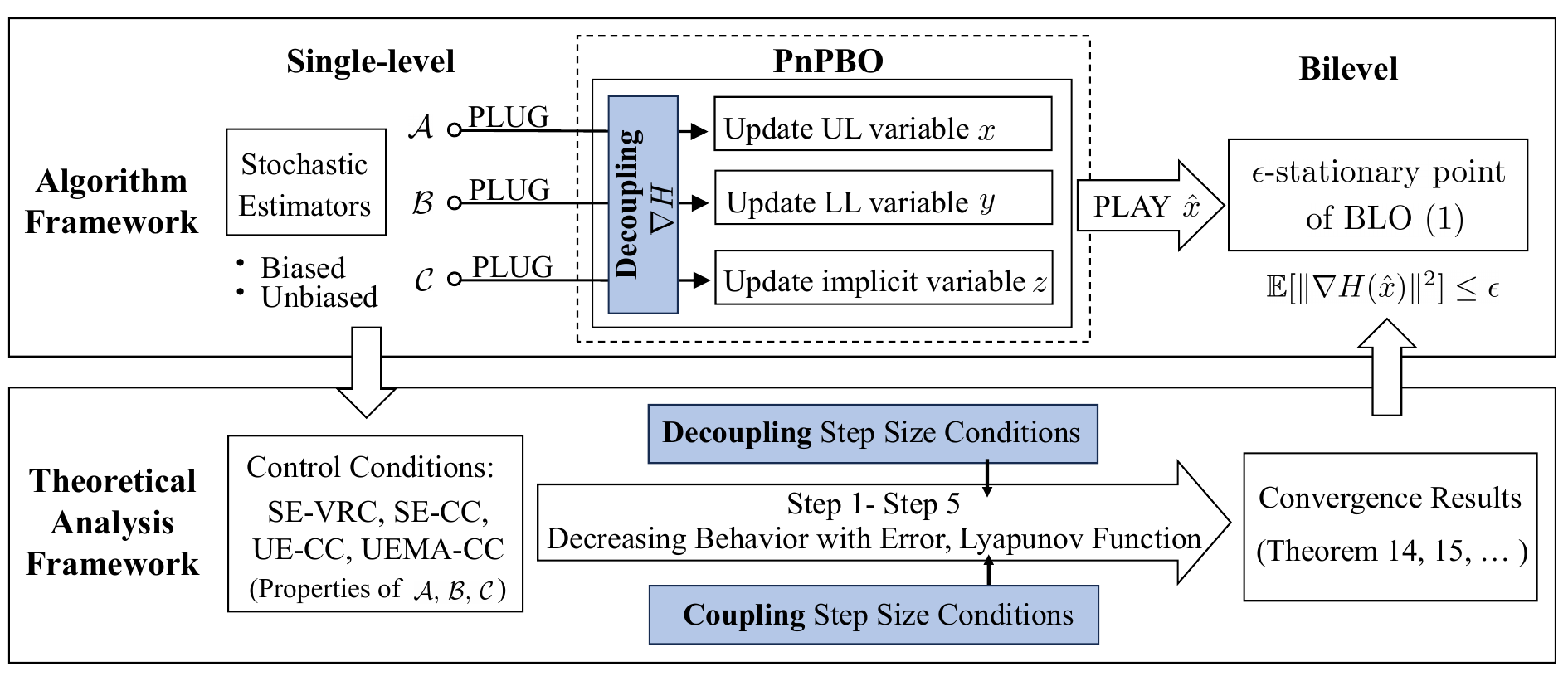}

    \caption{PnPBO and Unified Analysis Framework Schematic}
    \label{fig:main}
\end{figure}

\subsection{On the Variance Properties of Stochastic Estimators}\label{section_single}

In this subsection, we recall a generic template for single-level stochastic estimators and their variance properties.
These statements will be repeatedly invoked when we analyze the same estimators after embedding them into the bilevel framework.
Specifically, we consider the single-level non-convex problem
\(
\min_u h(u) := \frac{1}{Q}\sum_{q=1}^{Q} h_q(u).
\)

The stochastic estimator \(v_k\) is typically constructed using the current iterate \(u_k\), the sampled set \(P_k\)
and historical information such as \(\{(v_i,u_i)\}_{i=0}^{k-1}\).
Some estimators additionally maintain explicit memory variables \(\tilde{u}_{k-1}=(\tilde{u}_{k-1}^{1},\ldots,\tilde{u}_{k-1}^{Q})\), which are maintained and updated along the iterations based on the iterate and sampling histories.
Accordingly, the algorithm is driven by the recursion
\[v_k = V(\{\nabla h_q(u_k)\}_{q\in P_k}, \{v_i\}_{i=0}^{k-1}, \{u_i\}_{i=0}^{k-1}, \tilde{u}_{k-1}, \ldots ),\quad
u_{k+1} = u_k - a_k v_k,
\]
where \(a_k\) is the step size and other auxiliary variables are updated by prescribed rules.

We work with the following abstract mean-squared estimation error recursion for the estimator $\{v_k\}$:
\begin{equation}\label{lemma_single_vr}
\begin{aligned}
    \mathbb{E}[\|v_{k+1}-\nabla h(u_{k+1})\|^2]\leq &\theta_k\mathbb{E}[\|v_k-\nabla h(u_k)\|^2]
    +\eta_k'
    \mathbb{E}[\|\nabla h_{{q}_{k+1}}(u_{k+1})-\nabla h_{{q}_{k+1}}(u_{k})\|^2]
  \\
    &
    +\tau_k'\frac{1}{Q}\sum_{q=1}^Q\mathbb{E}[\|\nabla h_{{q}}(u_{k})-\nabla h_{{q}}(\tilde{u}_k^{{q}})\|^2]
    +\Delta_k,
 \end{aligned}
 \end{equation}
where $\theta_k\in[0,1)$ and $\eta_k'\ge 0$ may depend on the batch size $p_{k+1}:=|P_{k+1}|$.
Let $q_{k+1}$ be sampled uniformly from $P_{k+1}$ (given $P_{k+1}$).
The second term on the right-hand side of \eqref{lemma_single_vr} captures the local change of a sampled component gradient along the update, and we keep it in a single-sample form to facilitate the subsequent bilevel analysis.
Assume further that each $h_q$ has an $L_h$-Lipschitz gradient. Then
\[
\eta_k'\,\mathbb{E}\bigl[\|\nabla h_{q_{k+1}}(u_{k+1})-\nabla h_{q_{k+1}}(u_k)\|^2\bigr]
\le L_h^2\eta_k' a_k^2\,\mathbb{E}\bigl[\|v_k\|^2\bigr]
=: \eta_k a_k^2\,\mathbb{E}\bigl[\|v_k\|^2\bigr].
\]
If \(v_k\) depends on memory variables, define the mean-squared staleness
\(
\delta_k:=\frac{1}{Q}\sum_{q=1}^Q \mathbb{E}\bigl[\|u_k-\tilde u_k^q\|^2\bigr].
\)
By $L_h$-smoothness, the third term on the right-hand side of \eqref{lemma_single_vr} (the memory term) is bounded by $\tau_k \delta_k$ with $\tau_k:=\tau_k' L_h^2$.
This mean-squared memory staleness \(\delta_k\) typically satisfies a contraction-type recursion up to additional terms:
\begin{align}\label{lemma_single_memory}
    \delta_{k+1}\leq \lambda_k \delta_k
    + \hat{\eta}_k
    \mathbb{E}[\|u_{k+1}-u_k\|^2]
   + \hat{\eta}_k'a_k^2
    \mathbb{E}[\|\nabla h(u_k)\|^2],
\end{align}
where $\lambda_k\in[0,1)$ and the coefficients in the above inequality may depend on the size of the data set $Q$.
The term \(\Delta_k\) collects the method-dependent residual variance not captured by the other terms (e.g., the noise in plain SGD).

Moreover, for an unbiased estimator with \(\mathbb{E}[v_k]=\nabla h(u_k)\), there exist nonnegative sequences \(\{\mathrm{Coef}_k\}_{k\ge 0}\) and \(\{\mathrm{Coef}'_k\}_{k\ge 0}\) such that
\begin{equation}\label{condition_single_unbiased}
\begin{aligned}
&\mathrm{Coef}_{k+1}\theta_k-\mathrm{Coef}_k
\le -2a_k^2, \quad
\mathrm{Coef}_{k+1}\tau_k + \mathrm{Coef}'_{k+1}\lambda_k-\mathrm{Coef}'_{k} \le 0.
\end{aligned}
\end{equation}
For a biased estimator, there exist nonnegative sequences \(\{{\mathrm{Coef}}_k^b\}_{k\ge 0}\) and \(\{{\mathrm{Coef}}^{b'}_k\}_{k\ge 0}\) such that
\begin{equation}\label{condition_single_biased}
\begin{aligned}
&\mathrm{Coef}^b_{k+1}\theta_k-\mathrm{Coef}^b_k
\le -2a_k, \quad
\mathrm{Coef}^b_{k+1}\tau_k + \mathrm{Coef}^{b'}_{k+1}\lambda_k-\mathrm{Coef}^{b'}_{k} \le 0.
\end{aligned}
\end{equation}
The inequalities  \eqref{condition_single_unbiased}-\eqref{condition_single_biased} express compatibility between step sizes and Lyapunov coefficients, which will be summarized as control conditions next.
{Appendix~\ref{sec:single_instantiations} provides instantiations of \eqref{lemma_single_vr}--\eqref{condition_single_biased} for several standard estimators.}

\subsection{Control Conditions}\label{section_control_conditions}

We summarize the single-level template in \eqref{lemma_single_vr}--\eqref{condition_single_biased} into a set of abstract control conditions, which will be used as the interface between concrete stochastic estimators and our bilevel framework PnPBO.
We begin by introducing the SE-VRC  (Stochastic Estimator Variance Recursion Control) condition that characterizes the variance behavior of stochastic estimators.
This condition applies to both biased and unbiased estimators.
    \begin{condition}[SE-VRC condition]
     A stochastic estimator is said to satisfy the SE-VRC condition if it exhibits properties of the form \eqref{lemma_single_vr} and \eqref{lemma_single_memory} in the corresponding single-level stochastic optimization problem.
\end{condition}

Under the SE-VRC condition, we introduce coefficient control conditions, including the SE-CC condition (Stochastic Estimator Coefficient Control) for biased estimators and the UE-CC condition (Unbiased Estimator Coefficient Control) for unbiased estimators. These conditions relate the step size $a_k$, Lyapunov coefficients, and variance-related coefficients.

\begin{condition}[SE-CC condition]
    A stochastic estimator $\mathcal{A}$ is said to satisfy the SE-CC condition with $\{A,\ A'\}$ if there exist sequences of nonnegative coefficients $A = \{ A_k \}$ and $A' = \{ A'_k \}$ such that \eqref{condition_single_biased} hold, i.e. for $k \geq 0$:
\[
\begin{aligned}
A_{k+1}\theta_k^{\mathcal{A}}-A_k
\le -2a_k,\quad
A_{k+1}\tau_k^{\mathcal{A}} + A'_{k+1}\lambda_k^{\mathcal{A}} - A'_k \le 0.
\end{aligned}
\]
\end{condition}
\begin{condition}[UE-CC condition]
    A stochastic estimator $\mathcal{A}$ is said to satisfy the UE-CC condition with $\{A,\ A'\}$ if there exist sequences of nonnegative coefficients $A = \{ A_k \}$ and $A' = \{ A'_k \}$ such that \eqref{condition_single_unbiased} hold, i.e. for $k \geq 0$:
    \begin{align*}
    A_{k+1}\theta_k^{\mathcal{A}} - A_k
    \leq -2a_k^2, \quad
    A_{k+1}\tau_k^{\mathcal{A}}
    + A'_{k+1}\lambda_k^{\mathcal{A}} - A'_{k}
    \leq 0.
    \end{align*}
 \end{condition}

Furthermore, our framework involves unbiased estimators combined with MA, with momentum parameter
\(\rho_k\in(0,1]\).
In the Lyapunov analysis, the MA update introduces an additional deviation term, which is controlled by an extra Lyapunov coefficient \(A''\). This motivates the following UEMA-CC (Unbiased Estimator MA Coefficient Control) condition.
\begin{condition}[UEMA-CC condition]
 We say that a stochastic estimator $\mathcal{A}$ satisfies the UEMA-CC condition with $\{A,\ A',\ A''\}$ if there exist sequences of nonnegative coefficients $A = \{ A_k \}$, $A' = \{ A'_k \}$ and $A'' = \{ A''_k \}$ such that
for $k \geq 0$, the following inequalities hold:

\begin{gather*}
    \rho_k^2A_{k+1}''\theta_k^{\mathcal{A}}
+A_{k+1}\theta_k^{\mathcal{A}}-A_k
\leq 0,
\quad
    \frac{a_k}{2} +A_{k+1}''\left(1-\rho_k\right)-A_{k}''
    \leq 0,
    \\
    \rho_k^2 A''_{k+1} \tau_k^{\mathcal{A}} + \tau_k^{\mathcal{A}} A_{k+1} + \lambda_k^{\mathcal{A}} A'_{k+1} - A'_k \leq 0,
\end{gather*}
where \(\rho_k\in(0,1]\) denotes the moving average weight.
\end{condition}

\begin{remark}
 (Discussion of the Control conditions)

(1) The SE-VRC condition characterizes the variance property that stochastic estimators should possess in single-level stochastic optimization problems.
Based on this, we can establish the validity of the inequalities in the BLO setting, specifically \eqref{lemma_vr_a}–\eqref{lemma_memory_a} and \eqref{lemma_vr_b}–\eqref{lemma_memory_c}.

 (2) The SE-CC, UE-CC, and UEMA-CC conditions provide coefficient relations between step sizes, Lyapunov coefficients, and variance-related coefficients, ensuring that the SE-VRC recursions can be combined into a telescoping Lyapunov decrease in the bilevel analysis.
\end{remark}

\subsection{Analysis Steps of Framework}
This subsection provides the five-step roadmap used throughout the paper.
Steps 1--4 establish monotone descent behaviors with error terms, depending on whether the stochastic estimators are biased or unbiased.
These behaviors can be expressed as recursive inequalities in conditional expectation
\[
\mathbb{E}\big[\widetilde{D}_{k+1} \,|\, \mathcal{F}_k\big] + \Lambda_k \leq \omega_k \widetilde{D}_k + \Omega_k,
\]
where \( \widetilde{D}_k \), \( \Lambda_k \), and \( \Omega_k \) are nonnegative quantities, and \( \omega_k \in [0,1] \) is a contraction factor.
 Step 5 constructs a Lyapunov function, which integrates the lemmas from Steps 1-4 to derive the final convergence results.

Specifically, as illustrated in Figure \ref{fig:map}: (1) Starting from {the total objective function $H(x)$} in Step 1, and based on Lemma \ref{H}, the errors in the monotonic decreasing behavior of $H(x_k)$ include the gap between the UL variable update direction and the true hypergradient, i.e., \( \mathbb{E}[\|\nabla H(x_k)-v_k^x\|^2] \).
This gap is analyzed case by case in  Corollary \ref{H_biased} and Lemma \ref{lemma_ma}, depending on whether \( \mathcal{A} \) is biased or unbiased.
(2) The approximation errors \( \mathbb{E}[\|y_k - y^*_k\|^2] \) and \( \mathbb{E}[\|z_k - z^*_k\|^2] \) appear in the errors of the descent behavior derived in Step 1 (as shown in Corollary \ref{H_biased} and Lemma \ref{lemma_ma}). We further analyze these errors separately in Step 2 and Step 3.
(3) In Step 4, we control the variance term by examining the performance of the stochastic estimators when embedded in our bilevel setting.
(4) Finally, in Step 5, we construct an appropriate Lyapunov function to integrate the above analyses and derive the convergence and complexity results.
Next, we will provide the details for each step.
\begin{figure}
    \centering
\includegraphics[width=1\linewidth]{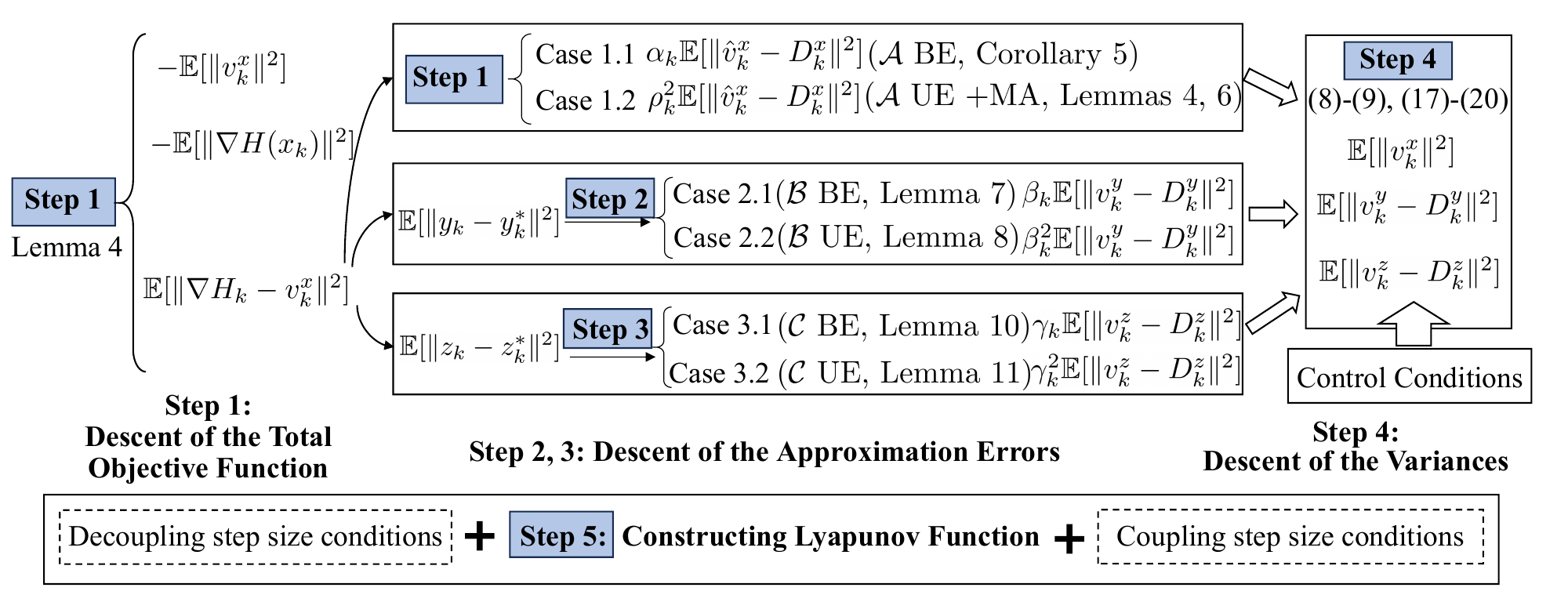}

    \caption{Roadmap of the Analysis Framework.
    {\footnotesize Above, ``BE'' represents biased stochastic estimator and ``UE'' represents unbiased stochastic estimator. In the figure, we have kept only the key terms from the formulas.}}
    \label{fig:map}
\end{figure}

\vspace{10pt}
\noindent\textit{\textbf{Step 1: Descent of the Total Objective Function \( H(x) \)}}

Under Assumptions \ref{assump UL} and \ref{assump LL}, \( H \) is \( L^H \)-smooth (Lemma \ref{Hsmooth}). Based on this, we can characterize the descent behavior of the overall objective function with errors, as described in the following lemma.

\begin{lemma}\label{H}
Suppose Assumptions \ref{assump UL} and \ref{assump LL} hold, and $\alpha_k\leq 1/(2L^H)$. Then we have
\begin{align*}
\mathbb{E}\left[H\left(x_{k+1}\right)\right]
\leq {}& \mathbb{E}\left[H\left(x_k\right)\right]
-\frac{\alpha_k}{2} \mathbb{E}\left[\left\|\nabla H\left(x_k\right)\right\|^2\right]\\
&-\frac{\alpha_k}{4}\mathbb{E}\left[\left\|v_k^x\right\|^2\right]
+\frac{\alpha_k}{2} \mathbb{E}\left[\left\|\nabla H\left(x_k\right)-v_k^x\right\|^2\right].
 \end{align*}
 \end{lemma}
Notably, Lemma \ref{H} does not depend on the type of stochastic estimator \( \mathcal{A}\).
To further quantify how the term \(\mathbb{E}\!\left[\|\nabla H(x_k)-v_k^x\|^2\right]\) affects the descent of \(H\), we establish two refined descent results, distinguishing whether \(\mathcal{A}\) is biased or unbiased.
\begin{itemize}[left=0pt, topsep=5pt, partopsep=0pt, itemsep=0pt]
    \item[]  \textit{\textbf{Case 1.1:} \( \mathcal{A} \) is a biased stochastic estimator.}

    We decompose the gap between the hypergradient and the iterative direction of
$x$, $\mathbb{E}[\|\nabla H(x_k)-v_k^x\|^2]$,  {into two terms}:
 $\mathbb{E}[\|\nabla H(x_k)-D^x_k\|^2]$ and $\mathbb{E}[\|v^x_k-D_k^x\|^2]$, leading to the following corollary.
 \begin{corollary}\label{H_biased}
 Suppose Assumptions \ref{assump UL} and \ref{assump LL} hold, and $\alpha_k\leq 1/(2L^H)$. Then we have
{
    \setlength{\abovedisplayskip}{3pt} 
    \setlength{\belowdisplayskip}{3pt}
 \begin{align*}
\mathbb{E}[H(x_{k+1})]
\leq &
\mathbb{E}[H\left(x_k\right)]
-\frac{\alpha_k}{2} \mathbb{E}\left[\left\|\nabla H\left(x_k\right)\right\|^2\right]
 -\frac{\alpha_k}{4} \mathbb{E}\left[\left\|v_k^x\right\|^2\right]
+{\alpha_k} \mathbb{E}\left[\left\|v^x_k-D_k^x\right\|^2\right]
\\&+3c_1{\alpha_k}  \mathbb{E}\left[\left\|y_k-y^*_k\right\|^2\right]
+3c_2{\alpha_k}  \mathbb{E}\left[\left\|z_k-z^*_k\right\|^2\right]
,
\end{align*}}
where $c_1=\left(L^f\right)^2+\left(L^g_2 R\right)^2$, $c_2=(L_1^g)^2$.
 \end{corollary}
\end{itemize}

\vspace{2pt}
\begin{itemize}[left=0pt, topsep=10pt, partopsep=0pt, itemsep=0pt]
    \item[]  \textit{\textbf{Case 1.2:} \( \mathcal{A} \) is an unbiased stochastic estimator.}

When \(\mathcal{A}\) is an unbiased stochastic estimator, combining it with the MA technique yields the following lemma.
\begin{lemma}\label{lemma_ma}
Supposing that Assumptions \ref{assump UL} and \ref{assump LL} hold and that \( \mathcal{A} \) is an unbiased stochastic estimator.
Then we have
\begin{align*}
&\mathbb{E}[\|v^x_{k+1}-\nabla H(x_{k+1})\|^2]
-\mathbb{E}[\|v^x_k-\nabla H(x_k)\|^2]\\
\leq{}&
-\rho_k \mathbb{E}\left[\left\|v^x_k-\nabla H\left(x_k\right)\right\|^2\right]
+ \rho_k^2\mathbb{E}\left[\left\| \hat{v}_{k+1}^x-  D^x_{k+1}\right\|^2\right]\\
&+\frac{3 \alpha_k^2\left(L^H\right)^2+12c_1 \rho_k^2\alpha_k^2}{\rho_k}
\mathbb{E}\left[\left\|v^x_k\right\|^2\right]\\
&+12\rho_k c_1 \beta_k^2\mathbb{E}\left[\left\|v_k^y\right\|^2\right]
+12\rho_k c_2\gamma_k^2\mathbb{E}\left[\left\|v_k^z\right\|^2\right]\\
&+9\rho_k c_1\mathbb{E}\left[\left\|y_k-y^*_k\right\|^2\right]
+9\rho_k c_2\mathbb{E}\left[\left\|z_k-z^*_k\right\|^2\right].
\end{align*}
\end{lemma}

\end{itemize}

\vspace{10pt}
\noindent\textit{\textbf{Step 2: Descent of the Approximation Error \( \mathbb{E}[\|y_k - y^*_k\|^2] \)}}

In this step, we evaluate the error from approximating \( y^*_k \) with \( y_k \).
We construct the corresponding lemmas separately, depending on whether \( \mathcal{B} \) is a biased or unbiased estimator.

\begin{itemize}[left=0pt, topsep=5pt, partopsep=0pt, itemsep=0pt]
    \item[]  \textit{\textbf{Case 2.1:} \( \mathcal{B} \) is a biased stochastic estimator.}
\begin{lemma}\label{y-y*_biased}
Suppose Assumptions \ref{assump UL} and \ref{assump LL} hold, \( \mathcal{B} \) is a biased stochastic estimator and the step size  satisfies
$\beta_k\leq\overline{c}$
.
Then we have
\begin{align*}
\mathbb{E}\left[\left\|y_{k+1}-y^*_{k+1}\right\|^2\right]
\leq{}&
\left(1-c_3 \beta_k\right)\mathbb{E}\left[\left\|y_k-y^*_k\right\|^2\right]
+ \frac{c_5\alpha_k^2}{\beta_k}\mathbb{E}\left[\left\|v_k^x\right\|^2\right]\\
&+c_4\beta_k\mathbb{E}\left[\| v^y_k-D_k^y\|^2\right],
\end{align*}
where
 $\overline{c}=\min\left\{\frac{\mu+L_1^g}{\mu L_1^g},\frac{1}{\mu+L_1^g}\right\}
 ,\,c_3=\frac{\mu L_1^g}{2(\mu+L_1^g)},\,c_4=6\frac{\mu+L_1^g}{\mu L_1^g},\,c_5=\frac{2(\mu+L_1^g)L_{y^*}^2}{\mu L_1^g}.$
\end{lemma}
\end{itemize}

\vspace{3pt}
\begin{itemize}[left=0pt, topsep=10pt, partopsep=0pt, itemsep=0pt]
    \item[]  \textit{\textbf{Case 2.2:} \( \mathcal{B} \) is an unbiased stochastic estimator.}
 \begin{lemma}\label{y-y*_unbiased}
Suppose Assumptions \ref{assump UL} and \ref{assump LL} hold, \( \mathcal{B} \) is an unbiased stochastic estimator and  the step size satisfies $\beta_k\leq1/(\mu+L^g_1).$
Then we have
{
    \setlength{\abovedisplayskip}{3pt} 
    \setlength{\belowdisplayskip}{3pt}
\begin{align*}
\mathbb{E}\left[\left\| y_{k+1}-y^*_{k+1}\right\|^2\right]
 \leq&(1-\beta_k \mu) \mathbb{E}\left[\left\|y_k-y^*_k\right\|^2\right] +\frac{ c_6\alpha_k^2}{\beta_k } \mathbb{E}\left[\left\|v_k^x\right\|^2\right]
+2 \beta_k^2\mathbb{E}\left[\left\|v^y_k-D_k^y\right\|^2\right],
\end{align*}}
where $c_6={2 L_{y^*}^2 }/{\mu}$.
\end{lemma}
\end{itemize}

\begin{remark}
The key difference between Lemmas~\ref{y-y*_biased} and \ref{y-y*_unbiased} is the coefficient on \(\mathbb{E}[\|v_k^y-D_k^y\|^2]\), which scales as \(\beta_k\) for biased estimators and as \(\beta_k^2\) for unbiased ones; an analogous pattern holds for Lemmas~\ref{z-z*_biased} and \ref{z-z*_unbiased}.

\end{remark}

\vspace{10pt}
\noindent\textit{\textbf{Step 3: Descent of the Approximation Error \( \mathbb{E}[\|z_k - z^*_k\|^2] \)}}

Similar to Step 2, we derive the following two lemmas based on the type of \( \mathcal{C} \) to measure the approximation error between \( z_k \) and \( z^*_k \).

\begin{itemize}[left=0pt, topsep=5pt, partopsep=0pt, itemsep=0pt]
    \item[]  \textit{\textbf{Case 3.1:} \( \mathcal{C} \) is a biased stochastic estimator.}
\begin{lemma}\label{z-z*_biased}
Suppose Assumptions \ref{assump UL} and \ref{assump LL} hold, \( \mathcal{C} \) is a biased stochastic estimator and the step size  satisfies
$\gamma_k\leq \overline{c}$
.
Then we have
{%
\setlength{\abovedisplayskip}{4pt}
\setlength{\belowdisplayskip}{4pt}
\setlength{\abovedisplayshortskip}{2pt}
\setlength{\belowdisplayshortskip}{2pt}
\begin{align*}
		\mathbb{E}\left[\left\|z_{k+1}-z^*_{k+1}\right\|^2\right]
		\leq&
		\left(1-{\frac{\mu}{4}}\gamma_k\right)\mathbb{E}\left[\left\|z_k-z^*_k\right\|^2\right]
		{+
		\frac{18c_1\gamma_k}{\mu}\mathbb{E}\left[\left \|y_k-y^*_k\right\|^2\right]}\notag
        \\
		&
        +{c_{10}}\gamma_k\mathbb{E}\left[\| v_k^z-D_k^z\|^2\right]
        + \frac{c_7\alpha_k^2}{\gamma_k }\mathbb{E}\left[\left\|v_k^x\right\|^2\right],
 	\end{align*}}
    where {$c_7=4L_{z^*}^2/\mu$, $c_{10}=12/\mu$.}
\end{lemma}
\end{itemize}
\vspace{3pt}
\begin{itemize}[left=0pt, topsep=10pt, partopsep=0pt, itemsep=0pt]
    \item[]  \textbf{Case 3.2:} \( \mathcal{C} \) is an unbiased stochastic estimator.
    \begin{lemma}\label{z-z*_unbiased}
Suppose Assumptions \ref{assump UL} and \ref{assump LL} hold, \( \mathcal{C} \) is an unbiased stochastic estimator and  the step size satisfies $\gamma_k\leq\min\{1/(10\mu),1/L_1^g\}.$
Then we have
{%
\setlength{\abovedisplayskip}{4pt}
\setlength{\belowdisplayskip}{4pt}
\setlength{\abovedisplayshortskip}{2pt}
\setlength{\belowdisplayshortskip}{2pt}
\begin{align*}
\mathbb{E}\left[\left\|z_{k+1}-z^*_{k+1}\right\|^2\right]
\leq&\left(1-\gamma_k \mu\right) \mathbb{E}\left[\left\|z_k-z^*_k\right\|^2\right]
+c_8\gamma_k \mathbb{E}\left[\left\| y_k-y^*_k\right\|^2\right]
\\&
+2 \gamma_k^2\mathbb{E}\left[\left\| v_k^z-D_k^z\right\|^2\right]
 +\frac{c_9 \alpha_k^2}{\gamma_k } \mathbb{E}\left[\left\|v_k^x\right\|^2\right],
\end{align*}}
where $c_8=8c_1/\mu$, ${c_9=3 L_{z^*}^2}/ {\mu}$.
\end{lemma}
\end{itemize}

\vspace{10pt}
\noindent\textit{\textbf{Step 4: Descent of the Variances}}

In this section, we characterize the variance behavior of single-level stochastic estimators when they are embedded in our bilevel framework.
Assuming that \( \mathcal{A} \), \( \mathcal{B} \), and \( \mathcal{C} \) in Algorithm~\ref{alg_framework} satisfy the SE-VRC condition in Section~\ref{section_control_conditions}, we establish the corresponding bilevel variance bounds as follows.
These variances include
\( \mathbb{E}[\|\hat{v}_k^x - D^x_k\|^2] \), \( \mathbb{E}[\|v_k^y - D^y_k\|^2] \), and \( \mathbb{E}[\|v_k^z - D^z_k\|^2] \).

\begin{itemize}[left=0pt]
    \item[]  \textit{\textbf{Case 4.1:} }
    For the stochastic estimator \( \mathcal{A} \) in the bilevel setting, we have
\begin{equation}\label{lemma_vr_a}
\begin{aligned}
   \mathbb{E}[\| \hat{v}_{k+1}^x - D^x_{k+1}\|^2]
    \leq &
    \theta_k^{\mathcal{A}} \mathbb{E}[\| \hat{v}_k^x - D^x_{k} \|^2]
    +
    \eta_k^{\mathcal{A}}c_1
    \alpha_k^2\mathbb{E}[\|v_k^x\|^2]
    +\eta_k^{\mathcal{A}} c_1\beta_k^2\mathbb{E}[\|v_k^y\|^2]
   \\& + \eta_k^{\mathcal{A}} c_2\gamma_k^2\mathbb{E}[\|v_k^z\|^2]
    +\tau_k^{\mathcal{A}}\delta_k^{\mathcal{A}}
     + \Delta_k^{\mathcal{A}}.
\end{aligned}
\end{equation}
Here $\delta_k^{\mathcal{A}}$ denotes the mean-squared staleness induced by the memory tables maintained by $\mathcal{A}$, and it satisfies
\begin{equation}\label{lemma_memory_a}
\begin{aligned}
\delta_{k+1}^{\mathcal{A}}\leq &
    \lambda_k^{\mathcal{A}} \delta_k^{\mathcal{A}}
    + \hat{\eta}_k^{\mathcal{A},x}\alpha_k^2
    \mathbb{E}[\|v_k^x\|^2]
    + \hat{\eta}_k^{\mathcal{A},y}\beta_k^2
    \mathbb{E}[\|v_k^y\|^2]
    + \hat{\eta}_k^{\mathcal{A},z}\gamma_k^2
    \mathbb{E}[\|v_k^z\|^2]
    \\
    &+\hat{\eta}_k'^{\mathcal{A},x}\alpha_k^2
   \mathbb{E}[\|D_k^x\|^2]
   +\hat{\eta}_k'^{\mathcal{A},y}\beta_k^2\mathbb{E}[\|D_k^y\|^2]
   +\hat{\eta}_k'^{\mathcal{A},z}\gamma_k^2\mathbb{E}[\|D_k^z\|^2].
\end{aligned}
\end{equation}
Detailed proofs are deferred to Appendix~\ref{appendix_variance}, where we also provide the corresponding
interface inequalities \eqref{lemma_vr_b}--\eqref{lemma_memory_b} for $\mathcal{B}$ and
\eqref{lemma_vr_c}--\eqref{lemma_memory_c} for $\mathcal{C}$.

\end{itemize}

\begin{remark}
In our PnPBO framework, which couples the UL and LL objectives $f$ and $g$
through the variables $(x,y,z)$, embedding a single-level estimator typically leads to recursions that
are slight adaptations of the single-level ones.
Here, we treat \eqref{lemma_vr_a}--\eqref{lemma_memory_a} and \eqref{lemma_vr_b}-\eqref{lemma_memory_c} as abstract interface inequalities for
plug-in estimators: to embed a specific estimator, one may start from its single-level derivation and
establish the corresponding bilevel counterparts term by term under Algorithm~\ref{alg_framework},
thereby identifying the involved parameters; see, for example, Lemma~\ref{lemma_SE-VRC_sffba}.
\end{remark}

\vspace{10pt}
\noindent\textit{\textbf{Step 5: Constructing  Lyapunov Function}}

We use a Lyapunov function to combine the above inequalities. Specifically, we construct the Lyapunov function as follows
\begin{align*}
L_k =
&\mathbb{E}\left[H(x_k)\right]
+A_k \mathbb{E}\left[\|\hat{v}_k^x-D^x_k\|^2\right]
+B_k \mathbb{E}\left[\|v_k^y-D^y_k\|^2\right]
+C_k \mathbb{E}\left[\|v_k^z-D^z_k\|^2\right]
+A'_k \delta_k^{\mathcal{A}}\\
&
+B'_k \delta_k^{\mathcal{B}}
+C'_k \delta_k^{\mathcal{C}}
+A''_k \mathbb{E}\left[\|v^x_k-\nabla H\left(x_k\right)\|^2\right]
+C_y\mathbb{E}\left[\|y_k-y^*_k\|^2\right]
+C_z\mathbb{E}\left[\|z_k-z^*_k\|^2\right]
.
\end{align*}
Here, all the Lyapunov coefficients are nonnegative.
By Assumption~\ref{assump UL}(c),
\(
L_k\geq \mathbb{E}[H(x_k)]\geq H_*.
\)
If \( \mathcal{A} \) is a biased estimator, set \( A''_k = 0 \).

The construction of the above Lyapunov function is analysis-driven: we systematically decouple all error sources arising in the algorithmic recursion until no further decoupling is possible. The underlying logic is as follows.
Since our stationarity criterion is based on the hypergradient $H(x_k)$, we must explicitly control both the \emph{approximation errors} and the \emph{stochastic errors} in hypergradient estimation. The approximation errors include
$\mathbb{E}\!\left[\|y_k-y_k^*\|^2\right]$ and $\mathbb{E}\!\left[\|z_k-z_k^*\|^2\right]$,
together with the additional tracking term
$\mathbb{E}[\|v_k^x - \nabla H(x_k)\|^2]$,
which captures the hypergradient-tracking improvement brought by the MA mechanism in the $x$-module.
The stochastic errors consist of the variance terms
$\mathbb{E}\!\left[\|\hat{v}_k^x-D^x_k\|^2\right]$,
$\mathbb{E}\!\left[\|v_k^y-D^y_k\|^2\right]$,
$\mathbb{E}\!\left[\|v_k^z-D^z_k\|^2\right]$,
as well as the memory (staleness) terms
$\delta_k^{\mathcal{A}}$, $\delta_k^{\mathcal{B}}$, and $\delta_k^{\mathcal{C}}$.

Combine the inequalities presented in Steps 1-4 to provide an upper bound for \( L_{k+1}-L_k \). The inequalities to be selected for estimating each term can be found in Figure \ref{fig:map}.

We choose the Lyapunov coefficients and step sizes so that the decoupling and coupling step size conditions required by our framework are satisfied (see Sections~\ref{sec:th_biased} and \ref{sec:th_unbiased}).
Under Assumptions~\ref{assump UL}--\ref{assump LL} and the corresponding control conditions, this yields
\begin{align*}
    \alpha_k\mathbb{E}\left[\left\|\nabla H(x_k)\right\|^2\right]\leq L_k-L_{k+1}+\widetilde{\Delta}_k,
\end{align*}
where {$\widetilde{\Delta}_k=(A_{k+1}
		+A_{k+1}''\rho_k^2)\Delta_k^{\mathcal{A}}
+B_{k+1}\Delta_k^{\mathcal{B}}
+C_{k+1}\Delta_k^{\mathcal{C}}$
and $\Delta_k^{\mathcal{A}},\Delta_k^{\mathcal{B}},\Delta_k^{\mathcal{C}}$ are given in
\eqref{lemma_vr_a}, \eqref{lemma_vr_b} and \eqref{lemma_vr_c}, respectively.}

Therefore, we can obtain the final convergence result
\begin{align*}
    \inf_{k< K} \mathbb{E}[\|\nabla H\left(x_k\right)\|^2]
   \leq \frac{L_0 - H_*}{K\theta}+\frac{\sum_{k=0}^{K-1}\widetilde{\Delta}_k}{K\theta},
\end{align*}
where $\theta$ is a parameter that depends on $\alpha_k$.
This bound is an abstract template: once the specific estimators are fixed, it suffices to verify the corresponding control conditions and step size relations to obtain an explicit rate.
In the next section, we instantiate our analysis framework and derive the corresponding general convergence theorems.

\section{Theoretical Results}\label{section_Convergence Results}
In this section, we establish general results for the cases where all estimators are biased and where all estimators are unbiased, as presented in Theorems \ref{theorem_biased} and \ref{theorem_unbiased}, respectively.

 Notably, our convergence analysis framework allows for the integration of both biased and unbiased estimators.
 However, due to space constraints, we omit explicit theorems for this combined case.

{\subsection{General Convergence Theorem \ref{theorem_biased}: biased estimators}\label{sec:th_biased}}

In this subsection, we present a general convergence result for the case where all stochastic estimators are biased.

Before stating the main result, we summarize the step size conditions required in Theorem \ref{theorem_biased}.
Specifically, we impose the following decoupling step size conditions
\begin{align}\label{decoupling_biased1}
&\alpha_k\leq \frac{1}{2L^H},
\,
\beta_k\leq\overline{c},
\,
\gamma_k\leq\overline{c},
\,
\left(B_{k+1}\eta_k^{\mathcal{B}}+B'_{k+1}\hat{\eta}_k^{\mathcal{B},y} \right)  \beta_k
    \leq
    M_1,
\,
B'_{k+1}\hat{\eta}_k^{\prime\mathcal{B},y} \beta_k\leq \frac{c_3^2}{36c_2},
\end{align}
and the coupling step size conditions
\begin{subequations}\label{coupling_biased}
\begin{align}
     &\alpha_k
    \leq m_{\alpha\beta}\beta_k,
    \quad
    \alpha_k
    \leq m_{\alpha\gamma}\gamma_k,
    \quad
    \gamma_k\leq m_{\gamma\beta}\beta_k,\label{coupling_biased_1}
\\[1pt]&
\left(A_{k+1}\eta_k^{\mathcal{A}}+A'_{k+1}\hat{\eta}_k^{\mathcal{A},\ell} \right)  \beta_k
    \leq
   M_2,
\quad
A'_{k+1}\hat{\eta}_k^{\prime\mathcal{A},\ell}\beta_k
\leq
\tilde{M}_2,\label{coupling_biased_a}
\\[1pt]&
\left(C_{k+1}\eta_k^{\mathcal{C}}+C'_{k+1}\hat{\eta}_k^{\mathcal{C},\ell} \right)  \beta_k
\leq M_2,
  \quad
C'_{k+1}\hat{\eta}_k^{\prime\mathcal{C},\ell}\beta_k
\leq
\tilde{M}_2,\label{coupling_biased_2}
\\[1pt]&
\left(B_{k+1}\eta_k^{\mathcal{B}}+B'_{k+1}\hat{\eta}_k^{\mathcal{B},x} \right)  \beta_k
\leq
 \frac{1}{96m_{\alpha\beta}\tilde{c}_2},
\quad
B'_{k+1}\hat{\eta}_k^{\prime\mathcal{B},x}\beta_k
\leq
\frac{1}{196m_{\alpha\beta}},
\label{coupling_biased_3}
\end{align}
\end{subequations}
where $\ell\in\{x,y,z\}$, and  the constants above are determined  by the constants in Assumptions~\ref{assump UL}--\ref{assump LL}. Specifically,
 $\tilde{c}_1=c_1+1,$ $\tilde{c}_2=c_2+1,$
$M_1=\min
    \left\{
    \frac{1}{6\tilde{c}_2},\,
    \frac{1}{4\tilde{c}_1},\,
    \frac{c_3^2}{72\tilde{c}_2^2}
    \right\},$
$ M_2=\min \left\{
\frac{1}{96m_{\alpha\beta}\tilde{c}_1}
,\,\frac{c_3^2}{72\tilde{c}_1c_2},\,
\frac{1}{6\tilde{c}_1},\,
\frac{1}{4m_{\gamma\beta}\tilde{c}_2},\,
\frac{c_3^2}{72m_{\gamma\beta}\tilde{c}_2L_z^2},\,
\frac{\mu}{40m_{\gamma\beta}\tilde{c}_2 c_{10}L_z^2}
 \right\},$
$m_{\alpha\beta}=\min\left\{\frac{3}{8L_{y^*}^2},\,\frac{c_3^2}{108c_1}\right\}$,
\\
$m_{\alpha\gamma}=\min\left\{\frac{3}{16L_{z^*}^2},\,\frac{\mu}{60c_2c_{10}}\right\}$, $m_{\gamma\beta}=\frac{ c_3^2}{54c_1}$,
 $\tilde{M}_2=\min \left\{\frac{c_3^2}{36c_2},\,
\frac{c_3^2}{36m_{\gamma\beta}L_z^2}, \frac{\mu}{20m_{\gamma\beta}c_{10}L_z^2},\,
\frac{1}{196m_{\alpha\beta}}\right\}$.

\begin{remark}\label{section_coupling}
    Decoupling strategies bring immediate benefits in BLO, as they render each search direction linear with respect to the functions $f$ and $g$, which simplifies both the algorithmic structure and the analysis.
Nevertheless, the decoupling is not complete: the step sizes remain intrinsically coupled, as reflected in \eqref{coupling_biased}.
Moreover, for ease of verification, the above step size conditions are simplified based on \eqref{coupling_biased_1}.
\end{remark}

With these step size conditions in place, we are ready to state the following theorem.

\begin{theorem}\label{theorem_biased}
Fix an iteration $K>1$ and suppose that Assumptions \ref{assump UL} and \ref{assump LL} hold.
Select stochastic estimators \( \mathcal{A} \), \( \mathcal{B} \), and \( \mathcal{C} \) as \textbf{biased}, and assume they satisfy the SE-CC and SE-VRC conditions with \( \{ A, A' \} \), \( \{ B, B' \} \), and \( \{ C, C' \} \), respectively.
 Assume that the  sequences $\{\alpha_k\}$, $\{\beta_k\}$ and $\{\gamma_k\}$ satisfy the decoupling step size conditions in \eqref{decoupling_biased1} and the coupling step size conditions in \eqref{coupling_biased}.
Then we have
{%
\setlength{\abovedisplayskip}{4pt}
\setlength{\belowdisplayskip}{4pt}
\setlength{\abovedisplayshortskip}{2pt}
\setlength{\belowdisplayshortskip}{2pt}
\begin{align*}
        \inf_{k< K} \mathbb{E}[\|\nabla H\left(x_k\right)\|^2]
\leq
\frac{ 2(L_{0}-H_*)}{\sum_{k=0}^{K-1}\alpha_k}
+\frac{2}{\sum_{k=0}^{K-1}\alpha_k}
\sum_{k=0}^{K-1}\left(A_{k+1}\Delta_k^{\mathcal{A}}
+B_{k+1}\Delta_k^{\mathcal{B}}
+C_{k+1}\Delta_k^{\mathcal{C}}\right).
\end{align*}}
\end{theorem}

\subsection{General Convergence Theorem \ref{theorem_unbiased}: unbiased estimators}\label{sec:th_unbiased}

Subsequently, we provide a general result for the scenario in which all stochastic estimators are unbiased.
Before stating the main result, we summarize the step size conditions required in Theorem~\ref{theorem_unbiased}, starting with the following decoupling conditions
\begin{subequations}\label{decoupling_unbiased}
\begin{align}
&\alpha_k\leq \frac{1}{2L^H},
\quad
\beta_k\leq
\min\left\{
\frac{1}{\mu+L^g_1},
\frac{\mu}{16c_2},\,
\frac{\mu}{16L_z^2}
\right\},
\quad
\gamma_k\leq \min\left\{
\frac{\mu}{10L_z^2},\,
\frac{1}{10\mu},\,
\frac{1}{L_1^g}
\right\},\label{decoupling_unbiased1}
\\&
\rho_k\leq \min\left\{\frac{L^H}{2\sqrt{c_1}},1\right\},
\quad
\left(\eta_k^{\mathcal{B}}B_{k+1} + \hat{\eta}^{\mathcal{B},y}_k B'_{k+1}\right)
     \leq
     \frac{1}{10\tilde{c}_2},
     \quad
     \hat{\eta}^{\prime\mathcal{B},y}_k B'_{k+1}\beta_k
\leq \frac{\mu}{16c_2},\label{decoupling_unbiased2}
\end{align}
\end{subequations}
and the coupling step size conditions are given by
\begin{subequations}\label{coupling_unbiased}
\begin{align}
&
   \alpha_k \leq m'_{\alpha\beta}\beta_k,
   \quad
   \alpha_k \leq m'_{\alpha\gamma}\gamma_k,
   \quad
   \gamma_k\leq m'_{\gamma\beta}\beta_k,
   \quad
   \rho_k\eta_k^{\mathcal{A}}\leq 4,
\quad
    \frac{\alpha_k}{\rho_k}A_{k+1}''
        \leq \frac{1}{96\left(L^H\right)^2},
   \label{coupling_unbiased_1}
\\&
 \rho_kA_{k+1}''\leq \min \left\{
        \frac{\mu\beta_k}{144c_1},\,
         \frac{\mu \gamma_k}{90c_2},\,
         \frac{1}{144c_1},\,
         \frac{1}{120c_2}
        \right\},
\quad
\left(\eta_k^{\mathcal{B}}B_{k+1} + \hat{\eta}^{\mathcal{B},x}_k B'_{k+1}\right)
        \leq  \frac{L^H}{16\tilde{c}_2},
\\[1pt]&
\left(\eta_k^{\mathcal{A}}A_{k+1} + \hat{\eta}^{\mathcal{A},\ell}_k A'_{k+1}\right)
        \leq  M_3,
       \quad
 \left(\eta_k^{\mathcal{C}}C_{k+1} + \hat{\eta}^{\mathcal{C},\ell}_k C'_{k+1}\right)
        \leq  M_3,
\\[1pt]&
\hat{\eta}^{\prime\mathcal{A},\ell}_k A'_{k+1}\beta_k
\leq M_4
,
       \quad
\hat{\eta}^{\prime\mathcal{B},x}_k B'_{k+1}\beta_k
\leq \frac{1}{24m'_{\alpha\beta}},
\quad
\hat{\eta}^{\prime\mathcal{C},\ell}_k C'_{k+1}\beta_k
\leq M_4,
\label{coupling_unbiased_4}
\end{align}
\end{subequations}
where $\ell\in\{x,y,z\}$,
$M_3=\min\left\{\frac{1}{10\tilde{c}_2},\,
       \frac{L^H}{16\tilde{c}_1}
       \right\}$,
$M_4=\min\left\{\frac{1}{24m'_{\alpha\beta}},\,\frac{\mu}{16c_2},\,\frac{\mu}{10L_z^2m'_{\gamma\beta}}\right\}$,
$m'_{\alpha\beta}=\min \left\{\frac{1}{32c_6},\,\frac{\mu}{24c_1}\right\},$
$
m'_{\alpha\gamma}=\min \left\{\frac{1}{32c_9},\,\frac{\mu}{15c_2}\right\},$
$
m'_{\gamma\beta}=\min \left\{ \frac{\mu }{8c_8 },\,\frac{5}{16}\right\}$.

With these step size conditions in place, we are ready to state the following theorem.

\begin{theorem}\label{theorem_unbiased}
Fix an iteration $K>1$ and suppose that Assumptions \ref{assump UL} and \ref{assump LL} hold.
Consider unbiased stochastic estimators $\mathcal{A}$, $\mathcal{B}$, and $\mathcal{C}$.
Assume that all three estimators satisfy the SE-VRC condition.
In addition, $\mathcal{A}$ satisfies the UEMA-CC condition with coefficient family $\{A,A',A''\}$,
and $\mathcal{B}$ and $\mathcal{C}$ satisfy the UE-CC condition with coefficient families
$\{B,B'\}$ and $\{C,C'\}$, respectively.
Assume that the  sequences $\{\alpha_k\}$, $\{\beta_k\}$, $\{\gamma_k\}$, and $\{\rho_k\}$
satisfy the decoupling conditions in \eqref{decoupling_unbiased}
and the coupling conditions in \eqref{coupling_unbiased}.
Then we have
\begin{align*}
\inf_{k< K} \mathbb{E}[\|\nabla H\left(x_k\right)\|^2]
\leq{}&
\frac{4(L_{0}-H_*)}{\sum_{k=0}^{K-1}\alpha_k}
\\
&\quad+\frac{4\sum_{k=0}^{K-1}
	\left(4(A_{k+1}+A_{k+1}''\rho_k^2)\Delta_k^{\mathcal{A}}
	+B_{k+1}\Delta_k^{\mathcal{B}}
	+C_{k+1}\Delta_k^{\mathcal{C}}\right)}{\sum_{k=0}^{K-1}\alpha_k}.
\end{align*}
\end{theorem}

\section{Illustrative Examples}\label{section_Examples}
To illustrate the effectiveness of the proposed algorithm and its convergence analysis framework, we provide several examples in this section.

We begin by presenting the proofs of convergence and sample complexity for SFFBA and MSEBA.
We then discuss how this framework can generate additional algorithms and extend to the expectation setting.
It is important to note that MA-SABA and SPABA, which were introduced in the conference version, also serve as specific instances of our framework. Due to space limitations, we will not delve further into these methods here.

\subsection{Convergence Rate and Sample Complexity of SFFBA}\label{section_sffba}
Using the convergence analysis framework, we can derive the following convergence result for SFFBA.
\begin{theorem}\label{thsffba}
Fix an iteration $K>1$ and assume that Assumptions \ref{assump UL} and \ref{assump LL} hold.
Select the batch size as $b=\sqrt{N}$,
and set the momentum  parameter \(\bar{\rho}_k\equiv\bar{\rho}= \frac{b}{2N}=\frac{1}{2\sqrt{N}}\).
Then there exist positive constants $\alpha$, $\beta$, $\gamma$, $c_{\beta}$, $c_{\alpha\beta}$, and $c_{\gamma\beta}$ such that if
\begin{align}\label{proof_sffba_decoupling_1}
&\alpha_k\equiv\alpha\leq 1/(2L^H),\quad
\gamma_k\equiv\gamma\leq \overline{c},\quad
\beta_k\equiv\beta\leq c_{\beta},\quad
\alpha_k\beta_k\leq c_{\alpha\beta} ,\quad
\gamma_k\beta_k \leq c_{\gamma\beta},
    \end{align}
and the step sizes satisfy \eqref{coupling_biased_1},
the iterates in SFFBA using the one-time full initialization \eqref{e-SFFBA-full-initialization} satisfy
\begin{align}\label{eq_sffba}
\inf_{k< K} \mathbb{E}[\|\nabla H\left(x_k\right)\|^2]
=\mathcal{O}\left({K}^{-1}\right).
\end{align}
To attain an \(\epsilon\)-stationary point, the sample complexity is
$\mathcal{O}(N+\sqrt{N}\epsilon^{-1})$.
\end{theorem}
\begin{proof}
Based on the convergence rate analysis framework we provided, and in conjunction with Theorem \ref{theorem_biased}, we will now give the values of the undetermined parameters, select the appropriate step sizes and Lyapunov function coefficients, and verify the conditions that need to be satisfied in Theorem \ref{theorem_biased}.
\begin{itemize}
    \item[\textcircled{1}] Select the Lyapunov function coefficients as $A_{k+1}=\frac{2\alpha}{\bar{\rho}},$
$B_{k+1}=\frac{2\beta}{\bar{\rho}},$
$C_{k+1}=\frac{2\gamma}{\bar{\rho}},$
\end{itemize}
$A_{k+1}'=8\alpha\bar{\rho}L'',$
$B_{k+1}'=8\beta\bar{\rho}L'',$
$C_{k+1}'=8\gamma\bar{\rho}L''.$
Let  $c_{\gamma\beta}=
c_{\alpha\beta}:=
 \sqrt{\frac{M_2}{16c_1+12L''} },$ and
 $c_{\beta}:=
    \min\left\{
    \overline{c},\,
    \sqrt{ \frac{M_1}{16+12L''}},\,
    \sqrt{\frac{1}{(16+12L'')96m_{\alpha\beta}\tilde{c}_2}}
    \right\}.$
\begin{itemize}
    \item[\textcircled{2}] Validation of the SE-VRC condition:
\end{itemize}
Lemma 2 in \cite{li2021zerosarah} verifies that the SE-VRC condition holds. More details on the properties of ZeroSARAH in the BLO setting can be found in Lemma \ref{lemma_SE-VRC_sffba}.
\begin{itemize}
    \item[\textcircled{3}] Validation of the SE-CC Condition: It holds by
\end{itemize}
{
    \setlength{\abovedisplayskip}{-3pt} 
    \setlength{\belowdisplayskip}{3pt}
\begin{align*}
A_{k+1}\theta_k^{\mathcal{A}}-A_k
&=
\frac{2\alpha}{\bar{\rho}}(1-\bar{\rho})^2-\frac{2\alpha}{\bar{\rho}}
\leq -2\alpha,
\\
A_{k+1}\tau_k^{\mathcal{A}}
+A'_{k+1}\lambda_k^{\mathcal{A}}-A'_{k}
&=
\frac{2\alpha}{\bar{\rho}}\frac{2L''\bar{\rho}^2}{b}
+8\alpha\bar{\rho}L''\left(1-\frac{b}{2N}\right)-8\alpha\bar{\rho}L''\leq 0.
\end{align*}}
For \( \mathcal{B} \) and \( \mathcal{C} \), the same reasoning applies.
\begin{itemize}[topsep=3pt, partopsep=0pt]
    \item[\textcircled{4}] Validation of the decoupling step size condition \eqref{decoupling_biased1}: It holds by \eqref{proof_sffba_decoupling_1} and
\end{itemize}
\begin{align*}
\left(B_{k+1}\eta_k^{\mathcal{B}}+B'_{k+1}\hat{\eta}_k^{\mathcal{B},y} \right)  \beta_k
    &=\left(\frac{2\beta}{\bar{\rho}}\frac{4}{b}+8\beta\bar{\rho}L''\frac{3N}{b}
     \right)\beta
     \leq \left(16+{12L''}\right)\beta^2\leq M_1.
\end{align*}
\begin{itemize}[topsep=3pt, partopsep=0pt]
        \item[\textcircled{5}]  Validation of the coupling step size condition: We have
\end{itemize}
{
    \setlength{\abovedisplayskip}{3pt} 
    \setlength{\belowdisplayskip}{3pt}
\begin{align*}
\left(A_{k+1}\eta_k^{\mathcal{A}}+A'_{k+1}\hat{\eta}_k^{\mathcal{A},\ell} \right)  \beta_k
    =
     \left(\frac{8}{b\bar{\rho}}+\frac{12\bar{\rho}N}{b}L''
     \right)\alpha\beta
    \leq M_2,
\end{align*}}
where $\ell\in\{x,y,z\}$.
 Similarly,
 we have
the coupling conditions \eqref{coupling_biased_2}-\eqref{coupling_biased_3} are satisfied.

Thus, all conditions in Theorem \ref{theorem_biased} are satisfied, and we have
$
\inf_{k< K} \mathbb{E}[\|\nabla H\left(x_k\right)\|^2]
\leq
\frac{2(L_{0}-H_*)}{\sum_{k=0}^{K-1}\alpha_k}.
$

Under the one-time full initialization, the initial
variance and staleness terms vanish.
Hence, $ L_0-H_* =\mathbb{E}\left[H(x_0)\right]+C_y\mathbb{E}\left[\|y_0-y^*_0\|^2\right]
	+C_z\mathbb{E}\left[\|z_0-z^*_0\|^2\right]-H_*$, and therefore does not depend on the $N$-dependent Lyapunov coefficients.
Since the one-time full initialization
requires $\mathcal{O}(N)$ stochastic-oracle calls and each subsequent iteration uses $\mathcal{O}(b)$
samples, the total sample complexity for attaining an
$\epsilon$-stationary point is
\(
\mathcal{O}(N+b\epsilon^{-1})
=
\mathcal{O}(N+\sqrt{N}\epsilon^{-1})
\) for $b=\sqrt{N}$.
\end{proof}

\begin{remark}
This implies that there is no gap between stochastic bilevel
and single-level optimization in the context of ZeroSARAH implementation. The lower bound established in \citep{dagreou2023lower} for BLO indicates that SFFBA achieves optimal sample complexity in the finite-sum setting when $m=\mathcal{O}(n)$ and $\epsilon=\mathcal{O}(n^{-1/2})$.
{However, memory remains a practical limitation of SFFBA: consistent with the ZeroSARAH estimator embedded in the framework, SFFBA maintains per-sample memory variables for variance reduction, resulting in a memory requirement of order  $\mathcal{O}\big(n(d_x + d_y)+m(d_x + 2d_y)\big)$.
}
\end{remark}

\subsection{Convergence Rate and Sample Complexity of MSEBA}\label{section_mseba}
In this section, we state the convergence rate and sample complexity of MSEBA.
\begin{theorem}\label{thmseba}
Fix an iteration $K>1$ and assume that Assumptions \ref{assump UL} and \ref{assump LL} hold.
Choose the batch size as \( b = \sqrt{N} \), the probability \( p = \frac{b}{N + b} \),
and set the parameter \( \bar{\rho}_k \equiv\bar{\rho}=\frac{1}{2\sqrt{N}} \).
If the step sizes satisfy \eqref{coupling_biased_1} and
$
\alpha_k\equiv\alpha\leq 1/(2L^H),$
$
\gamma_k\equiv\gamma\leq \overline{c},$
$
\beta_k\equiv\beta\leq c_{\beta},$
$
\alpha_k\beta_k\leq M_2/2,$
$
\gamma_k\beta_k \leq M_2/2,
$
then the iterates in MSEBA using the one-time full initialization \eqref{e-MSEBA-full-initialization} satisfy
\begin{align*}
\inf_{k< K} \mathbb{E}[\|\nabla H\left(x_k\right)\|^2]
=\mathcal{O}\left({K}^{-1}\right),
\end{align*}
where $c_\beta>0$ is chosen as in the proof of Theorem~16.

Moreover, to achieve an \(\epsilon\)-stationary point, the sample complexity is
$\mathcal{O}(N+\sqrt{N}\epsilon^{-1})$.
\end{theorem}
\begin{proof}
We apply Theorem \ref{theorem_biased} to prove this result.
\begin{itemize}
    \item[\textcircled{1}] Select Lyapunov function coefficients as:
\end{itemize}
$A_{k+1}=\frac{2\alpha}{p},$
$B_{k+1}=\frac{2\beta}{\bar{\rho}}$,
$C_{k+1}=\frac{2\gamma}{p},$
$A_{k+1}'=C_{k+1}'=0$ and
$B_{k+1}'=8\beta\bar{\rho}L''$.

\begin{itemize}
    \item[\textcircled{2}] Validation of the SE-VRC condition and SE-CC Condition:
\end{itemize}
According to Lemma 3 in \cite{li2021page} and Lemma \ref{lemma_SE-VRC_sffba}, we obtain that the stochastic estimators $\mathcal{A}$, $\mathcal{B}$ and $\mathcal{C}$ satisfy SE-VRC condition with the following values
\begin{align*}
& \theta_k^{\mathcal{A}}=\theta_k^{\mathcal{C}}=(1-p),\quad
  \theta_k^{\mathcal{B}}=(1-\bar{\rho})^2,
    \quad
    \eta_k^{\mathcal{A}}=\eta_k^{\mathcal{C}}=\frac{1-p}{b},\quad
    \eta_k^{\mathcal{B}}=\frac{4}{b},
    \quad
  \hat{\eta}_k^{\mathcal{B},\ell}=\frac{3N}{b},
        \\
   &  \tau_k^{\mathcal{A}}=\tau_k^{\mathcal{C}}=
   \delta_k^{\mathcal{A}}= \delta_k^{\mathcal{C}}=
   \lambda_k^{\mathcal{A}} = \lambda_k^{\mathcal{C}} =
   \hat{\eta}_k^{\mathcal{A},\ell} = \hat{\eta}_k^{\mathcal{C},\ell} =0,
        \quad
        \tau_k^{\mathcal{B}}=\frac{2L''\bar{\rho}^2}{b},\,
  \delta_k^{\mathcal{B}}= \delta_k, \,
  \lambda_k^{\mathcal{B}}=1-\frac{b}{2N},
  \\
&
\hat{\eta}_k^{\prime\mathcal{A},\ell} = \hat{\eta}_k^{\prime\mathcal{B},\ell} = \hat{\eta}_k^{\prime\mathcal{C},\ell} =\Delta_k^{\mathcal{A}}=\Delta_k^{\mathcal{B}}=\Delta_k^{\mathcal{C}}=0.
\end{align*}
The stochastic estimator \(\mathcal{B}\) satisfies the SE-CC condition, as verified in \textcircled{3} in the proof of Theorem \ref{thsffba}.
For \( \mathcal{A} \), we have
\begin{align*}
A_{k+1}\theta_k^{\mathcal{A}}-A_k
=(1-p)\frac{2\alpha}{p}-\frac{2\alpha}{p}+0
=-2\alpha,
\,
A_{k+1}\tau_k^{\mathcal{A}}
+A'_{k+1}\lambda_k^{\mathcal{A}}-A'_{k}
= 0.
\end{align*}
The same applies to \( \mathcal{C} \) .
\begin{itemize}[topsep=3pt, partopsep=0pt]
    \item[\textcircled{3}] Validation of the decoupling step size condition: see \textcircled{4} in the proof of Theorem~\ref{thsffba}.
\end{itemize}
\begin{itemize}[topsep=3pt, partopsep=0pt]
    \item[\textcircled{4}] Validation of the coupling step size condition: for conditions \eqref{coupling_biased_a}-\eqref{coupling_biased_2}, we have
\end{itemize}
{
    \setlength{\abovedisplayskip}{3pt} 
    \setlength{\belowdisplayskip}{3pt}
\begin{align*}
    \left(A_{k+1}\eta_k^{\mathcal{A}}+A'_{k+1}\hat{\eta}_k^{\mathcal{A},\ell} \right)  \beta_k
    =2\alpha\beta\frac{1-p}{pb} = 2\alpha\beta \leq M_2, \quad
    \left(C_{k+1}\eta_k^{\mathcal{C}}+C'_{k+1}\hat{\eta}_k^{\mathcal{C},\ell} \right)  \beta_k \leq M_2.
\end{align*}}

Thus, all the conditions in Theorem \ref{theorem_biased} are satisfied, and we finally obtain the convergence result,
which implies that, to achieve an \(\epsilon\)-stationary point, the sample complexity is
$\mathcal{O}\big(N+b\epsilon^{-1}+(pN+(1-p)b)\epsilon^{-1}\big)=\mathcal{O}(N+\sqrt{N}\epsilon^{-1}).$
\end{proof}

\subsection{Extension to the Expectation Setting}\label{subsection_expectation}
It is worth noting that our algorithm framework, PnPBO, and its convergence analysis framework can be directly applied to the expectation setting, where the objective functions are
\[
f(x, y) = \mathbb{E}_{\xi}\left[F\left(x, y ; \xi\right)\right],
\quad
g(x, y) = \mathbb{E}_{\zeta}\left[G\left(x, y ; \zeta\right)\right].
\]
This extends problem \eqref{pro_fini} and captures scenarios where the data comes from a stream or online setting with a substantial or potentially infinite number of samples.

In this expectation setting, we propose two illustrative algorithms as examples. First, we introduce SRMBA, which directly integrates the stochastic estimator STORM into PnPBO. Secondly, by adjusting the batch size, we propose the generalized SPABA, which is the application of PAGE within PnPBO. These algorithms are described in Section 2.4 and 2.3 of the conference version \cite{chuspaba}.

The analysis framework remains the same as described in Section \ref{section_Convergence_Analysis_Framework}, with the difference being that its variance reduction in Step 4 may depend on the addition of the following variance bounding assumption:
\begin{assumption}\label{assumporacle}
There exist positive constants \( \sigma_f \), \( \sigma_{g,1} \), and \( \sigma_{g,2} \) such that
\begin{gather*}
    \mathbb{E}[\|\nabla F(x, y; \xi) - \nabla f(x, y)\|^2] \leq \sigma_f^2,\\
    \mathbb{E}[\|\nabla G(x, y; \zeta) - \nabla g(x, y)\|^2] \leq \sigma_{g,1}^2,
    \quad
    \mathbb{E}[\|\nabla^2 G(x, y; \zeta) - \nabla^2 g(x, y)\|^2] \leq \sigma_{g,2}^2.
\end{gather*}
\end{assumption}

Under Assumptions \ref{assump UL}-\ref{assumporacle}, the sample complexity of SRMBA is  \( \tilde{\mathcal{O}}(\epsilon^{-1.5}) \), with an improved batch size of  \( {\mathcal{O}}(1) \). For the generalized SPABA, it achieves optimal sample complexity \( \mathcal{O}(\epsilon^{-1.5}) \). For more details, please refer to Theorem 3.11 and Theorem 3.9 in the conference version.
We compare SRMBA and the generalized SPABA with other methods in the expectation setting in Table \ref{table:exp}.

\begin{table}[ht]
\centering
\renewcommand{\arraystretch}{1.8}
\begin{center}
\begin{footnotesize}
\begin{tabular}{cccccc}
\hline
Algorithm & Stochastic Estimator &\makecell{Sample Complexity}  & Gap& Batch Size & Reference\\ \hline
stocBiO
&  \makecell{SGD}
& \makecell{ $\tilde{\mathcal{O}}(\epsilon^{-2})$}
& \ding{51}
& $\tilde{\mathcal{O}}(\epsilon^{-1})$
& \cite{pmlr-v139-ji21c}
\\
MRBO
&  \makecell{STORM}
& \makecell{ $\tilde{\mathcal{O}}(\epsilon^{-1.5})$}
& \ding{51}
& $\tilde{\mathcal{O}}(1)$
& \cite{yang2021provably}
\\
SUSTAIN
& STORM
& \makecell{ $\tilde{\mathcal{O}}(\epsilon^{-1.5})$}
& \ding{51}
& $\tilde{\mathcal{O}}(1)$
& \cite{khanduri2021near}
\\
VRBO
& SARAH
& \makecell{ $\tilde{\mathcal{O}}(\epsilon^{-1.5})$}
& \ding{51}
& $\tilde{\mathcal{O}}(\epsilon^{-0.5})$
&\cite{yang2021provably}\\
SRMBA
& STORM
& \makecell{ $\tilde{\mathcal{O}}(\epsilon^{-1.5})$}
& \ding{55}
& $\mathcal{O}(1)$
& Theorem 3.11$^*$\\
SPABA
& PAGE
& \makecell{ ${\mathcal{O}}(\epsilon^{-1.5})$}
& \ding{55}
& $\mathcal{O}(1)$
& Theorem 3.9$^*$
\\
\hline
\multicolumn{6}{c}{{ Lower Bound: $\Omega(\epsilon^{-1.5})$ \citep{arjevani2023lower}}} \\
\hline
\end{tabular}
\end{footnotesize}
\end{center}

\caption{Comparison of Stochastic Algorithms for BLO in the Expectation Setting.\\
\footnotesize{
$^*$ indicates that this result has been published in our conference paper \citep{chuspaba}.
 We omit comparisons with BSA \citep{ghadimi2018approximation}, TTSA \citep{hong2023two}, SOBA \citep{dagreou2022framework}, AmIGO \citep{arbel2022amortized}, and MA-SOBA \citep{chen2024optimal}, which rely on smoothness assumptions for \( f \) and \( g \) (instead of the mean-squared smoothness condition), with a best-case sample complexity of \( \mathcal{O}(\epsilon^{-2}) \).}}
 \label{table:exp}
\end{table}

\section{Numerical Experiments}

In this section, we conduct numerical experiments grounded in our theoretical results to validate the effectiveness of the proposed framework and associated algorithms.

We first compare the performance of our algorithms---SPABA, SFFBA, and MSEBA---with several benchmark methods on two widely studied tasks:
data hyper-cleaning on the corrupted \texttt{MNIST}\footnote{\url{http://yann.lecun.com/exdb/mnist/}} data set, and hyperparameter optimization for \( \ell^2 \)-penalized logistic regression on the \texttt{IJCNN1}\footnote{\url{https://www.csie.ntu.edu.tw/~cjlin/libsvmtools/datasets/binary.html}} and \texttt{covtype}\footnote{\url{https://scikit-learn.org/stable/modules/generated/sklearn.datasets.fetch_covtype.html}} benchmark data sets.
In addition, to validate the effectiveness of the PnPBO framework, we conduct ablation studies on both the moving average technique and the variable clipping technique.
All experiments were repeated five times, and the average results were reported.

\textbf{Setting.}
We strictly follow the experimental settings provided in benchmark\_bilevel in \cite{dagreou2022framework}. The original results and configurations are also publicly available at \url{https://benchopt.github.io/results/benchmark_bilevel.html}.

In all experiments, we use a batch size of 64 across all methods. The step sizes and momentum parameters for the benchmark algorithms are adopted directly from the fine-tuned values provided by \cite{dagreou2022framework}. For our methods---SPABA, SFFBA, and MSEBA---we select the best constant step sizes based on a grid search.

In this experiment, the optimal pair $(\alpha, \beta, \gamma)$ is selected via grid search. The parameter $\alpha$ is chosen from a set of 11 values logarithmically spaced between $10^{-3}$ and $10^0$. For $\beta$ and $\gamma$, we first select $\phi$ and $\kappa$ from 11 values logarithmically spaced between $10^{-5}$ and $10^0$, and then set $\beta = \frac{\alpha}{\phi}$ and $\gamma = \frac{\alpha}{\kappa}$.
In SPABA, the probability parameter $1 - p$ is selected from the set $\{0.001, 0.002, \ldots, 1\}$.
For SFFBA and MSEBA, the momentum parameter $\bar{\rho}$ is chosen from the set $\{0.01, 0.02, \ldots, 1\}$.
The clipping radius is set to \( R = 1 \).

The benchmark methods we compared include SGD-based algorithms: stocBio \citep{pmlr-v139-ji21c}, SOBA \citep{dagreou2022framework}, AmIGO \citep{arbel2022amortized}, TTSA \citep{hong2023two}; variance-reduction algorithms: SUSTAIN \citep{khanduri2021near}, MRBO \citep{yang2021provably}; algorithms proposed for finite-sum setting: SABA \citep{dagreou2022framework} and SRBA \citep{dagreou2023lower}.

\subsection{Data Hyper-Cleaning}

The first learning task we address is data hyper-cleaning on the \texttt{MNIST} data set.
The data set is partitioned into a training set $(d_i^{\text{train}}, y_i^{\text{train}})$ with 20,000 samples,
a validation set $(d_j^{\text{val}}, y_j^{\text{val}})$ with 5,000 samples, and a test set with 10,000 samples.
The input samples \( d \) are 784-dimensional, with target labels \( y \) ranging from 0 to 9. In the training set, each sample is independently corrupted with probability \( \tilde{p} \) by randomly replacing its label \( y_i \) with one from \(\{0, \ldots, 9\}\).
The validation and test sets remain uncorrupted.

The objective of data hyper-cleaning is to train a multinomial logistic regression model on the corrupted training set, assigning a weight to each training sample to ideally reduce the influence of corrupted samples by assigning them weights close to zero. This is formulated as a BLO problem, with the UL objective function
\(
f(\lambda, \theta) = \frac{1}{m} \sum_{j=1}^m \ell\left(\theta d_j^{\mathrm{val}}, y_j^{\mathrm{val}}\right),
\)
and the LL objective function
\(
g(\lambda, \theta) = \frac{1}{n} \sum_{i=1}^n \sigma\left(\lambda_i\right) \ell\left(\theta d_i^{\text{train}}, y_i^{\text{train}}\right) + C_r\|\theta\|^2,
\)
where \( \ell \) denotes the cross-entropy loss, \( \sigma \) is the sigmoid function, and we choose \( C_r = 0.2 \) after a manual search to achieve the best final test accuracy. In this experiment, the LL variable \( \theta \) is a matrix of size \( 10 \times 784 \), and the UL variable \( \lambda \) is a vector in dimension \( n_{\text{train}} = 20{,}000 \).

Figure~\ref{datacleaning_9_vr}, we report the test error---defined as the percentage of incorrect predictions on the test set---comparing SPABA, SFFBA, and MSEBA with variance-reduction-based methods, including those designed for the finite-sum setting, under a corruption probability of \( \tilde{p} = 0.9 \). To supplement this comparison, we conducted additional experiments using corruption probabilities \( \tilde{p} \in \{0.5, 0.7, 0.9\} \) with a broader set of comparison methods. The corresponding results are presented in Figures~\ref{datacleaning_5}–\ref{datacleaning_9}.

We observe that our method, SPABA, consistently outperforms other methods across all three values of \( \tilde{p} \), particularly on the more challenging settings with \( \tilde{p} = 0.7 \) and \( \tilde{p} = 0.9 \), achieving lowest test errors at a faster rate. Our second algorithm, MSEBA, which employs a mixed strategy by combining the stochastic estimators from both SPABA and SFFBA, performs between SPABA and SFFBA. It surpasses most other baseline algorithms and demonstrates strong performance on the data hyper-cleaning task.

\subsection{Hyperparameter Selection on the covtype Data Set}

In the second task, we address the problem of hyperparameter selection for determining regularization parameters in \(\ell^2\)-penalized logistic regression, evaluated on the \texttt{\texttt{covtype}} data set. This data set comprises 581{,}012 samples, each with \(\hat{p} = 54\) features, and includes \(C = 7\) classes. We use \(n = 371{,}847\) training samples, \(m = 92{,}962\) validation samples, and \(n_{\text{test}} = 116{,}203\) test samples.

Let \(\{(d_i^{\text{train}}, y_i^{\text{train}})\}_{1 \leq i \leq n}\) and \(\{(d_j^{\text{val}}, y_j^{\text{val}})\}_{1 \leq j \leq m}\) denote the training and validation sets, respectively. The bilevel objective functions are defined as
\begin{align*}
f(\theta,\lambda)
&= \frac{1}{m}\sum_{j=1}^m
   \ell\left(\theta d_j^{\mathrm{val}},y_j^{\mathrm{val}}\right),\\
g(\theta,\lambda)
&= \frac{1}{n}\sum_{i=1}^n
   \ell\left(\theta d_i^{\mathrm{train}},y_i^{\mathrm{train}}\right)
   + \sum_{c=1}^C e^{\lambda_c}
     \sum_{i=1}^{\hat{p}}\theta_{i,c}^2.
\end{align*}
where the LL variable \(\theta \in \mathbb{R}^{54 \times 7}\) represents the model parameters, and the UL variable \(\lambda \in \mathbb{R}^7\) denotes the class-wise regularization coefficients.

Figure~\ref{cov} shows comparisons with all benchmark methods on the \texttt{\texttt{covtype}} data set, focusing on the test error. All algorithms exhibit a rapid decrease in test error, with the performance of SRBA and our method, SPABA, closely aligned. Both methods achieve the lowest test error value and do so at the fastest rate.
\begin{figure}[tb]
	\centering
        \subfigure{
		\centering
		\includegraphics[width=0.7\linewidth]{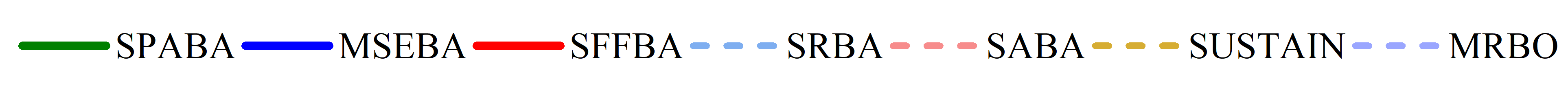}
	} \\
        \setcounter{subfigure}{0}
        \subfigure[Data hyper-cleaning on \texttt{MNIST}]{
		\centering
		\includegraphics[width=0.42\linewidth]{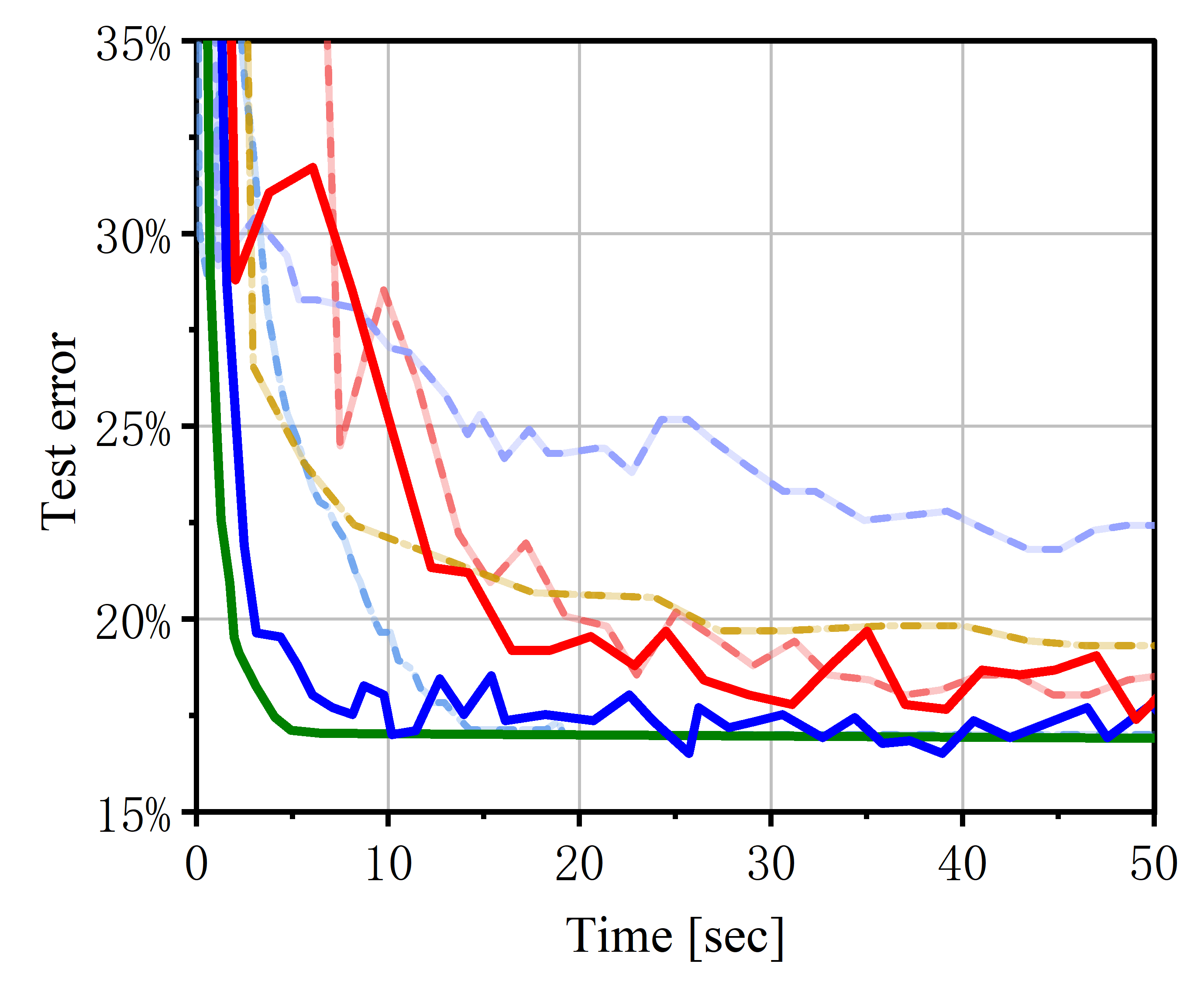}
		\label{datacleaning_9_vr}
	}
            \subfigure[Logistic regression on \texttt{IJCNN1}]{
		\centering
		\includegraphics[width=0.42\linewidth]{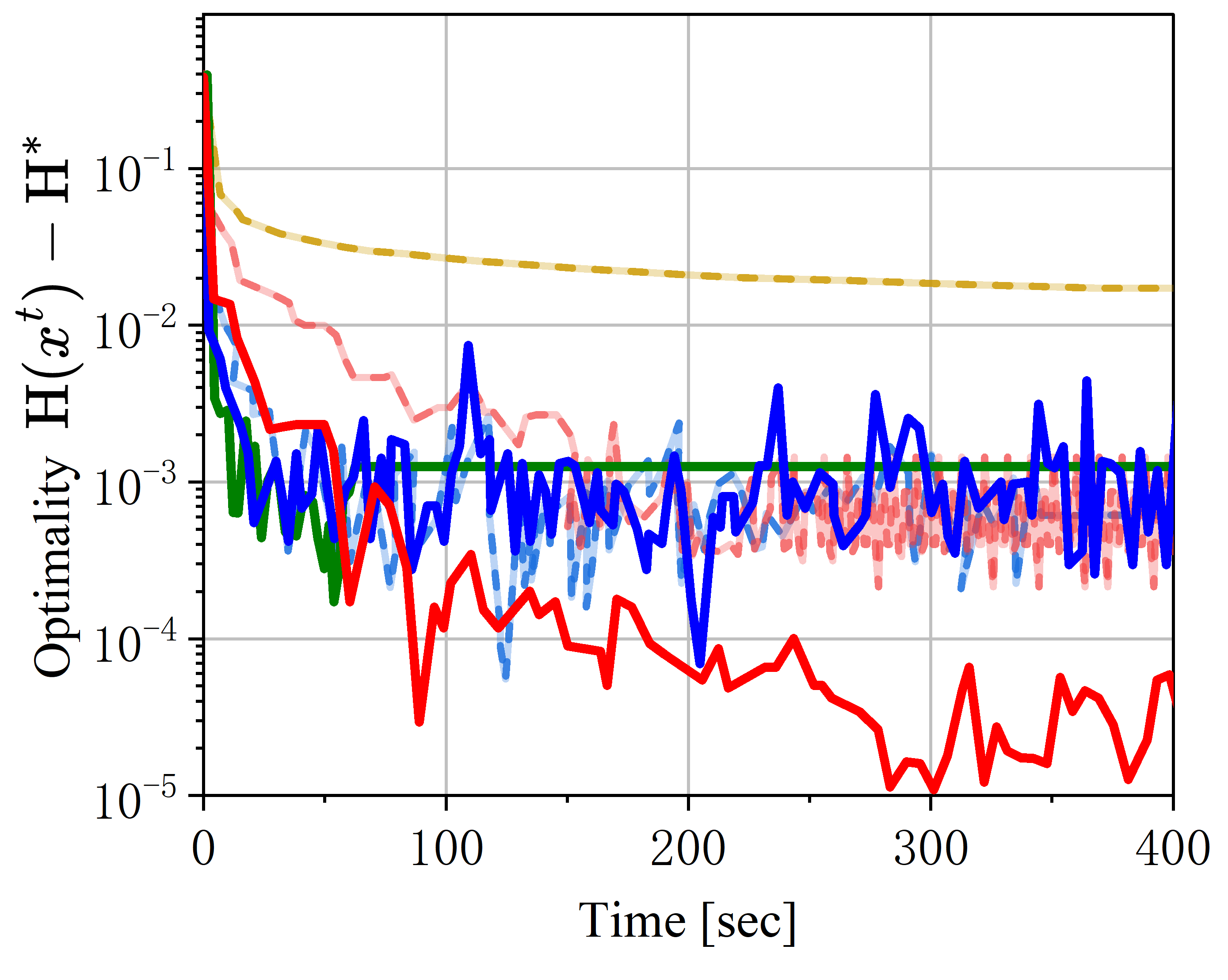}
		\label{ijcnn_vr}
	}

	\caption{Comparison of SPABA, SFFBA and MSEBA with other variance-reduction-based stochastic BLO methods. \textbf{Left:} Test error for data hyper-cleaning on \texttt{MNIST} with
$\tilde{p} = 0.9$ corruption rate, \textbf{ Right:} Suboptimality gap for hyperparameter optimization for $l^2$ penalized logistic regression on \texttt{IJCNN1} data set.}
	\label{dataclean}

\end{figure}

\subsection{Hyperparameter Selection on the IJCNN1 Data Set}
In the third task, we fit a
logistic regression model (for binary classification), and select the regularization parameters (one hyperparameter per feature) on the \texttt{IJCNN1} data set. Let $\{(d_i^{\text{train}}, y_i^{\text{train}})\}_{1 \leq i \leq n}$ and $\{(d_j^{\text{val}}, y_j^{\text{val}})\}_{1 \leq j \leq m}$ denote the training and validation sets, respectively.

In this context, the LL variable $\theta$ corresponds to the model parameters, while the UL variable $\lambda$ represents the regularization parameters. The UL and LL objective functions are defined as
\(
f(\lambda, \theta) = \frac{1}{m} \sum_{j=1}^m \varphi\left(y_j^{\mathrm{val}}\langle d_j^{\mathrm{val}}, \theta\rangle\right) \)
and
\(
g(\lambda, \theta) = \frac{1}{n} \sum_{i=1}^n \varphi\left(y_i^{\mathrm{train}}\langle d_i^{\mathrm{train}}, \theta\rangle\right) + \frac{1}{2} \sum_{k=1}^p e^{\lambda_k} \theta_k^2.
\)
In this case, the data set is partitioned into a training set $\mathcal{D}_{\text {train}}$, a validation set $\mathcal{D}_{\text {val}}$, and a test set $\mathcal{D}_{\text {test}}$, where $|\mathcal{D}_{\text{train}}| = 49{,}990$, and $|\mathcal{D}_{\text{val}}| = 91{,}701$. The LL variable $\theta \in \mathbb{R}^{22}$ is the regression coefficient, while the UL variable $\lambda \in \mathbb{R}^{22}$ is a vector of regularization parameters.
Unfortunately, the previously reported results for MRBO on the \texttt{IJCNN1} data set and SABA on the \texttt{covtype} data set are currently not reproducible due to conflicts with the latest developer version of \texttt{Benchopt}.

Figure~\ref{ijcnn_vr} compares the suboptimality gap of our methods—SPABA, SFFBA, and MSEBA—with that of other variance-reduction-based algorithms for hyperparameter optimization in \( \ell^2 \)-penalized logistic regression on the \texttt{IJCNN1} data set. Figure~\ref{ijcnn} further compares our methods with all other benchmark algorithms.
We observe that SFFBA achieves the lowest values, clearly outperforming the other algorithms.
Notably, SFFBA does not require computing full gradients, and its advantage over the other algorithms becomes more pronounced.
Specifically, we observe that SFFBA is the only algorithm that consistently reaches an objective value below \( 10^{-4} \). Both MSEBA and SPABA also achieve objective values close to \(10^{-4} \), outperforming most baseline algorithms, including SRBA.
\begin{figure}[tb]
	\centering
        \subfigure{
		\centering
		\includegraphics[width=0.7\linewidth]{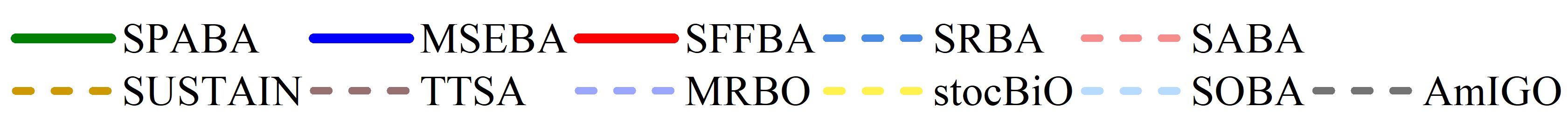}
	} \\
        \setcounter{subfigure}{0}
	\subfigure[Data hyper-cleaning ($\tilde{p}=0.5$)]{
		\centering
		\includegraphics[width=0.31\linewidth]{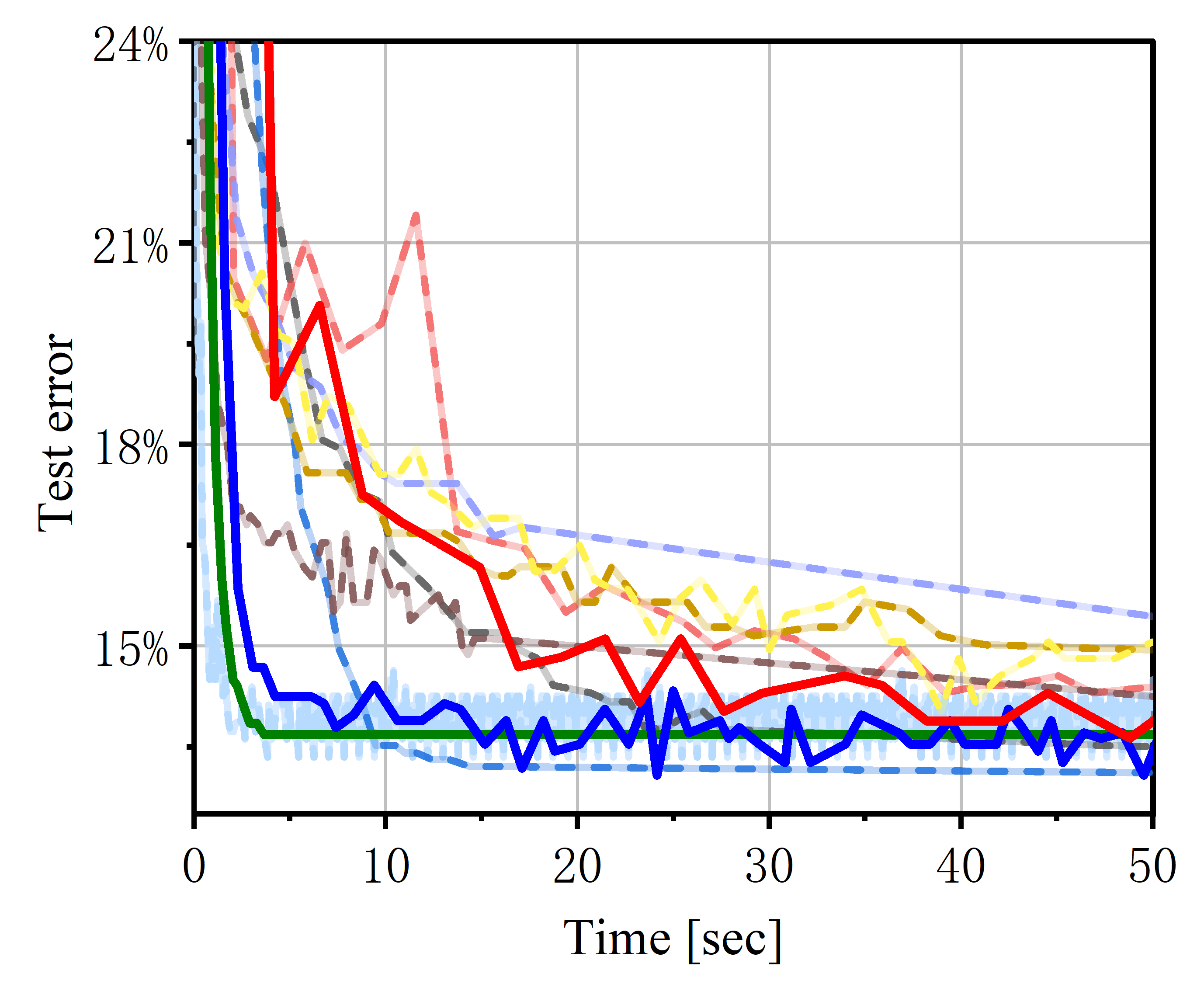}
		\label{datacleaning_5}
	}
    	\subfigure[Data hyper-cleaning ($\tilde{p}=0.7$)]{
		\centering
		\includegraphics[width=0.31\linewidth]{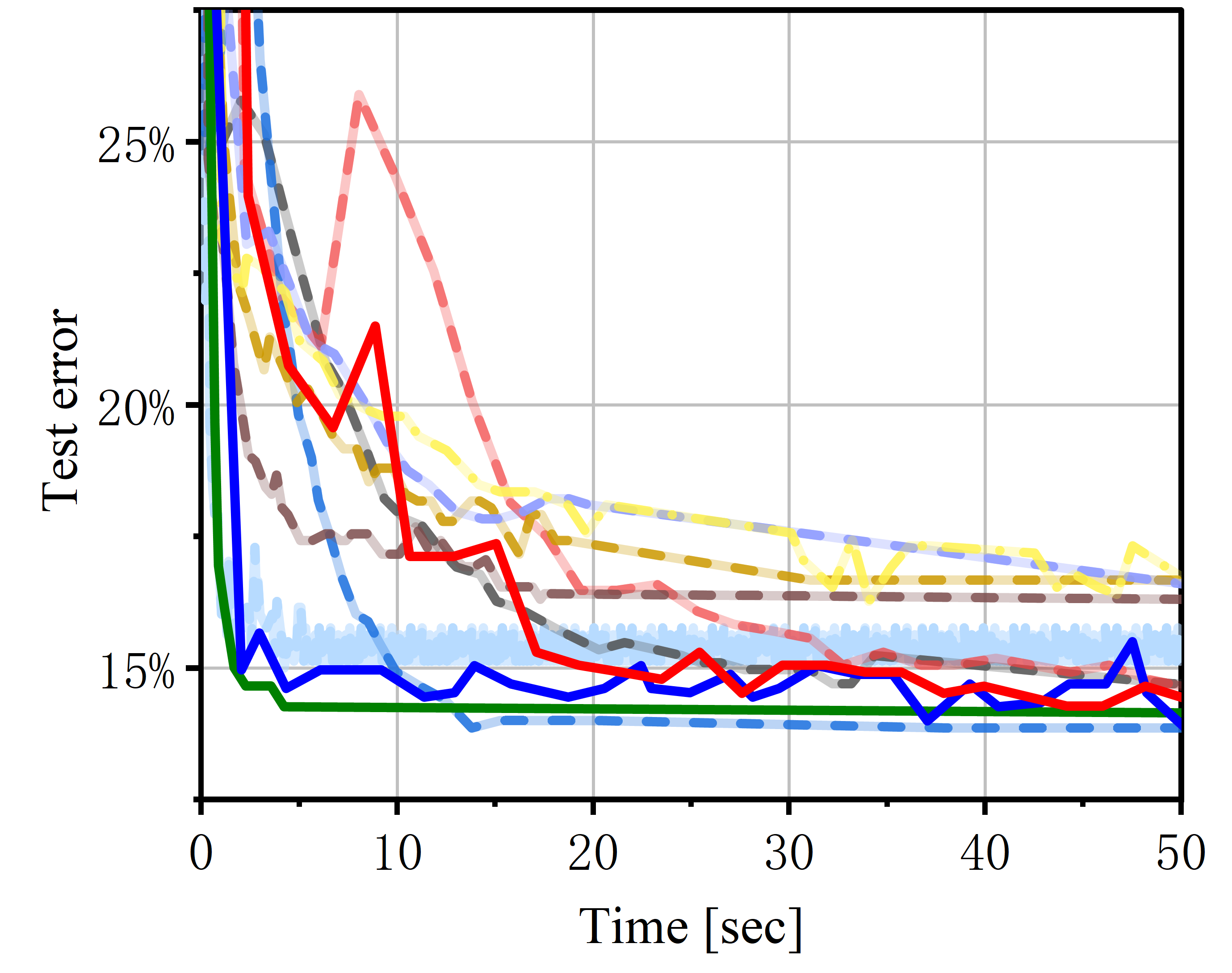}
		\label{datacleaning_7}
	}
    	\subfigure[Data hyper-cleaning ($\tilde{p}=0.9$)]{
		\centering
		\includegraphics[width=0.31\linewidth]{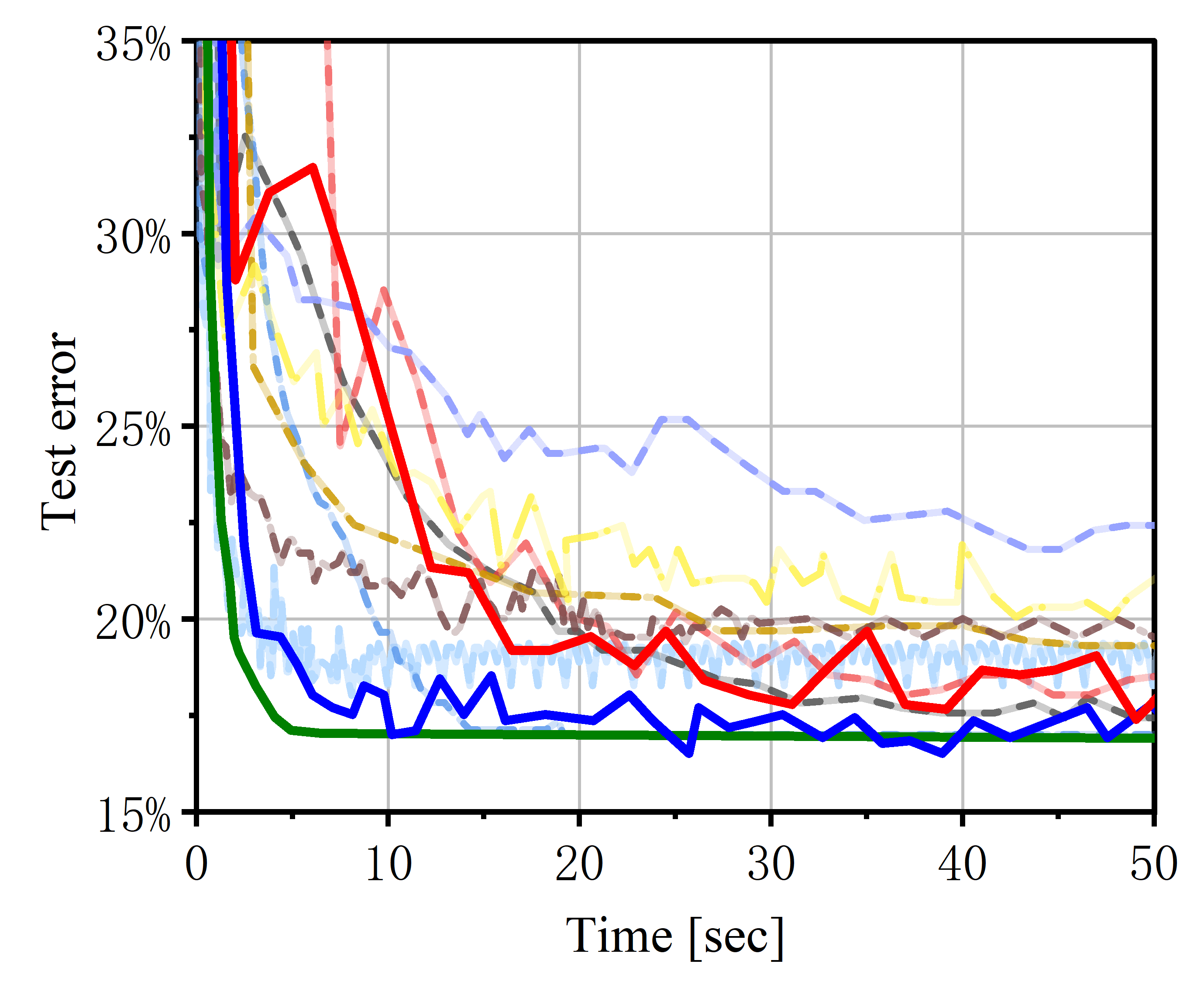}
		\label{datacleaning_9}
	}\\
		\subfigure[Logistic regression on \texttt{covtype}]{
		\centering
		\includegraphics[width=0.31\linewidth]{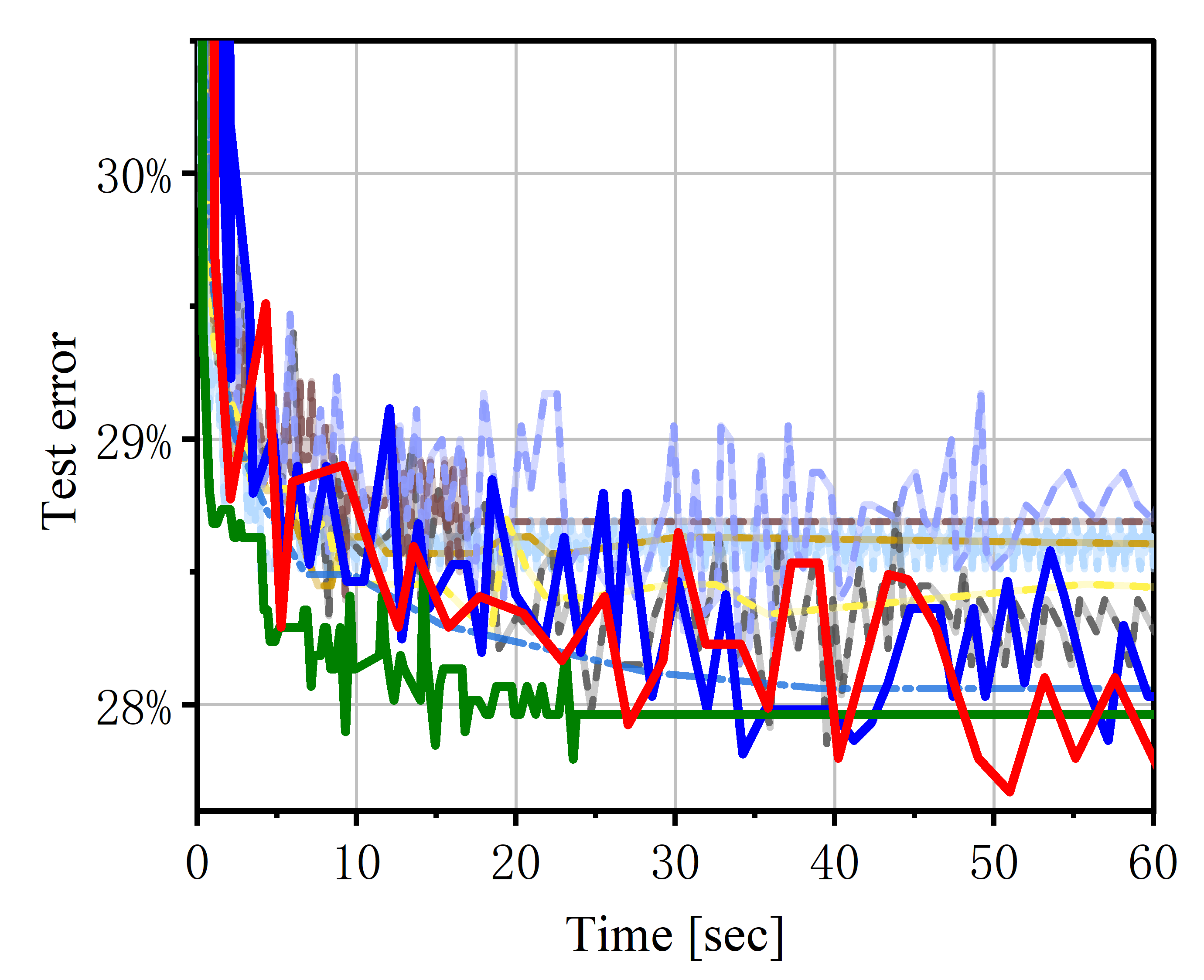}
		\label{cov}
	}
    	\subfigure[Logistic regression on \texttt{IJCNN1}]{
		\centering
		\includegraphics[width=0.31\linewidth]{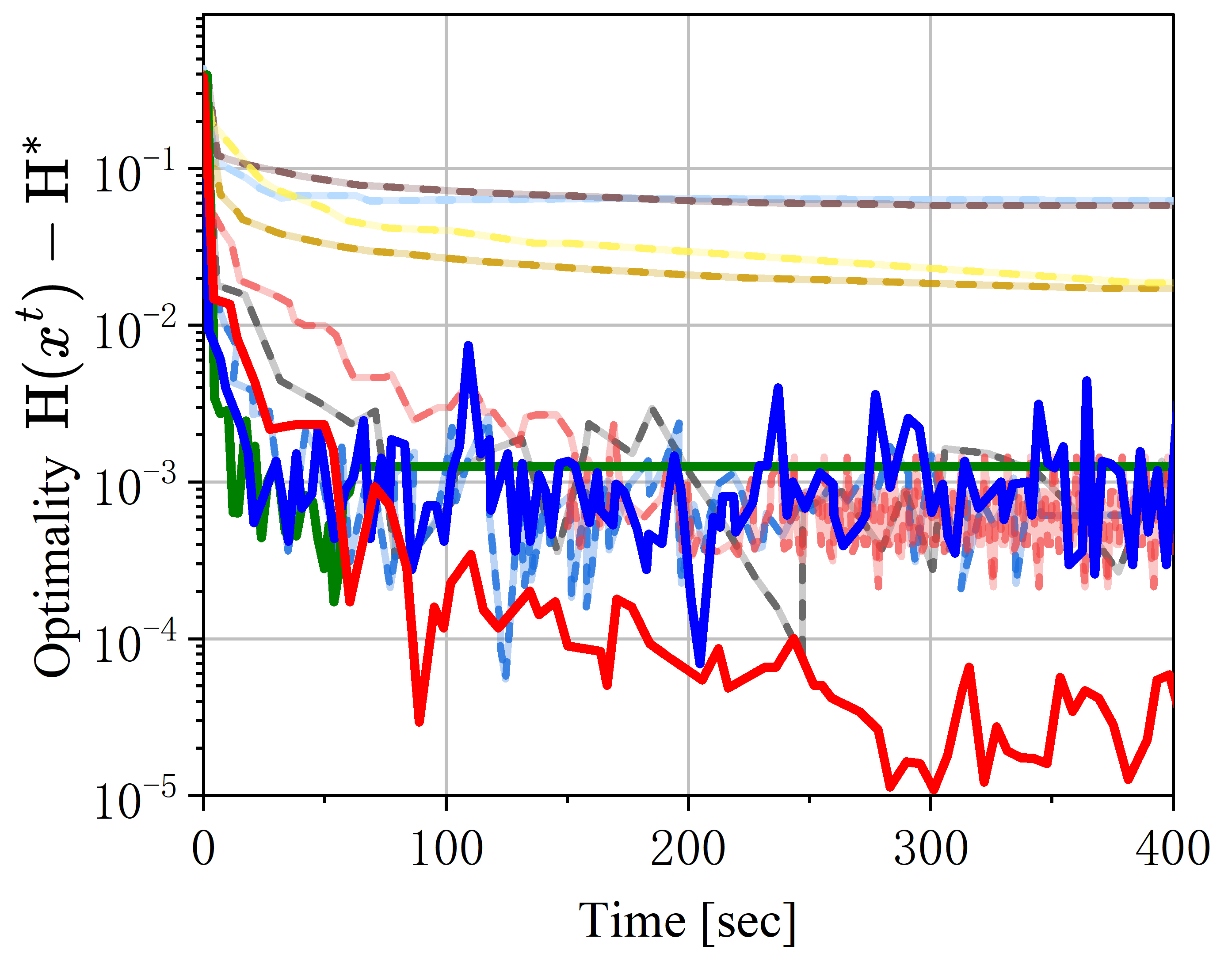}
		\label{ijcnn}
	}

	\caption{\textbf{Top:}
    Compare SPABA, SFFBA, and MSEBA with other benchmark algorithms for the data hyper-cleaning problem under different $\tilde{p}$. \textbf{Bottom:} Compare SPABA, SFFBA, and MSEBA with other benchmark algorithms for hyperparameter selection experiments on \texttt{covtype} and \texttt{IJCNN1}.}

\end{figure}

\subsection{Ablation Study}\label{section_ablation}
Our PnPBO framework introduces clipping for implicit variable updates and combines MA with the unbiased estimator \( \mathcal{A} \). In this section, we conduct ablation studies to further demonstrate the effects of these two techniques.

First, we perform an ablation study to assess the impact of the MA technique on updating UL variables when using unbiased estimators. Figure~\ref{cleaning_ma} presents the performance comparison of SOBA and SABA with their MA-enhanced counterparts, MA-SOBA and MA-SABA, in the $l^2$-penalized logistic regression task on the \texttt{IJCNN1} data set. We observe that the algorithms using the MA technique reach the minimum faster, demonstrating a potential acceleration effect.

Next, we compare the use of the clipping technique in the implicit variable update for SPABA, SFFBA and MSEBA.
Omitting the clipping step may degrade convergence speed, 
as shown in Figure~\ref{cov_clip}.

\begin{figure}[h]
	\centering
    	\subfigure[Logistic regression on \texttt{IJCNN1}]{
		\centering
		\includegraphics[width=0.4\linewidth]{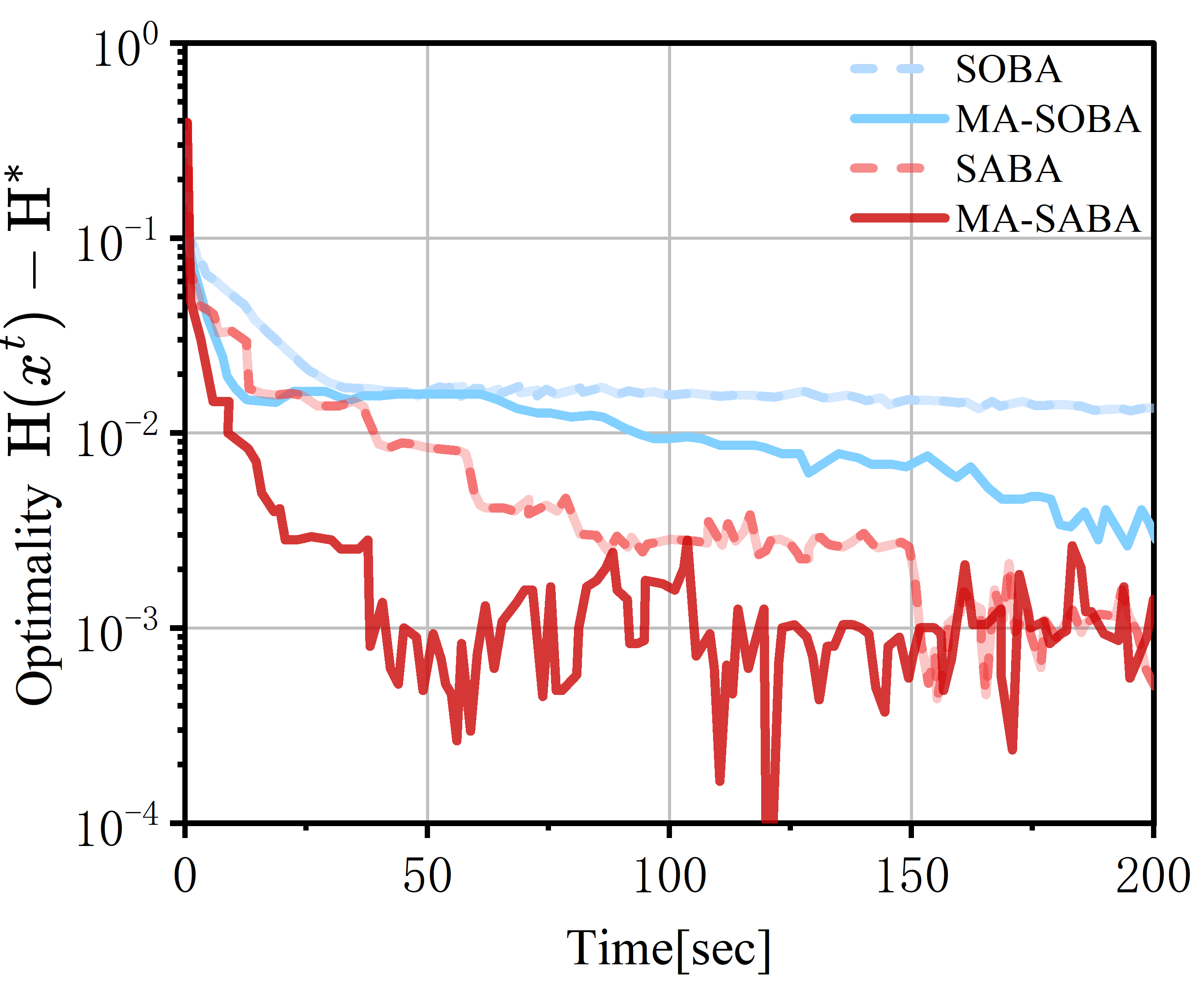}
		\label{cleaning_ma}
	}
	\subfigure[Logistic regression on \texttt{covtype}]{
		\centering
		\includegraphics[width=0.4\linewidth]{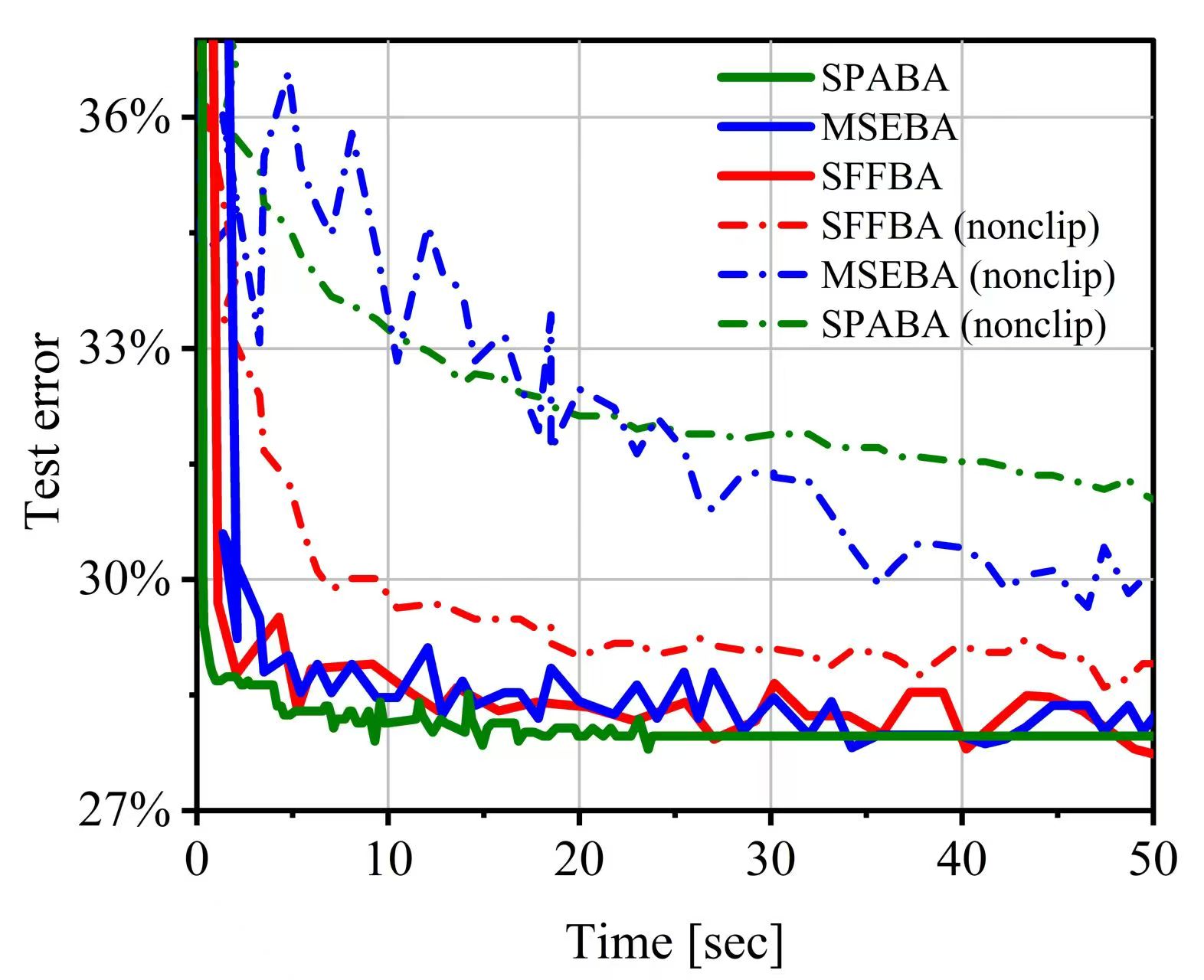}
		\label{cov_clip}
	}
	\caption{\textbf{Left:} Compare the performance of SOBA and MA-SOBA, and SABA and MA-SABA in the problem of hyperparameter selection experiment on the \texttt{IJCNN1} data set. \textbf{Right:} Compare the performance of SPABA, SFFBA, and MSEBA with their versions without the clipping technique in the problem of hyperparameter selection experiment on the \texttt{covtype} data set.
 }
	\label{ma}
\end{figure}
\section{Conclusion}

In this work, we propose the plug-and-play stochastic BLO algorithm framework, PnPBO, which enhances the framework in \citep{dagreou2022framework,arbel2022amortized} and introduces a more general convergence analysis framework. Based on this, we apply stochastic estimators such as SAGA, STORM, PAGE, ZeroSARAH, and their mixed strategies in the bilevel setting with convergence guarantees. We resolve the open question of whether the optimal complexity bounds for bilevel optimization are the same as those for single-level optimization.
It’s worth mentioning that the proposed algorithm framework still relies on second-order information. An interesting and promising direction is to use value function approaches \citep{kwon2024on,yao2025overcoming} to develop single-loop, Hessian-free stochastic bilevel algorithms.

\newpage

\appendix
\section{General Lemmas}\label{section_general_lemmas}
In this section, we present general results that will be used.
For proof details of Lemma \ref{Ly*}-\ref{dL}, see Appendix D.2. in the conference version \citep{chuspaba}.

\begin{lemma}\textbf{(Lipschitz continuity of $y^*(x)$ and $z^*(x)$, boundedness of $z^*(x)$)}\label{Ly*}

(1) Under the Assumptions \ref{assump LL}, $y^*(x)$ is $L_{y^*}$-Lipschitz continuous,
where $L_{y^*}={L_1^g}/{\mu}$.

(2) Under the Assumptions \ref{assump UL} and \ref{assump LL},
$z^*(x)$ is $L_{z^*}$-Lipschitz continuous, and for each $x$, $ \|z^*(x)\|\leq R$, where
$$L_{z^*}=\left({L^f}/{\mu}+{C^f L_2^g}/{\mu^2}\right)\left(1+{L_1^g}/{\mu}\right),\quad R= {C^f}/{\mu}.$$
\end{lemma}

\begin{lemma}\label{Hsmooth}\textbf{(smoothness of function $H$)}

Suppose Assumptions \ref{assump UL} and \ref{assump LL} hold, the function $H(x)$ is $L^H$-smooth, where $$L^H=L^f+\frac{2L^f L^g_2+\left(C^f\right)^2L^g_2}{\mu}+\frac{L^f \left(L_1^g\right)^2+2C^fL_1^gL_2^g}{\mu^2}+\frac{C^f\left(L_1^g\right)^2L_2^g}{\mu^3}.$$
\end{lemma}

\begin{lemma}\label{dL}
Suppose Assumptions \ref{assump UL} and \ref{assump LL} hold.
Then the following inequalities hold
\begin{gather}
\mathbb{E} \left[\left\|D^y_k\right\|^2\right] \leq\left(L_1^g\right)^2 \mathbb{E}\left[\| y_k-y^*_k \|^2\right],\quad
 \mathbb{E}\left[\left\|D^z_k\right\|^2\right]
\leq L_z^2\mathbb{E}\left[\left\|z_k-z^*_k\right\|^2\right]+ L_z^2\mathbb{E}\left[\left\|y_k-y^*_k\right\|^2\right],\nonumber\\
\mathbb{E}\left[\left\|D^x_k-\nabla H\left(x_k\right)\right\|^2\right]
 \leq
  3c_1 \mathbb{E}\left[\left\|y_k-y^*_k\right\|^2\right]
 +3c_2 \mathbb{E}\left[\left\|z_k-z^*_k\right\|^2\right],\label{DH}
\end{gather}
where
$L_z^2=\max
\{3\left(L_1^g\right)^2,
3R^2\left(L_2^g\right)^2+3\left(L^f\right)^2\}$.
\end{lemma}
\begin{lemma}\label{Dx{k+1}-Dx{k}}
    Under the Assumptions \ref{assump UL} and \ref{assump LL}, we have
\begin{align*}
    \left\|D^x_{k+1}-D^x_k\right\|^2
 \leq
     4c_1\left(\alpha_k^2\left\|v_k^x\right\|^2
                +\beta_k^2\left\|v_k^y\right\|^2  \right)
    +4c_2\gamma_k^2\left\|v_k^z\right\|^2.
    \end{align*}
\end{lemma}
\begin{proof}
Using Assumptions \ref{assump UL} and \ref{assump LL}, we have
    \begin{align*}
        \left\|D^x_{k+1}-D^x_k\right\|^2
    \leq&
        2\left\|\nabla_1 f(x_{k+1}, y_{k+1})- \nabla_1 f(x_{k}, y_{k})\right\|^2
        +4\left\| \nabla_{12}^2 g(x_{k}, y_{k})( z_{k} - z_{k+1} )\right\|^2
           \\
        & +4\left\| \nabla_{12}^2 g(x_{k}, y_{k})z_{k+1} - \nabla_{12}^2 g(x_{k+1}, y_{k+1}) z_{k+1}
        \right\|^2
        \\
        \leq &
         \left[2(L^f)^2 + 4 R^2 (L_2^g)^2 \right]\left(\alpha_k^2\|v_k^x\|^2 + \beta_k^2\|v_k^y\|^2  \right)
         +4(L_1^g)^2\gamma_k^2\|v_k^z\|^2.
    \end{align*}
    The second inequality also uses the boundedness of \( z_{k+1} \), which arises from the clipping.
\end{proof}

\section{{Single-level Instantiations }
}
\label{sec:single_instantiations}

In Table \ref{tab:singleloop_mapping}, we list representative single-level instantiations of the abstract interface inequalities from Section~\ref{section_single}.

\begin{table}[h]
\centering
\small
\begin{tabular}{
@{\hspace{5pt}}
m{0.42\linewidth}
@{\hspace{16pt}}
m{0.30\linewidth}
@{\hspace{16pt}}
m{0.18\linewidth}
@{\hspace{3pt}}
}
\toprule[1.2pt]
\renewcommand{\arraystretch}{1.15}
\setlength{\extrarowheight}{0pt}
\multirow{2}{*}{\textbf{Method}} &
\multirow{2}{*}{\centering\shortstack{\textbf{Active terms}\\\textbf{in \eqref{lemma_single_vr}--\eqref{lemma_single_memory}}}} &
\multirow{2}{*}{\centering\shortstack{\textbf{Coefficients}\\\textbf{in \eqref{condition_single_unbiased} or \eqref{condition_single_biased}}}} \\
& & \\
\midrule\addlinespace[4pt]
\renewcommand{\arraystretch}{2}
\setlength{\extrarowheight}{15pt}
\textbf{SGD (Unbiased)}\par

$ v_k = \nabla h_{P_k}(u_k)$;\par
step size $a_k=\mathcal{O}(K^{-{1}/{2}})$,\par
batch size $p_k=1$
&
$\Delta_k=\sigma_h^2$\par
(Assumption: there exist positive constant $\sigma_h$ such that $\mathbb{E}[\|\nabla h(u)-\nabla h_q(u)\|^2]\leq \sigma_h^2$)
&
$\text{Coef}_k=\mathcal{O}(a_k^2)$\par
$\text{Coef}'_k =0$
\\\addlinespace[3pt]
\midrule\midrule\addlinespace[5pt]
\textbf{SAGA (Unbiased)}\par

$\begin{aligned}
v_k
= &\nabla h_{P_k}(u_k) - \nabla h(\tilde{u}_{k};P_k)
\\&
+\nabla h(\tilde{u}_{k};[Q]),\end{aligned}$\par
where $[Q]=\{1,\cdots,Q\}$;\par

update
$\tilde{u}_{k+1}^{q}=
\begin{cases}
u_{k}, & q\in P_{k},\\
\tilde{u}_{k}^{q}, & q\notin P_{k};
\end{cases}$\par

step size $a_k=\mathcal{O}({Q}^{-\frac{2}{3}})$,
batch size $p_k=1$
&
$\theta_k = (4+4Q)a_k^2L_h^2\ $\par
$\eta_k = (5+4Q)L_h^2$\par
$ \tau_k = (1-1/(4Q))L^2_h$\par
$\delta_k = \frac{1}{Q}\sum\limits_{q=1}^Q\mathbb{E}[\|u_k-\tilde{u}_{k}^q\|^2]$\par
$\lambda_k = 1-1/(4Q)$\par
$\hat{\eta}_k=1$,
$
\hat{\eta}'_k=2+2Q$
&
$\text{Coef}_k = \mathcal{O}(Q^{-\frac{4}{3}})$\par
$\text{Coef}'_k = \mathcal{O}(Q^{-\frac{1}{3}})$
\\\addlinespace[5pt]
\midrule
\addlinespace[5pt]
\textbf{ZeroSARAH (Biased)}\par

$\begin{aligned}
v_k
&=(1-\hat{\rho}_k)(v_{k-1}-\nabla h_{P_k}(u_{k-1}))\\
&+\hat{\rho}_k(\nabla h(\tilde{u}_{k-1};[Q])-\nabla h(\tilde{u}_{k-1};P_{k}))\\
&+\nabla h_{P_k}(u_{k});
\end{aligned}$\par

update
$\tilde{u}_{k}^{q}=
\begin{cases}
u_{k}, & q\in P_{k},\\
\tilde{u}_{k-1}^{q}, & q\notin P_{k};
\end{cases}$
\par
step size $a_k=\mathcal{O}(1)$, parameter $\hat{\rho}_k=\mathcal{O}(1/\sqrt{Q})$,
batch size $p=\mathcal{O}(\sqrt{Q})$
&
$\theta_k = (1-\hat{\rho}_{k+1})^2$\par
$\eta_k = 2L_h^2/p$\par
$\tau_k = {2L_h^2\hat{\rho}_k^2}/{p}$\par
$\delta_k = \frac{1}{Q}\sum\limits_{q=1}^Q\mathbb{E}[\|u_k-\tilde{u}_{k}^q\|^2]$\par
$\lambda_k = 1-p_k/2Q$\par
$\hat{\eta}_k = 3Q /p$
\par
&
$\text{Coef}_k=\mathcal{O}\left(\frac{\alpha_k}{\hat{\rho}_k}\right)$\par
$\text{Coef}_k'=\mathcal{O}(\alpha_k\hat{\rho}_k)$
\\\addlinespace[5pt]
\midrule\addlinespace[5pt]

\textbf{PAGE (Biased)}\par

with probability $\hat{\rho}$, set
$v_k=\nabla h(u_k);$
otherwise, set
$v_k=v_{k-1}+\nabla h_{P_k}(u_{k})-\nabla h_{P_k}(u_{k-1})$
\par
step size $a_k=\mathcal{O}(1)$, batch size $p=\sqrt{Q}$,
probability $\hat{\rho}=p_k/(Q+p_k)$
&
$\theta_k=(1-\hat{\rho})$\par
$\eta_k=(1-\hat{\rho})L^2/p$
&
$\text{Coef}_k=\mathcal{O}(\alpha_k/\hat{\rho})$\par
$\text{Coef}_k'=0$
\\\addlinespace[5pt]
\midrule[1.2pt]
\end{tabular}
\caption{Single-level instantiations of the abstract inequalities in Section~\ref{section_single}.\\
\footnotesize{
In the \emph{Active terms} column, we report only nonzero coefficients.\\
The SAGA row is based on Lemma~C.5 in \citep{dagreou2022framework}; the coefficients reported here are obtained by specializing that lemma to the single-level finite-sum problem and applying the bound \(\mathbb{E}[\|v_k-\nabla h(u_k)\|^2]\leq L_h^2\,\delta_k\) together with simple constant relaxations.
The ZeroSARAH row follows from Lemmas~2--3 in \citep{li2021zerosarah}.
The PAGE row follows from Lemma~3 in \citep{li2021page}.
}}\label{tab:singleloop_mapping}
\end{table}

\section{The Missing Proof in Section \ref{section_Convergence_Analysis_Framework}}
\begin{proof}\textbf{(Proof of Lemma \ref{lemma_ma})}
When \( \mathcal{A} \) is an unbiased estimator, it is necessary to combine it with the MA technique, i.e.,
$v_{k+1}^x=(1-\rho_{k})v_{k}^x + \rho_{k} \hat{v}_{k+1}^x$.
Using the unbiasedness of $\hat{v}_{k+1}^x$, we obtain
\begin{align*}
 \mathbb{E}\big[\big\|v^x_{k+1}-\nabla H\left(x_{k+1}\right)\big\|^2\big]
=&  \mathbb{E}\Big[\Big\|\left(1-\rho_k\right) \big(v^x_k-\nabla H(x_k)\big)
+\nabla H(x_k)- \nabla H(x_{k+1})
+ \rho_k D^x_k \\
&
 -\rho_k \nabla H(x_k)
+ \rho_k \left( D^x_{k+1} - D^x_k\right)\Big\|^2\Big]\\
&+ \rho_k^2\mathbb{E}\left[\left\| \hat{v}_{k+1}^x- D_{k+1}^x\right\|^2\right].
\end{align*}
Based on the convexity of $\|\cdot\|^2$, we can derive the following
\begin{align*}
 &\mathbb{E}\big[\big\|v^x_{k+1}-\nabla H\left(x_{k+1}\right)\big\|^2\big]
\leq
\left(1-\rho_k\right) \mathbb{E}\left[\left\|v^x_k-\nabla H\left(x_k\right)\right\|^2\right]
+ \rho_k^2\mathbb{E}\left[\left\| \hat{v}_{k+1}^x-  D^x_{k+1}\right\|^2\right]\\
& \quad\quad +\rho_k \mathbb{E}\Big[\Big\|D^x_{k+1} - D^x_k
+D^x_k -  \nabla H(x_k)
+\frac{\nabla H\left(x_k\right)-\nabla H\left(x_{k+1}\right)}{\rho_k}\Big\|^2\Big]  \\
&\leq
\left(1-\rho_k\right) \mathbb{E}\left[\left\|v^x_k-\nabla H\left(x_k\right)\right\|^2\right]
+ \rho_k^2\mathbb{E}\left[\left\| \hat{v}_{k+1}^x-  D^x_{k+1}\right\|^2\right]
+3\rho_k \mathbb{E}[\|D^x_{k+1} - D^x_k\|^2]
\\
&\quad+3\rho_k \mathbb{E}[\|D^x_k-  \nabla H(x_k)\|^2]
+ \frac{3}{\rho_k} \mathbb{E}\left[\left\|\nabla H\left(x_k\right)-\nabla H\left(x_{k+1}\right)\right\|^2\right] .
\end{align*}
Using \eqref{DH}, Lemma \ref{Dx{k+1}-Dx{k}} and the $L^H-$smoothness of \( H \), we conclude the proof.
\end{proof}

\begin{proof}
{\textbf{(Proof of Lemma \ref{z-z*_biased})}}

Based on the definition of $z_{k+1}$ and utilizing Young's inequality twice, we obtain
	\begin{eqnarray*}
		\left\|z_{k+1}-z^*\left(x_{k+1}\right)\right\|^2
        & \leq & \left(1+\delta_k'\right)\left\|z_{k+1}-z^*\left(x_k\right)\right\|^2
		+ \left(1+\frac{1}{\delta_k'}\right)L_{z^*}^2\alpha_k^2\left\|v_k^x\right\|^2.
	\end{eqnarray*}
    Since \(\|z^*(x_k)\|\le R\), the clipping operator \(\mathrm{Clip}(\cdot;R)\) coincides with the Euclidean projection onto \(B(0,R)\), and hence it is nonexpansive.
	Using this nonexpansiveness and applying Young's inequality, we have
	\begin{align*}
		\left\|z_{k+1}-z^*\left(x_k\right)\right\|^2
        &= \left\|\text{Clip}(z_k-\gamma_k v_k^z;R)-z^*\left(x_k\right)\right\|^2
		\\
		&\leq \left\|z_k-\gamma_k v_k^z-z^*\left(x_k\right)\right\|^2
        \\
		& \leq\left(1+\frac{\mu\gamma_k}{4}\right)\left\|z_k-\gamma_k D^z_k-z^*\left(x_k\right)\right\|^2
		+\left(1+\frac{4}{\mu\gamma_k}\right)\| \gamma_k\left(v_k^z-D^z_k\right) \|^2.
	\end{align*}
    For the first term, we use the identity
\(\nabla_{22}^2 g(x_k,y^*(x_k))z^*(x_k)-\nabla_2 f(x_k,y^*(x_k))=0\)
to obtain the decomposition
\begin{align*}
&\left\|z_k-\gamma_k D_k^z-z^*\left(x_k\right)\right\|^2\notag
 = \Big\|
(I-\gamma_k\nabla_{22}^2g(x_k,y_{k}))(z_k-z^*(x_k))
\notag\\
&\quad\quad\quad+\gamma_k\big((\nabla_{22}^2g(x_k,y^*(x_k))-\nabla_{22}^2g(x_k,y_k))z^*(x_k)+\nabla_2f(x_k,y_k)-\nabla_2f(x_k,y^*(x_k))\big) \Big\|^2.\notag
\end{align*}
Applying Young's inequality yields
\begin{align*}
&\left\|z_k-\gamma_k D_k^z-z^*\left(x_k\right)\right\|^2 \\
&\leq
\left(1+\frac{\gamma_k\mu}{2}\right)
\left\|(I-\gamma_k\nabla_{22}^2g(x_k,y_k))(z_k-z^*(x_k))\right\|^2 \\
&\quad+
\left(2+\frac{4}{\gamma_k\mu}\right)\gamma_k^2
\left\|\nabla_{22}^2g(x_k,y^*(x_k))-\nabla_{22}^2g(x_k,y_k)\right\|^2
\left\|z^*(x_k)\right\|^2 \\
&\quad+
\left(2+\frac{4}{\gamma_k\mu}\right)\gamma_k^2
\left\|\nabla_2f(x_k,y_k)-\nabla_2f(x_k,y^*(x_k))\right\|^2.
\end{align*}
Using Assumption \ref{assump LL}(a) and the step size condition \(\gamma_k\le 1/L_1^g\), we have
\[\bigl\|I-\gamma_k\nabla_{22}^2 g(x_k,y_k)\bigr\|^2\le (1-\gamma_k\mu)^2.\]
Moreover, by Assumptions~\ref{assump UL}(b) and~\ref{assump LL}(b), together with \(\|z^*(x_k)\|\le R\), we obtain
\begin{align*}
&\left\|z_k-\gamma_k D_k^z-z^*\left(x_k\right)\right\|^2\notag\\
\leq &
\left(1+\frac{\gamma_k\mu}{2}\right)(1-\gamma_k\mu)^2\left\|z_k-z^*(x_k)\right\|^2
{+\left(2+\frac{4}{\gamma_k\mu}\right)\gamma_k^2c_1\|y_k-y^*(x_k)\|^2},
\end{align*}
    where $c_1=\left(L_2^gR\right)^2+(L^f)^2$.

	Rearranging the above inequalities and taking the total expectation yields
	\begin{align*}
		\mathbb{E}\left[\left\|z_{k+1}-z^*\left(x_{k+1}\right)\right\|^2\right]
		\leq&\left(1+\frac{\mu\gamma_k}{4}\right)\left(1+\frac{\gamma_k\mu}{2}\right)^2(1-\gamma_k\mu)^2\mathbb{E}\left[\left\|z_k-z^*\left(x_k\right)\right\|^2\right]\\
		&+
		\left(1+\frac{\mu\gamma_k}{4}\right)\left(1+\frac{\gamma_k\mu}{2}\right)\left(2+\frac{4}{\gamma_k\mu}\right)\gamma_k^2c_1\mathbb{E}\left[\left \|y_k-y^*(x_k)\right\|^2\right]\\
		&+\left(1+\frac{\mu\gamma_k}{2}\right)\left(1+\frac{4}{\mu\gamma_k}\right)\gamma_k^2\mathbb{E}\left[\| v_k^z-D_k^z\|^2\right]\\
		&+ \left(1+\frac{2}{\mu\gamma_k }\right)L_{z^*}^2\alpha_k^2\mathbb{E}\left[\left\|v_k^x\right\|^2\right].
	\end{align*}
	Using  the step size condition $\gamma_k\leq\frac{1}{L_1^g+\mu},$ we have
	\begin{align*}
		\mathbb{E}\left[\left\|z_{k+1}-z^*\left(x_{k+1}\right)\right\|^2\right]
		\leq&
		\left(1-\frac{\gamma_k\mu}{4}\right)\mathbb{E}\left[\left\|z_k-z^*\left(x_k\right)\right\|^2\right]
		+
		\frac{18c_1\gamma_k}{\mu}\mathbb{E}\left[\left \|y_k-y^*(x_k)\right\|^2\right]
        \\
		&+\frac{12\gamma_k}{\mu}\mathbb{E}\left[\| v_k^z-D_k^z\|^2\right]
        + \frac{4L_{z^*}^2\alpha_k^2}{\mu\gamma_k }\mathbb{E}\left[\left\|v_k^x\right\|^2\right].
	\end{align*}
\end{proof}

\begin{proof}
\textbf{(Proofs of the Remaining Lemmas and Corollary):}

They have already been proved in our conference version. Specifically, Lemma \ref{H} can be found in Lemma D.7 of the conference version, Corollary \ref{H_biased} in Corollary F.1, Lemma \ref{y-y*_biased} in Lemma F.3, Lemmas \ref{y-y*_unbiased} and \ref{z-z*_unbiased} in Lemma E.2.
\end{proof}

\section{Variance Results of Stochastic Estimators in the Bilevel Setting}\label{appendix_variance}

{For the stochastic estimator \( \mathcal{B} \),} we have
\begin{equation}\label{lemma_vr_b}
\begin{aligned}
\mathbb{E}\left[\left\|D^y_{k+1}- v_{k+1}^y\right\|^2\right]
    \leq &
\theta_k^{\mathcal{B}} \mathbb{E}\left[\left\|D^y_k - v_k^y\right\|^2\right]
+ \eta_k^{\mathcal{B}}c_2\left(\alpha_k^2 \mathbb{E}\left[\left\|v_k^x\right\|^2\right] + \beta_k^2 \mathbb{E}\left[\left\|v_k^y\right\|^2\right] \right)
\\&+ \tau_k^{\mathcal{B}}\delta_k^{\mathcal{B}}
+\Delta_k^{\mathcal{B}},
\end{aligned}
\end{equation}
and
\begin{equation}\label{lemma_memory_b}
\begin{aligned}
\delta_{k+1}^{\mathcal{B}}\leq &
    \lambda_k^{\mathcal{B}} \delta_k^{\mathcal{B}}
    + \hat{\eta}_k^{\mathcal{B},x}\alpha_k^2
    \mathbb{E}[\|v_k^x\|^2]
    + \hat{\eta}_k^{\mathcal{B},y}\beta_k^2
    \mathbb{E}[\|v_k^y\|^2]
 +\hat{\eta}_k'^{\mathcal{B},x}\alpha_k^2
   \mathbb{E}[\|D_k^x\|^2]
   +\hat{\eta}_k'^{\mathcal{B},y}\beta_k^2\mathbb{E}[\|D_k^y\|^2].
\end{aligned}
\end{equation}
{For the stochastic estimator \( \mathcal{C} \),} we have
\begin{equation}\label{lemma_vr_c}
\begin{aligned}
\mathbb{E}\left[\left\|D^z_{k+1} - v_{k+1}^z\right\|^2\right]
    \leq&
\theta_k^{\mathcal{C}} \mathbb{E}\left[\left\|D^z_k - v_k^z\right\|^2\right]
+
    \eta_k^{\mathcal{C}}c_1
    \alpha_k^2\mathbb{E}[\|v_k^x\|^2]
    +\eta_k^{\mathcal{C}}c_1 \beta_k^2\mathbb{E}[\|v_k^y\|^2]
    \\&
    + \eta_k^{\mathcal{C}} c_2\gamma_k^2\mathbb{E}[\|v_k^z\|^2]
  +\tau_k^{\mathcal{C}}\delta_k^{\mathcal{C}}
     + \Delta_k^{\mathcal{C}},
\end{aligned}
\end{equation}
and
\begin{equation}\label{lemma_memory_c}
\begin{aligned}
    \delta_{k+1}^{\mathcal{C}}\leq &
    \lambda_k^{\mathcal{C}} \delta_k^{\mathcal{C}}
    + \hat{\eta}_k^{\mathcal{C},x}\alpha_k^2
    \mathbb{E}[\|v_k^x\|^2]
    + \hat{\eta}_k^{\mathcal{C},y}\beta_k^2
    \mathbb{E}[\|v_k^y\|^2]
    + \hat{\eta}_k^{\mathcal{C},z}\gamma_k^2
    \mathbb{E}[\|v_k^z\|^2]
    \\
    &+\hat{\eta}_k'^{\mathcal{C},x}\alpha_k^2
   \mathbb{E}[\|D_k^x\|^2]
   +\hat{\eta}_k'^{\mathcal{C},y}\beta_k^2\mathbb{E}[\|D_k^y\|^2]
   +\hat{\eta}_k'^{\mathcal{C},z}\gamma_k^2\mathbb{E}[\|D_k^z\|^2].
\end{aligned}
\end{equation}

\begin{proof}\textbf{(Proof of Inequalities
\eqref{lemma_vr_a}, \eqref{lemma_memory_a}, and \eqref{lemma_vr_b}-\eqref{lemma_memory_c})}

We take $\mathcal{A}$ as an illustrative example.
Since $\mathcal{A}$ satisfies the SE-VRC condition,
we have
\begin{equation}\label{eq_appendix_blo_vr}
\begin{aligned}
    &\mathbb{E}\left[\left\|D^x_{k+1}- \hat{v}_{k+1}^x\right\|^2\right]
   \\ \leq &
    \theta_k^{\mathcal{A}} \mathbb{E}\left[\left\|D^x_{k} - \hat{v}_k^x\right\|^2\right]
    + \eta_k'^{\mathcal{A}} \mathbb{E}\left[\left\|\nabla_1 F_i(x_{k+1},y_{k+1}) -
    \nabla_1 F_i(x_k,y_k)\right\|^2\right]
     \\&
    + \eta_k'^{\mathcal{A}} \mathbb{E}\left[\left\|\nabla_{12}^2 G_j(x_{k+1},y_{k+1})z_{k+1}
    - \nabla_{12}^2 G_j(x_k,y_k)z_k\right\|^2\right]
    \\
    &+\tau_k'^{\mathcal{A}}\frac{1}{n}\sum_{i=1}^n\mathbb{E}[\|\nabla_1 F_i(x_k,y_k)-\nabla_1 F_i(w_{k,i}^x,w_{k,i}^y)\|^2]
     \\
    &+\tau_k'^{\mathcal{A}}\frac{1}{m}\sum_{j=1}^m\mathbb{E}[\|\nabla_{12}^2 G_j(x_k,y_k)z_k-\nabla_{12}^2 G_j(\tilde{w}_{k,j}^x,\tilde{w}_{k,j}^y)\tilde{w}_{k,j}^z\|^2]
    + \Delta_k^{\mathcal{A}}
    ,
\end{aligned}
\end{equation}
where, conditional on $I_{k+1}$ and $J_{k+1}$, the indices $i$ and $j$ are sampled independently and uniformly
from $I_{k+1}$ and $J_{k+1}$, respectively.
Here $w_{k,i}=(w_{k,i}^x,w_{k,i}^y)$ for $i\in[n]$ and
$\tilde w_{k,j}=(\tilde w_{k,j}^x,\tilde w_{k,j}^y,\tilde w_{k,j}^z)$ for $j\in[m]$
are the memory variables maintained by $\mathcal{A}$ (when memory is required),
corresponding to the oracle calls for $f$ and $g$, respectively.
We define
\begin{align*}
\delta_k^{f,x}=\frac{1}{n}\sum\limits_{i=1}^n\mathbb{E}[\|x_k-w_{k,i}^x\|^2],\quad
\delta_k^{g,x}=\frac{1}{m}\sum\limits_{j=1}^m\mathbb{E}[\|x_k-\tilde{w}_{k,j}^x\|^2]\end{align*}
and analogously define $\delta_k^{f,y}$, $\delta_k^{g,y}$, $\delta_k^{g,z}$.

For the second and third terms on the right-hand side of \eqref{eq_appendix_blo_vr}, we can bound them as follows:

\begin{align}
   &\eta_k'^{\mathcal{A}}\mathbb{E}\left[\left\| {\nabla_1 F_i(x_{k+1}, y_{k+1}) - \nabla_1 F_i(x_{k}, y_{k})}\right\|^2\right]
   \leq
   \eta_k'^{\mathcal{A}}(L^f)^2\left(\alpha_k^2\mathbb{E}[\|v_k^x\|^2]
   + \beta_k^2\mathbb{E}[\|v_k^y\|^2]\right),\label{eq_fi_xk}
   \end{align}
   \begin{equation}\label{eq_gj_xk}
   \begin{aligned}
   &\eta_k'^{\mathcal{A}} \mathbb{E}\left[\left\|\nabla_{12}^2 G_j(x_{k+1},y_{k+1})z_{k+1}
    - \nabla_{12}^2 G_j(x_k,y_k)z_k\right\|^2\right] \\
    \leq
    &2\eta_k'^{\mathcal{A}}\mathbb{E}\left[\left\| {\nabla_{12}^2 G_j(x_{k}, y_{k})(z_{k+1} - z_k)}\right\|^2\right]\\
    &+
    2\eta_k'^{\mathcal{A}} \mathbb{E}[\|{(\nabla_{12}^2 G_j(x_{k+1}, y_{k+1}) - \nabla_{12}^2 G_j(x_{k}, y_{k}))}\|^2 \|z_{k+1}\|^2] \\
     \leq &{2R^2 (L_2^g)^2 \eta_k'^{\mathcal{A}}\alpha_k^2}(\mathbb{E}[\|v_k^x\|^2]
    + \beta_k^2 \mathbb{E}[\|v_k^y\|^2])
    + {2 (L_1^g)^2 \eta_k'^{\mathcal{A}}\gamma_k^2}\mathbb{E}[\|v_k^z\|^2].
\end{aligned}
\end{equation}

For the fourth term on the right-hand side of \eqref{eq_appendix_blo_vr},
by Assumption~\ref{assump UL}, we have
\begin{align}\label{eq_fi_memory}
    &\tau_k'^{\mathcal{A}}\frac{1}{n}\sum_{i=1}^n\mathbb{E}[\|\nabla F_i(x_k,y_k)-\nabla F_i(w_{k,i}^x,w_{k,i}^y)\|^2]
    \leq
    \tau_k'^{\mathcal{A}}(L^f)^2\delta_k^{f,x}+\tau_k'^{\mathcal{A}}(L^f)^2\delta_k^{f,y}.
\end{align}

For the fifth term on the right-hand side of \eqref{eq_appendix_blo_vr},
invoking Assumption~\ref{assump LL} and Lemma~\ref{Ly*},
and adding and subtracting $\nabla_{12}^2 G_j(\tilde w_{k,j}^x,\tilde w_{k,j}^y) z_k$ yield
\begin{equation}\label{eq_gj_meomry}
\begin{aligned}
    &\tau_k'^{\mathcal{A}}\frac{1}{m}\sum_{j=1}^m\mathbb{E}[\|\nabla_{12}^2 G_j(x_k,y_k)z_k-\nabla_{12}^2 G_j(\tilde{w}_{k,j}^x,\tilde{w}_{k,j}^y)\tilde{w}_{k,j}^z\|^2]
    \\\leq&
    2\tau_k'^{\mathcal{A}}(L_2^g)^2R^2(\delta_k^{g,x}+\delta_k^{g,y})
    +2\tau_k'^{\mathcal{A}}(L_1^g)^2\delta_k^{g,z}.
\end{aligned}\end{equation}
Substituting the above inequalities into \eqref{eq_appendix_blo_vr} yields
\begin{align*}
    \mathbb{E}\left[\left\|D^x_{k+1}- \hat{v}_{k+1}^x\right\|^2\right]
    \leq &
    \theta_k^{\mathcal{A}} \mathbb{E}\left[\left\|D^x_{k} - \hat{v}_k^x\right\|^2\right]
    +
    {2 c_1 \eta_k'^{\mathcal{A}}\alpha_k^2}\mathbb{E}[\|v_k^x\|^2]
    + {2 c_1 \eta_k'^{\mathcal{A}}\beta_k^2}\mathbb{E}[\|v_k^y\|^2]
    \\&
    + {2 c_2 \eta_k'^{\mathcal{A}}\gamma_k^2}\mathbb{E}[\|v_k^z\|^2]
    +\tau_k'^{\mathcal{A}}(L^f)^2(\delta_k^{f,x}+\delta_k^{f,y})
     \\&+2\tau_k'^{\mathcal{A}}(L_2^g)^2R^2(\delta_k^{g,x}+\delta_k^{g,y})
    +2\tau_k'^{\mathcal{A}}(L_1^g)^2\delta_k^{g,z}
     + \Delta_k^{\mathcal{A}}
     \\
     \leq &
    \theta_k^{\mathcal{A}} \mathbb{E}[\|D^x_{k} - \hat{v}_k^x\|^2] +\tau_k^{\mathcal{A}}\delta_k^{\mathcal{A}}
     + \Delta_k^{\mathcal{A}}
    \\&
    +
    \eta_k^{\mathcal{A}}
    \left(c_1\alpha_k^2\mathbb{E}[\|v_k^x\|^2]
    + c_1\beta_k^2\mathbb{E}[\|v_k^y\|^2]
    +  c_2\gamma_k^2\mathbb{E}[\|v_k^z\|^2] \right)
    ,
\end{align*}
where we set $\eta_k^{\mathcal{A}}=2\eta_k'^{\mathcal{A}}$,
$\tau_k^{\mathcal{A}}=\tau_k'^{\mathcal{A}}L''$, $L''=\max\{(L^f)^2,\,2(L_2^gR)^2,\,2(L_1^g)^2\}$, and $\delta_k^{\mathcal{A}}:=\delta_k^{f,x}+\delta_k^{g,x}+\delta_k^{f,y}+\delta_k^{g,y}+\delta_k^{g,z}$.

Based on the memory-update rule of $\mathcal{A}$, which satisfies the SE-VRC condition, the mean-squared staleness admits the recursions
\begin{subequations}\label{eq:delta_recursions}
\begin{align}
    &\delta_{k+1}^{f,x}\leq
    \lambda_k^{\mathcal{A},n} \delta_k^{f,x}
    + \hat{\eta}_k^{\mathcal{A},n,x}\alpha_k^2
    \mathbb{E}[\|v_k^x\|^2]
   + \hat{\eta}_k'^{\mathcal{A},n,x}\alpha_k^2
    \mathbb{E}[\|D_k^x\|^2],\\&
    \delta_{k+1}^{g,x}\leq
    \lambda_k^{\mathcal{A},m} \delta_k^{g,x}
    + \hat{\eta}_k^{\mathcal{A},m,x}
    \alpha_k^2
    \mathbb{E}[\|v_k^x\|^2]
   + \hat{\eta}_k'^{\mathcal{A},m,x}\alpha_k^2
    \mathbb{E}[\|D_k^x\|^2],\\&
    \delta_{k+1}^{f,y}\leq
    \lambda_k^{\mathcal{A},n} \delta_k^{f,y}
    + \hat{\eta}_k^{\mathcal{A},n,y}\beta_k^2
    \mathbb{E}[\|v_k^y\|^2]
   + \hat{\eta}_k'^{\mathcal{A},n,y}\beta_k^2
    \mathbb{E}[\|D_k^y\|^2],\\&
    \delta_{k+1}^{g,y}\leq
    \lambda_k^{\mathcal{A},m} \delta_k^{g,y}
    + \hat{\eta}_k^{\mathcal{A},m,y}\beta_k^2
    \mathbb{E}[\|v_k^y\|^2]
   + \hat{\eta}_k'^{\mathcal{A},m,y}\beta_k^2
    \mathbb{E}[\|D_k^y\|^2],\\&
    \delta_{k+1}^{g,z}\leq
    \lambda_k^{\mathcal{A},m} \delta_k^{g,z}
    + \hat{\eta}_k^{\mathcal{A},m,z}\gamma_k^2
    \mathbb{E}[\|v_k^z\|^2]
   + \hat{\eta}_k'^{\mathcal{A},m,z}\gamma_k^2
    \mathbb{E}[\|D_k^z\|^2].
\end{align}
\end{subequations}
Note that the staleness terms are induced by the memory tables maintained by $\mathcal{A}$.
The coefficients in the recursion may depend on the data set sizes $n$ and $m$ through the sampling and refresh mechanisms.
In addition, $\hat{\eta}_k^{\mathcal{A},\cdot,\cdot}$ and $\hat{\eta}_k^{\prime\mathcal{A},\cdot,\cdot}$ may vary with the properties of the update directions $(v_k^x,v_k^y,v_k^z)$ (e.g., biased vs.\ unbiased), generated by $\mathcal{A}$, $\mathcal{B}$, and $\mathcal{C}$, respectively.
We therefore distinguish these coefficients using the superscripts in \eqref{eq:delta_recursions}.\\
Define
\(
\hat{\eta}_k^{\mathcal{A},x}
:=\hat{\eta}_k^{\mathcal{A},n,x}+\hat{\eta}_k^{\mathcal{A},m,x},\quad
\hat{\eta}_k^{\prime\mathcal{A},x}
:=\hat{\eta}_k^{\prime\mathcal{A},n,x}+\hat{\eta}_k^{\prime\mathcal{A},m,x},
\)
with $\hat{\eta}_k^{\mathcal{A},y}$, $\hat{\eta}_k^{\mathcal{A},z}$, $\hat{\eta}_k^{\prime\mathcal{A},y}$, and $\hat{\eta}_k^{\prime\mathcal{A},z}$ defined analogously.
Summing \eqref{eq:delta_recursions} yields
\begin{align*}
    \delta_{k+1}^{\mathcal{A}}\leq &
    \lambda_k^{\mathcal{A}} \delta_k^{\mathcal{A}}
    + \hat{\eta}_k^{\mathcal{A},x}\alpha_k^2
    \mathbb{E}[\|v_k^x\|^2]
    + \hat{\eta}_k^{\mathcal{A},y}\beta_k^2
    \mathbb{E}[\|v_k^y\|^2]
    + \hat{\eta}_k^{\mathcal{A},z}\gamma_k^2
    \mathbb{E}[\|v_k^z\|^2]
    \\
    &+\hat{\eta}_k'^{\mathcal{A},x}\alpha_k^2
   \mathbb{E}[\|D_k^x\|^2]
   +\hat{\eta}_k'^{\mathcal{A},y}\beta_k^2\mathbb{E}[\|D_k^y\|^2]
   +\hat{\eta}_k'^{\mathcal{A},z}\gamma_k^2
   \mathbb{E}[\|D_k^z\|^2].
\end{align*}
The proofs of \eqref{lemma_vr_b}--\eqref{lemma_memory_c} follow analogously and are omitted.
\end{proof}

\clearpage
\section{Verification of SFFBA Satisfying the SE-VRC condition.}\label{section_proof_sffba}
\begin{lemma}\label{lemma_SE-VRC_sffba}
    Under the conditions stated in Assumptions \ref{assump UL} and \ref{assump LL}, SFFBA satisfies the SE-VRC condition and we have
    {
    \setlength{\abovedisplayskip}{3pt} 
    \setlength{\belowdisplayskip}{3pt}
     \begin{align*}
       & \theta_k^{\mathcal{A}}=\theta_k^{\mathcal{B}}=\theta_k^{\mathcal{C}}=(1-\bar{\rho}_{k+1})^2,\quad
        \eta_k^{\mathcal{A}}=\eta_k^{\mathcal{B}}=\eta_k^{\mathcal{C}}=\frac{4}{b},\quad
    \tau_k^{\mathcal{A}}=\tau_k^{\mathcal{B}}=\tau_k^{\mathcal{C}}=\frac{2\bar{\rho}_{k+1}^2L''}{b},\\
       &
    \delta_k^{\mathcal{A}}=\delta_k^{\mathcal{C}}=
    \delta_k^{\mathcal{B}}=\delta_k,\quad
    \lambda_k^{\mathcal{A}}=\lambda_k^{\mathcal{B}}=\lambda_k^{\mathcal{C}}=1-\frac{b}{2N},
\quad
\Delta_k^{\mathcal{A}}=\Delta_k^{\mathcal{B}}=\Delta_k^{\mathcal{C}}=0,
        \\
       &\hat{\eta}_k^{\mathcal{A},\ell}=\hat{\eta}_k^{\mathcal{B},\ell}=\hat{\eta}_k^{\mathcal{C},\ell}=\frac{3N}{b},
       \quad
       \hat{\eta}_k^{\prime\mathcal{A},\ell}=\hat{\eta}_k^{\prime\mathcal{B},\ell}=\hat{\eta}_k^{\prime\mathcal{C},\ell}=0,\quad  \forall\,\ell\in\{x,y,z\}, 
    \end{align*}}
    where $\delta_k=\delta_k^{f,x}+\delta_k^{g,x}+\delta_k^{f,y}+\delta_k^{g,y}+\delta_k^{g,z}$, $L''=\max\{(L^f)^2,\,2(L_2^gR)^2,\,2(L_1^g)^2\}$.
\end{lemma}
\begin{proof}
According to Lemma 2 in \cite{li2021zerosarah}, we can directly obtain that the SE-VRC condition holds and the selection of the parameters.
For the sake of completeness in the proof, we provide a detailed proof in BLO setting for $\mathbb{E}[\|\hat{v}_{k}^x-D^x_k\|^2]$ and $E_{k,f}^x$ as an example.

First, for the descent property of $\mathbb{E}[\|\hat{v}_{k}^x-D^x_k\|^2]$, from the update rule of \( \hat{v}_k^x\), we have
{
    \setlength{\abovedisplayskip}{3pt} 
    \setlength{\belowdisplayskip}{3pt}
\begin{align*}
       &  \mathbb{E}\left[\left\|v_{k+1}^x-D^x_{k+1}\right\|^2\right]
     =
     \mathbb{E}\Big[\Big\|
        (1-\bar{\rho}_{k+1})\left(v_{k}^x-D^x_k\right)
        +D^x_{k+1;I,J}-D^x_{k;I,J}
        +D^x_k-D^x_{k+1}
        \\
        &\quad\quad\quad\quad\quad\quad\quad\quad\quad\quad\quad
        +\bar{\rho}_{k+1}\left(
        D^x_{k;I,J} - D^x_{k}
        -\hat{D}^x_{k;I,J}(w,\tilde{w})+\hat{D}^x_{k;[n],[m]}(w,\tilde{w})
    \right)
     \Big\|^2\Big]
\\
     \leq&
     (1-\bar{\rho}_{k+1})^2\mathbb{E}[\|v_{k}^x-D^x_k\|^2]
     +2\mathbb{E}\left[\|D^x_{k+1;I,J}-D^x_{k;I,J}
        +D^x_k-D^x_{k+1}\|^2\right]\\
    & + 2\bar{\rho}_{k+1}^2\mathbb{E}\left[\|D^x_{k;I,J} - D^x_{k}
        -\hat{D}^x_{k;I,J}(w,\tilde{w})+\hat{D}^x_{k;[n],[m]}(w,\tilde{w})\|^2\right]\\
    \leq&
    (1-\bar{\rho}_{k+1})^2\mathbb{E}[\|v_{k}^x-D^x_k\|^2]
     +\frac{2}{b^2}\sum_{i\in I}\mathbb{E}\left[\|\nabla_1 F_i(x_{k+1},y_{k+1}) -
    \nabla_1 F_i(x_k,y_k)\|^2\right]\\
    &+\frac{2}{b^2}\sum_{j\in J}\mathbb{E}\left[\|\nabla_{12}^2 G_j(x_{k+1},y_{k+1})z_{k+1}
    - \nabla_{12}^2 G_j(x_k,y_k)z_k\|^2\right]\\
    & + \frac{2\bar{\rho}_{k+1}^2}{b}\frac{1}{n}\sum_{i=1}^n\mathbb{E}[\|\nabla_1 F_i(x_k,y_k)-\nabla_1 F_i(w_{k,i}^x,w_{k,i}^y)\|^2]
    \\
    & + \frac{2\bar{\rho}_{k+1}^2}{b}\frac{1}{m}\sum_{j=1}^m\mathbb{E}[\|\nabla_{12}^2 G_j(x_k,y_k)z_k-\nabla_{12}^2 G_j(\tilde{w}_{k,j}^x,\tilde{w}_{k,j}^y)\tilde{w}_{k,j}^z\|^2],
    \end{align*}}
where \( \hat{v}_k^x=v_k^x \), the first inequality uses that $D^x_{k+1;I,J}-D^x_{k;I,J}$ is an unbiased estimator of
$D^x_{k+1}-D^x_k$ and that $D^x_{k;I,J}-\hat{D}^x_{k;I,J}(w,\tilde{w})$ is an unbiased estimator of
$D^x_k-\hat{D}^x_{k;[n],[m]}(w,\tilde{w})$, so the cross
terms vanish; we then apply $\|a+b\|^2\le 2\|a\|^2+2\|b\|^2$.
The second inequality follows from  using the independence
between the $i$- and $j$-samples, and applying
$\mathbb{E}\,[\|X-\mathbb{E}X\|^2]\le \mathbb{E}\,[\|X\|^2]$.\\
Following the derivations in \eqref{eq_fi_xk}--\eqref{eq_gj_meomry}, we further obtain
{
    \setlength{\abovedisplayskip}{1pt} 
    \setlength{\belowdisplayskip}{3pt}
     \begin{align*}
         \mathbb{E}\left[\left\|v_{k+1}^x-D^x_{k+1}\right\|^2\right]
     &\leq
     (1-\bar{\rho}_{k+1})^2\mathbb{E}\left[\left\|v_{k}^x-D^x_{k}\right\|^2\right]\nonumber
+\frac{2L''\bar{\rho}_{k+1}^2}{b}\delta_k
\nonumber\\
&\quad
+\frac{4}{b}\left(c_1 \alpha_k^2 \mathbb{E}\left[\left\|v_k^x\right\|^2\right] +  c_1 \beta_k^2 \mathbb{E}\left[\left\|v_k^y\right\|^2\right] +  c_2 \gamma_k^2 \mathbb{E}\left[\left\|v_k^z\right\|^2\right]
    \right),
    \end{align*}}
where $\delta_k=\delta_k^{f,x}+\delta_k^{g,x}+\delta_k^{f,y}+\delta_k^{g,y}+\delta_k^{g,z}$.

Next, for the descent property of $\delta_k^{f,x}$,
we have
\begin{align*}
&E_k[\|x_{k+1}-w_{k+1,i}^x\|^2]
=\left(1-\frac{b}{n}\right)E_k[\|x_{k+1}-w_{k,i}^x\|^2]\\
&\leq
\left(1-\frac{b}{n}\right)\left(1+\frac{b}{2n}\right)E_k[\|x_{k}-w_{k,i}^x\|^2]
+\left(1-\frac{b}{n}\right)(1+\frac{2n}{b})E_k[\|x_{k+1}-x_{k}\|^2]\\
&\leq
\left(1-\frac{b}{2n}\right)E_k[\|x_{k}-w_{k,i}^x\|^2]
+\frac{3n\alpha_k^2}{b}E_k[\|v_k^x\|^2].
\end{align*}
Taking the full expection yields
{
    \setlength{\abovedisplayskip}{1pt} 
    \setlength{\belowdisplayskip}{1pt}
\begin{align*}
\delta_{k+1}^{f,x}
     &=\frac{1}{n}\sum_{i=1}^n \mathbb{E}[\|x_{k+1}-w_{k+1,i}^x\|^2]
     \leq\left(1-\frac{b}{2n}\right)\delta_k^{f,x}
+\frac{3n\alpha_k^2}{b}\mathbb{E}[\|v_k^x\|^2].
\end{align*}}
Similarly, the same reasoning applies to the other terms in $\delta_k$. Then we have
\begin{align*}
\delta_{k+1}
\leq {}& \left(1-\frac{b}{2N}\right)\delta_k
+\frac{3N\alpha_k^2}{b}\mathbb{E}[\|v_k^x\|^2] \\
&+\frac{3N\beta_k^2}{b}\mathbb{E}[\|v_k^y\|^2]
+\frac{3N\gamma_k^2}{b}\mathbb{E}[\|v_k^z\|^2].
\end{align*}
\end{proof}

\section{Proof of Theorem \ref{theorem_biased}}

\begin{proof}
Since the stochastic estimator $\mathcal{A}$ is biased, it follows that $v_k^x = \hat{v}_k^x$, and we set $A''_k = 0$. Under the conditions of Assumptions \ref{assump UL}, \ref{assump LL}, and the SE-VRC condition, we have that Corollary \ref{H_biased}, Lemmas \ref{y-y*_biased} and \ref{z-z*_biased}, and the equations \eqref{lemma_vr_a}-\eqref{lemma_memory_a} and \eqref{lemma_vr_b}-\eqref{lemma_memory_c} hold. Based on these results, and by combining Lemma \ref{dL} with the following inequalities
\begin{align*}
    \mathbb{E}\left[\left\|v_k^y\right\|^2\right]
    &\leq
    2\mathbb{E}\left[\left\|v_k^y-D^y_k\right\|^2\right]
    +2\mathbb{E}\left[\left\|D^y_k\right\|^2\right],\\
    \mathbb{E}\left[\left\|v_k^z\right\|^2\right]
    &\leq
    2\mathbb{E}\left[\left\|v_k^z-D^z_k\right\|^2\right]
    +2\mathbb{E}\left[\left\|D^z_k\right\|^2\right].
\end{align*}
we obtain
{\setlength{\abovedisplayskip}{3pt}
\setlength{\belowdisplayskip}{3pt}
\begin{align*}
&L_{k+1}-L_k
\leq
-\frac{\alpha_k}{2}
        \mathbb{E}\left[\left\|\nabla H\left(x_k\right)\right\|^2\right]
+\text{Part}_{1}
        \mathbb{E}\left[\left\|v_k^x\right\|^2\right]
+\text{Part}_{2}
        \mathbb{E}\left[\left\|y_k-y^*_k\right\|^2\right]
        \\
&\quad
+\text{Part}_{3}
        \mathbb{E}\left[\left\|z_k-z^*_k\right\|^2\right]
+\text{Part}_{4}
        \mathbb{E}\left[\left\|D^x_k-v_k^x\right\|^2\right]
+\text{Part}_{5}
       \mathbb{E}\left[\left\|v_k^y-D^y_k\right\|^2\right]
              \\
&\quad
+\text{Part}_{6}
       \mathbb{E}[\| D^z_k - v_k^z\|^2]
+\text{Part}_{7}
        \delta_k^{\mathcal{A}}
+\text{Part}_{8}
        \delta_k^{\mathcal{B}}
+\text{Part}_{9}
        \delta_k^{\mathcal{C}}\\
&\quad
+ A_{k+1}\Delta_k^{\mathcal{A}}
+ B_{k+1}\Delta_k^{\mathcal{B}}
+ C_{k+1}\Delta_k^{\mathcal{C}},
\end{align*}}
where
\begin{align*}
\text{Part}_{1}=
    & -\frac{\alpha_k}{4}
    + c_5C_y\frac{\alpha_k^2}{\beta_k}
    + c_7 C_z\frac{\alpha_k^2}{\gamma_k}
    + 2(A^{\prime}_{k+1}\hat{\eta}^{\prime \mathcal{A},x}
    +B^{\prime}_{k+1}\hat{\eta}^{\prime \mathcal{B},x}
     + C^{\prime}_{k+1}\hat{\eta}^{\prime \mathcal{C},x})\alpha_k^2,
    \\&
    + \left(A_{k+1}\eta_k^{\mathcal{A}}c_1+A'_{k+1}\hat{\eta}_k^{\mathcal{A},x}
    +
    B_{k+1}c_2\eta_k^{\mathcal{B}} +B'_{k+1}\hat{\eta}_k^{\mathcal{B},x}
    +
    C_{k+1}c_1\eta_k^{\mathcal{C}}+C'_{k+1}\hat{\eta}_k^{\mathcal{C},x}
    \right) \alpha_k^2,
\\[3pt]
\text{Part}_{2}=
    &
    3c_1{\alpha_k} - c_3 C_y \beta_k
    +\frac{18C_zc_1}{\mu}\gamma_k
    +
    c_2\left( 2B_{k+1}c_2\eta_k^{\mathcal{B}}+B'_{k+1}(2\hat{\eta}_k^{\mathcal{B},y}+\hat{\eta}_k^{\prime\mathcal{B},y}) \right) \beta_k^2
    \\
    &
    +
    c_2\left(2A_{k+1}c_1\eta_k^{\mathcal{A}}+A'_{k+1}(2\hat{\eta}_k^{\mathcal{A},y}+\hat{\eta}_k^{\prime\mathcal{A},y})
    + 2C_{k+1}c_1\eta_k^{\mathcal{C}}+C'_{k+1}
    (2\hat{\eta}_k^{\mathcal{C},y}+\hat{\eta}_k^{\prime\mathcal{C},y}) \right) \beta_k^2
    \\&
    +L_z^2\left(2A_{k+1}c_2\eta_k^{\mathcal{A}}
    +A'_{k+1}(2\hat{\eta}_k^{\mathcal{A},z}+\hat{\eta}_k^{\prime\mathcal{A},z} )
    +2C_{k+1}c_2\eta_k^{\mathcal{C}}+C'_{k+1}(2\hat{\eta}_k^{\mathcal{C},z}+\hat{\eta}_k^{\prime\mathcal{C},z})
    \right) \gamma_k^2
    ,\end{align*}\begin{align*}
\text{Part}_{3}=
    &
    3c_2{\alpha_k}
    -\frac{C_z\mu}{4}\gamma_k
    +L_z^2\left(2A_{k+1}c_2\eta_k^{\mathcal{A}} +2A'_{k+1}\hat{\eta}_k^{\mathcal{A},z} +A'_{k+1}\hat{\eta}_k^{\prime\mathcal{A},z} \right) \gamma_k^2
    \\&+L_z^2\left(2C_{k+1}c_2\eta_k^{\mathcal{C}}+C'_{k+1}(2\hat{\eta}_k^{\mathcal{C},z} +\hat{\eta}_k^{\prime\mathcal{C},z} )\right) \gamma_k^2,
    \\[3pt]
\text{Part}_{4}=
&\alpha_k
+A_{k+1}\theta_k^{\mathcal{A}}-A_k + 2A^{\prime}_{k+1}\hat{\eta}^{\prime \mathcal{A},x}\alpha_k^2+ 2B^{\prime}_{k+1}\hat{\eta}^{\prime \mathcal{B},x}\alpha_k^2
     + 2C^{\prime}_{k+1}\hat{\eta}^{\prime \mathcal{C},x}\alpha_k^2,
\\[5pt]
\text{Part}_{5}=
&c_4\beta_k C_y
+B_{k+1}\theta_k^{\mathcal{B}}-B_k
+ 2\left(A_{k+1}c_1\eta_k^{\mathcal{A}}+A'_{k+1}\hat{\eta}_k^{\mathcal{A},y} \right)  \beta_k^2
\\ &
+2 \left(B_{k+1}c_2\eta_k^{\mathcal{B}}+B'_{k+1}\hat{\eta}_k^{\mathcal{B},y} \right)  \beta_k^2
+2\left(C_{k+1}c_1\eta_k^{\mathcal{C}}+C'_{k+1}\hat{\eta}_k^{\mathcal{C},y} \right)  \beta_k^2,
\\
\text{Part}_{6}=
&c_{10}\gamma_k C_z
+C_{k+1}\theta_k^{\mathcal{C}}-C_k
+2\left(A_{k+1}c_2\eta_k^{\mathcal{A}}+A'_{k+1}\hat{\eta}_k^{\mathcal{A},z} \right) \gamma_k^2
\\ &
+2\left(C_{k+1}c_2\eta_k^{\mathcal{C}}+C'_{k+1}\hat{\eta}_k^{\mathcal{C},z} \right) \gamma_k^2,
\\[5pt]
\text{Part}_{7}=
&A_{k+1}\tau_k^{\mathcal{A}}
+A'_{k+1}\lambda_k^{\mathcal{A}}-A'_{k},
\\[5pt]
\text{Part}_{8}=
&B_{k+1}\tau_k^{\mathcal{B}}
+B'_{k+1}\lambda_k^{\mathcal{B}}-B'_{k},
\\[5pt]
\text{Part}_{9}=&
C_{k+1}\tau_k^{\mathcal{C}}
+C'_{k+1}\lambda_k^{\mathcal{C}}-C'_{k}.
\end{align*}

It remains to verify that, under the conditions stated in Theorem~\ref{theorem_biased},
$
\text{Part}_i \le 0,$ $i=1,2,\ldots,9.
$
Granting this claim for the moment, we obtain
\begin{align*}
L_{k+1}-L_k
\leq
-\frac{\alpha_k}{2}
        \mathbb{E}\left[\left\|\nabla H\left(x_k\right)\right\|^2\right]
+ A_{k+1}\Delta_k^{\mathcal{A}}
+ B_{k+1}\Delta_k^{\mathcal{B}}
+ C_{k+1}\Delta_k^{\mathcal{C}}.
\end{align*}
By rearranging and summing up, we obtain
\begin{align*}
\sum_{k=0}^{K-1}
{\alpha_k}
        \mathbb{E}\left[\left\|\nabla H\left(x_k\right)\right\|^2\right]
        \leq &
2L_0-2L_{K}
+ 2\sum_{k=0}^{K-1}\left(A_{k+1}\Delta_k^{\mathcal{A}}
+B_{k+1}\Delta_k^{\mathcal{B}}
+C_{k+1}\Delta_k^{\mathcal{C}}\right),
\end{align*}
which leads to
\begin{align*}
 \inf_{k< K} \mathbb{E}[\|\nabla H\left(x_k\right)\|^2]
\leq
\frac{ 2(L_{0}-H_*)}{\sum_{k=0}^{K-1}\alpha_k}
+\frac{2}{\sum_{k=0}^{K-1}\alpha_k}
\sum_{k=0}^{K-1}\left(A_{k+1}\Delta_k^{\mathcal{A}}+
B_{k+1}\Delta_k^{\mathcal{B}}
+C_{k+1}\Delta_k^{\mathcal{C}}\right).
\end{align*}
For completeness, we will now verify $\text{Part}_i \leq 0$ for $i = 1, 2, \dots, 9$ individually.

\begin{itemize}
    \item[\textcircled{1}] First, using \eqref{coupling_biased_1}, we have the following implication:
\end{itemize}
{\setlength{\abovedisplayskip}{1pt}
\setlength{\belowdisplayskip}{3pt}
\begin{align*}
&\left(B_{k+1}\eta_k^{\mathcal{B}}+B'_{k+1}\hat{\eta}_k^{\mathcal{B},y} \right)  \alpha_k
\leq \left(B_{k+1}\eta_k^{\mathcal{B}}+B'_{k+1}\hat{\eta}_k^{\mathcal{B},y} \right)  m_{\alpha\beta}\beta_k
\leq \frac{1}{64\tilde{c}_2}.
\end{align*}}
    Similarly, we can derive that
\begin{align*}
&\left(A_{k+1}\eta_k^{\mathcal{A}}+A'_{k+1}\hat{\eta}_k^{\mathcal{A},x} \right)  \alpha_k
\leq
\frac{1}{96\tilde{c}_1},
\,\,
\left(A^{\prime}_{k+1}\hat{\eta}^{\prime \mathcal{A},x}
    +B^{\prime}_{k+1}\hat{\eta}^{\prime \mathcal{B},x}
     + C^{\prime}_{k+1}\hat{\eta}^{\prime \mathcal{C},x}\right)\alpha_k\leq
     \frac{1}{64},
\\&
\left(B_{k+1}\eta_k^{\mathcal{B}} +B'_{k+1}\hat{\eta}_k^{\mathcal{B},x}\right)  \alpha_k \leq  \frac{1}{96\tilde{c}_2},
\quad
\left(C_{k+1}\eta_k^{\mathcal{C}}+C'_{k+1}\hat{\eta}_k^{\mathcal{C},x} \right)  \alpha_k
\leq
\frac{1}{96\tilde{c}_1},
\\
&
\left(A_{k+1}\eta_k^{\mathcal{A}}+A'_{k+1}\hat{\eta}_k^{\mathcal{A},z} \right)  \gamma_k
\leq \min\left\{
\frac{1}{4\tilde{c}_2},\,
\frac{c_3^2}{72\tilde{c}_2L_z^2},\,
\frac{\mu}{40\tilde{c}_2 c_{10}L_z^2}
\right\},
\end{align*}
\begin{align*}&
\left(C_{k+1}\eta_k^{\mathcal{C}}+C'_{k+1}\hat{\eta}_k^{\mathcal{C},z} \right)  \gamma_k
\leq \min\left\{
\frac{1}{4\tilde{c}_2},\,
\frac{c_3^2}{72\tilde{c}_2L_z^2},\,
\frac{\mu}{40\tilde{c}_2 c_{10}L_z^2}
\right\},
\\&
A'_{k+1}\hat{\eta}_k^{\prime\mathcal{A},z}\gamma_k
\leq \min\left\{\frac{c_3^2}{36L_z^2}, \frac{\mu}{20c_{10}L_z^2}\right\},
\,
C'_{k+1}\hat{\eta}_k^{\prime\mathcal{C},z}\gamma_k
\leq \min\left\{\frac{c_3^2}{36L_z^2}, \frac{\mu}{20c_{10}L_z^2}\right\}.
\end{align*}

\begin{itemize}
    \item[\textcircled{2}] For \(\text{Part}_i\) \((i =  7, 8, 9)\), by applying the SE-CC Condition for $\mathcal{A}$, $\mathcal{B}$ and $\mathcal{C}$, we have \(\text{Part}_i \leq 0\) for \(i =  7, 8, 9\).
Moreover, by setting $ C_y  = {1}/{c_4}$, $ C_z=1/c_{10}$, we obtain
\begin{align}
    c_4\beta_k C_y
+B_{k+1}\theta_k^{\mathcal{B}}-B_k
\leq
-{\beta_k},\quad
c_{10}\gamma_k C_z
+C_{k+1}\theta_k^{\mathcal{C}}-C_k
\leq
-{\gamma_k}.\label{proof_biased_condition_BC}
\end{align}
\end{itemize}

\begin{itemize}
    \item[\textcircled{3}] For $\text{Part}_1$, since
\end{itemize}
\begin{align*}
    &\alpha_k
    \leq
    \min\left\{
    \frac{3}{8L_{y^*}^2}\beta_k,\,
    \frac{3}{16L_{z^*}^2}\gamma_k
    \right\},
    \quad
\left(A_{k+1}\eta_k^{\mathcal{A}}+A'_{k+1}\hat{\eta}_k^{\mathcal{A},x} \right) \alpha_k\leq \frac{1}{96\tilde{c}_1},\\
&
\left(A^{\prime}_{k+1}\hat{\eta}^{\prime \mathcal{A},x}
    +B^{\prime}_{k+1}\hat{\eta}^{\prime \mathcal{B},x}
     + C^{\prime}_{k+1}\hat{\eta}^{\prime \mathcal{C},x}\right)\alpha_k\leq
     \frac{1}{64},
\\
     &\left(B_{k+1}\eta_k^{\mathcal{B}} +B'_{k+1}\hat{\eta}_k^{\mathcal{B},x}\right)  \alpha_k \leq  \frac{1}{96\tilde{c}_2},
     \quad
     \left(C_{k+1}\eta_k^{\mathcal{C}}+C'_{k+1}\hat{\eta}_k^{\mathcal{C},x} \right) \alpha_k \leq \frac{1}{96\tilde{c}_1},
\end{align*}
it follows that
\begin{align*}
\text{Part}_1
&\leq
-\frac{\alpha_k}{16}
+2\Big(
A'_{k+1}\hat{\eta}^{\prime\mathcal{A},x}
+B'_{k+1}\hat{\eta}^{\prime\mathcal{B},x}
+C'_{k+1}\hat{\eta}^{\prime\mathcal{C},x}
\Big)\alpha_k^2\\
&\quad+
\tilde{c}_2\Big(B_{k+1}\eta_k^{\mathcal{B}}
+B'_{k+1}\hat{\eta}_k^{\mathcal{B},x}\Big)\alpha_k^2\\
&\quad+
\tilde{c}_1\Big(A_{k+1}\eta_k^{\mathcal{A}}
+A'_{k+1}\hat{\eta}_k^{\mathcal{A},x}
+C_{k+1}\eta_k^{\mathcal{C}}
+C'_{k+1}\hat{\eta}_k^{\mathcal{C},x}\Big)\alpha_k^2\\
&\leq
-\frac{\alpha_k}{32}
+\tilde{c}_2\Big(B_{k+1}\eta_k^{\mathcal{B}}
+B'_{k+1}\hat{\eta}_k^{\mathcal{B},x}\Big)\alpha_k^2\\
&\quad+
\tilde{c}_1\Big(C_{k+1}\eta_k^{\mathcal{C}}
+C'_{k+1}\hat{\eta}_k^{\mathcal{C},x}\Big)\alpha_k^2\\
&\quad+
\tilde{c}_1\Big(A_{k+1}\eta_k^{\mathcal{A}}
+A'_{k+1}\hat{\eta}_k^{\mathcal{A},x}\Big)\alpha_k^2\\
&\leq 0.
\end{align*}
\begin{itemize}
    \item[\textcircled{4}] For $\text{Part}_2$, under the parameter choices and bounds
\end{itemize}
\begin{align*}
&C_y=\frac{1}{c_4}\,\left(\Rightarrow c_3C_y=\frac{c_3^2}{3}\right),\quad
\alpha_k\leq \frac{c_3^2}{108c_1}\beta_k,\quad
\gamma_k\leq\frac{ c_3^2}{54c_1}\beta_k,
\quad
B'_{k+1}\hat{\eta}_k^{\prime\mathcal{B},y} \beta_k\leq \frac{c_3^2}{36c_2},
\\&
C'_{k+1}\hat{\eta}_k^{\prime\mathcal{C},y}\beta_k\leq \frac{c_3^2}{36c_2},
\quad
C'_{k+1}\hat{\eta}_k^{\prime\mathcal{C},z}\gamma_k\leq \frac{c_3^2}{36L_z^2},
\quad
A'_{k+1}\hat{\eta}_k^{\prime\mathcal{A},z}\gamma_k\leq \frac{c_3^2}{36L_z^2},
\\&
A'_{k+1}\hat{\eta}_k^{\prime\mathcal{A},y}\beta_k\leq \frac{c_3^2}{36c_2},
\quad
\left(A_{k+1}\eta_k^{\mathcal{A}}+A'_{k+1}\hat{\eta}_k^{\mathcal{A},y} \right) \beta_k \leq \frac{c_3^2}{72\tilde{c}_1c_2},
\\
&
\left(B_{k+1}\eta_k^{\mathcal{B}}+B'_{k+1}\hat{\eta}_k^{\mathcal{B},y} \right) \beta_k \leq
\frac{c_3^2}{72\tilde{c}_2^2},
\quad
\left(C_{k+1}\eta_k^{\mathcal{C}}+C'_{k+1}\hat{\eta}_k^{\mathcal{C},y} \right) \beta_k \leq
\frac{c_3^2}{72\tilde{c}_1c_2},
\\
&
\left(C_{k+1}\eta_k^{\mathcal{C}}+C'_{k+1}\hat{\eta}_k^{\mathcal{C},z} \right) \gamma_k \leq \frac{c_3^2}{72\tilde{c}_2L_z^2},
\quad
\left(A_{k+1}\eta_k^{\mathcal{A}}+A'_{k+1}\hat{\eta}_k^{\mathcal{A},z} \right) \gamma_k \leq \frac{c_3^2}{72\tilde{c}_2L_z^2}.
\end{align*}
Therefore, the sum of all nonnegative terms in $\text{Part}_2$ is at most $c_3 C_y \beta_k$, and hence $\text{Part}_2 \le 0$.

\begin{itemize}
    \item[\textcircled{5}] For $\text{Part}_3$,
    under the parameter choices and bounds
\end{itemize}
\begin{align*}
&\alpha_k\leq \frac{\mu}{60c_2c_{10}}\gamma_k,
\quad
A'_{k+1}\hat{\eta}_k^{\prime\mathcal{A},z}\gamma_k\leq
\frac{\mu}{20c_{10}L_z^2},
\quad
C'_{k+1}\hat{\eta}_k^{\prime\mathcal{C},z}\gamma_k\leq
\frac{\mu}{20c_{10}L_z^2},
\\
&\left(A_{k+1}\eta_k^{\mathcal{A}}+A'_{k+1}\hat{\eta}_k^{\mathcal{A},z} \right) \gamma_k \leq \frac{\mu}{40\tilde{c}_2 c_{10}L_z^2},
\quad
\left(C_{k+1}\eta_k^{\mathcal{C}}+C'_{k+1}\hat{\eta}_k^{\mathcal{C},z} \right) \gamma_k \leq \frac{\mu}{40\tilde{c}_2c_{10}L_z^2}.
\end{align*}
By substituting the above bounds into the definition of $\text{Part}_3$ and summing the resulting termwise estimates, we conclude that the total contribution of all nonnegative terms is dominated by the negative term $-(C_z\mu/4)\gamma_k$. Hence, $\text{Part}_3 \le 0$.

\begin{itemize}
    \item[\textcircled{6}] For $\text{Part}_4$:
    by the SE-CC condition and
    $(A^{\prime}_{k+1}\hat{\eta}^{\prime \mathcal{A},x}
    +B^{\prime}_{k+1}\hat{\eta}^{\prime \mathcal{B},x}
     + C^{\prime}_{k+1}\hat{\eta}^{\prime \mathcal{C},x})\alpha_k\leq {1}/{2}$, we have $\text{Part}_4\leq 0$.
\end{itemize}

\begin{itemize}
    \item[\textcircled{7}] For $\text{Part}_5$:
    By combining \eqref{proof_biased_condition_BC} with the following conditions
\end{itemize}
{\setlength{\abovedisplayskip}{-5pt}
\setlength{\belowdisplayskip}{3pt}
\begin{align*}
&\left(A_{k+1}\eta_k^{\mathcal{A}}+A'_{k+1}\hat{\eta}_k^{\mathcal{A},y} \right)  \beta_k \leq {1}/{(6\tilde{c}_1)},\quad
\left(B_{k+1}\eta_k^{\mathcal{B}}+B'_{k+1}\hat{\eta}_k^{\mathcal{B},y} \right)  \beta_k \leq {1}/{(6\tilde{c}_2)},\\
&
\left(C_{k+1}\eta_k^{\mathcal{C}}+C'_{k+1}\hat{\eta}_k^{\mathcal{C},y} \right)  \beta_k \leq {1}/{(6\tilde{c}_1)}.
\end{align*}}
Plugging these bounds into $\text{Part}_5$ shows that the sum of all nonnegative terms is at most $\beta_k/2$. Therefore, $\text{Part}_5 \le 0$.

\begin{itemize}
\item[\textcircled{8}] For $\text{Part}_6$:
Similarly, the validity of inequality $\text{Part}_6\leq 0$ relies on inequality \eqref{proof_biased_condition_BC} and
\[
\left(A_{k+1}\eta_k^{\mathcal{A}}+A'_{k+1}\hat{\eta}_k^{\mathcal{A},z} \right) \gamma_k \leq {1}/{(4\tilde{c}_2)},
\quad
\left(C_{k+1}\eta_k^{\mathcal{C}}+C'_{k+1}\hat{\eta}_k^{\mathcal{C},z} \right) \gamma_k \leq {1}/{(4\tilde{c}_2)}.
\]
\end{itemize}
Thus, the proof of the theorem is completed.
\end{proof}

\section{Proof of Theorem \ref{theorem_unbiased}}
\begin{proof}
Let \( C_y = C_z = 1/2 \). Then, applying Lemmas \ref{H}, \ref{lemma_ma}, \ref{y-y*_unbiased}, \ref{z-z*_unbiased} and combining with \eqref{lemma_vr_a}-\eqref{lemma_memory_a} and \eqref{lemma_vr_b}-\eqref{lemma_memory_c}, we obtain
\begin{align*}
    L_{k+1}-L_k
    \leq &
\text{Part}_0'
            \mathbb{E}\left[\left\|\nabla H\left(x_k\right)\right\|^2\right]
+\text{Part}_1'
            \mathbb{E}\left[\left\|\nabla H\left(x_k\right)-v_k^x\right\|^2\right]
+\text{Part}_2'
            \mathbb{E}\left[\left\|v_k^x\right\|^2\right]\\
&
+\text{Part}_3'
            \mathbb{E}\left[\left\|y_k-y^*_k\right\|^2\right]
+\text{Part}_4'
            \mathbb{E}\left[\left\|z_k-z^*_k\right\|^2\right]
+ \text{Part}_5'
        \mathbb{E}\left[\left\| \hat{v}_{k}^x-  D^x_k\right\|^2\right]\\
&
+\text{Part}_6'
        \mathbb{E}\left[\left\|D^y_k-v_k^y\right\|^2\right]
+\text{Part}_7'
        \mathbb{E}\left[\left\|D^z_k-v_k^z\right\|^2\right]
+\text{Part}_8'
        \delta_k^{\mathcal{A}}
+\text{Part}_9'
        \delta_k^{\mathcal{B}}
        \\
&
+\text{Part}_{10}'
        \delta_k^{\mathcal{C}}
+A_{k+1}''\rho_k^2\Delta_k^{\mathcal{A}}
+A_{k+1}\Delta_k^{\mathcal{A}}
+B_{k+1}\Delta_k^{\mathcal{B}}
+C_{k+1}\Delta_k^{\mathcal{C}},
\end{align*}
where\begin{align*}
&\text{Part}_0'=
    -{\alpha_k}/{2}
    +2\left(A_{k+1}'\hat{\eta}^{\prime\mathcal{A},x}+B_{k+1}'\hat{\eta}^{\prime\mathcal{B},x}+C_{k+1}'\hat{\eta}^{\prime\mathcal{C},x}\right)\alpha_k^2,
\\&
\text{Part}_1'=
    {\alpha_k}/{2}
    +A_{k+1}''\left(1-\rho_k\right)-A_{k}'',
\\
&\text{Part}_2'=
    -\frac{\alpha_k}{4}
    +\frac{ c_6\alpha_k^2}{2\beta_k }
    +\frac{c_9 \alpha_k^2}{2\gamma_k }
    +\frac{3 \alpha_k^2\left(L^H\right)^2}{\rho_k}A_{k+1}''
    +12c_1 \rho_k\alpha_k^2A_{k+1}''
    +c_1\rho_k^2A_{k+1}''\eta_k^{\mathcal{A}}\alpha_k^2
    \\
    &\quad
    + \left(\eta_k^{\mathcal{A}}A_{k+1}c_1 + \hat{\eta}^{\mathcal{A},x}_k A'_{k+1}
    +
    \eta_k^{\mathcal{C}}C_{k+1}c_1+\hat{\eta}_k^{\mathcal{C},x}C'_{k+1}
    + \eta^{\mathcal{B}}_k B_{k+1}c_2 + \hat{\eta}^{\mathcal{B},x}_k B'_{k+1} \right)\alpha_k^2,
    \end{align*}
    \begin{align*}
    & \text{Part}_3'=
    -\mu  \beta_k/2
    +c_8  \gamma_k/2
    +c_2 \left( 2{B}_{k+1}\eta_k^{\mathcal{B}}c_2 + B'_{k+1}(2\hat{\eta}^{\mathcal{B},y}_k + \hat{\eta}^{\prime\mathcal{B},y}_k)\right)  \beta_k^2
    \\&\quad
    +c_2\left(2A_{k+1}c_1\eta_k^{\mathcal{A}} + A'_{k+1}(2\hat{\eta}^{\mathcal{A},y}_k +\hat{\eta}^{\prime\mathcal{A},y}_k)
    +2{C}_{k+1}\eta_k^{\mathcal{C}}c_1 + C'_{k+1}(2\hat{\eta}^{\mathcal{C},y}_k + \hat{\eta}^{\prime\mathcal{C},y}_k)
    \right)\beta_k^2
    \\
    & \quad
    +L_z^2 \left( 2\eta_k^{\mathcal{C}}C_{k+1}c_2 + C'_{k+1}(2\hat{\eta}^{\mathcal{C},z}_k + \hat{\eta}^{\prime\mathcal{C},z}_k )+2\eta_k^{\mathcal{A}}A_{k+1}c_2 + A'_{k+1}(2\hat{\eta}^{\mathcal{A},z}_k + \hat{\eta}^{\prime\mathcal{A},z}_k )\right)  \gamma_k^2
    \\
    &\quad
    +24c_1c_2\rho_k\beta_k^2A_{k+1}''
    +2 c_1c_2\rho_k^2A_{k+1}''\eta_k^{\mathcal{A}}\beta_k^2
    +2c_2 L_z^2\rho_k^2A_{k+1}''\eta_k^{\mathcal{A}}\gamma_k^2
    +24c_2 L_z^2\rho_k\gamma_k^2A_{k+1}''
    \\
    &\quad+9c_1\rho_kA_{k+1}''
    +6c_1\left(\hat{\eta}_k^{\prime \mathcal{A},x}A_{k+1}'+\hat{\eta}_k^{\prime \mathcal{B},x}B_{k+1}'+\hat{\eta}_k^{\prime \mathcal{C},x}C_{k+1}'\right)\alpha_k^2
    ,
\\[5pt]
    &\text{Part}_4'=
    9c_2\rho_kA_{k+1}''
    -\mu  \gamma_k/2
    +24c_2 L_z^2\rho_k\gamma_k^2A_{k+1}''
    +2c_2 L_z^2\rho_k^2A_{k+1}''\eta_k^{\mathcal{A}}\gamma_k^2
    \\
    &\quad\quad\quad
    + L_z^2 \left(2c_2\eta_k^{\mathcal{A}}A_{k+1} + A'_{k+1}(2\hat{\eta}^{\mathcal{A},z}_k + \hat{\eta}^{\prime\mathcal{A},z}_k)
    +2c_2\eta_k^{\mathcal{C}}C_{k+1} + C'_{k+1}(2\hat{\eta}^{\mathcal{C},z}_k + \hat{\eta}^{\prime\mathcal{C},z}_k)
    \right)  \gamma_k^2
     \\
    &\quad\quad\quad
    +
     6c_2\left(\hat{\eta}_k^{\prime \mathcal{A},x}A_{k+1}'+\hat{\eta}_k^{\prime \mathcal{B},x}B_{k+1}'+\hat{\eta}_k^{\prime \mathcal{C},x}C_{k+1}'\right)\alpha_k^2,
\\[5pt]
    &\text{Part}_5'=
    \rho_k^2A_{k+1}''\theta_k^{\mathcal{A}}
+A_{k+1}\theta_k^{\mathcal{A}}-A_k,
\\
    &
\text{Part}_6'=
     \beta_k^2
    +\theta_k^{\mathcal{B}}B_{k+1}-B_k
    +24c_1 \rho_k\beta_k^2A_{k+1}''
    +2c_1\rho_k^2A_{k+1}''\eta_k^{\mathcal{A}}\beta_k^2
     \\
    &
    \quad\quad\quad
    +2\left(\eta_k^{\mathcal{B}}B_{k+1}c_2 + \hat{\eta}^{\mathcal{B},y}_k B'_{k+1}
    + \eta_k^{\mathcal{A}}A_{k+1}c_1 + \hat{\eta}^{\mathcal{A},y}_k A'_{k+1}
    +\eta_k^{\mathcal{C}}C_{k+1} c_1+\hat{\eta}_k^{\mathcal{C},y}C'_{k+1} \right) \beta_k^2,
\\[5pt]
    &
\text{Part}_7'=
     \gamma_k^2
    +\theta_k^{\mathcal{C}}C_{k+1}-C_k
    +2c_2\rho_k^2A_{k+1}''\eta_k^{\mathcal{A}}\gamma_k^2
    +24c_2\rho_k\gamma_k^2A_{k+1}''
    \\
    &
    \quad\quad\quad
    +2 \left(A_{k+1}\eta_k^{\mathcal{A}}c_2 + \hat{\eta}^{\mathcal{A},z}_k A'_{k+1}
    + \eta_k^{\mathcal{C}}C_{k+1}c_2 +\hat{\eta}_k^{\mathcal{C},z}C'_{k+1}\right)  \gamma_k^2
    ,
    \\
    &
\text{Part}_8'=
\rho_k^2 A_{k+1}''\tau_k^{\mathcal{A}}
+\tau_k^{\mathcal{A}}A_{k+1}
+\lambda_k^{\mathcal{A}} A'_{k+1}-A'_k,
    \\
        &
\text{Part}_9'=
    \tau_k^{\mathcal{B}}B_{k+1}
    +\lambda_k^{\mathcal{B}} B'_{k+1}-B'_k,
\\
&
\text{Part}_{10}'=
    \tau_k^{\mathcal{C}}C_{k+1}
+\lambda_k^{\mathcal{C}} C'_{k+1}-C'_k.
\end{align*}
Following a similar approach to the proof of Theorem \ref{theorem_biased},
we have $\text{Part}_0' \leq -\frac{\alpha_k}{4}$ and
 $\text{Part}_i' \leq 0$ for $i = 1, 2, \dots, 10$.
Specifically, the UE-CC and UEMA-CC conditions ensure that $\text{Part}_i' \leq 0$ holds for $i = 1, 5, 8, 9, 10$. For the remaining $i$, following a similar process to \textcircled{1} in the proof of Theorem \ref{theorem_biased},
we can derive the following conditions:
\begin{align*}
&\hat{\eta}^{\prime\mathcal{A},x}_k A'_{k+1}\alpha_k
\leq \frac{1}{24},\quad
\hat{\eta}^{\prime\mathcal{B},x}_k B'_{k+1}\alpha_k
\leq \frac{1}{24},\quad
\hat{\eta}^{\prime\mathcal{C},x}_k C'_{k+1}\alpha_k
\leq \frac{1}{24},\quad
\hat{\eta}^{\prime\mathcal{B},y}_k B'_{k+1}\beta_k
\leq \frac{\mu}{16c_2},
\\&
    \left(\eta_k^{\mathcal{A}}A_{k+1} + \hat{\eta}^{\mathcal{A},x}_k A'_{k+1}\right)  \alpha_k \leq \frac{1}{32\tilde{c}_1},
\quad
\left(\eta_k^{\mathcal{A}}A_{k+1} + \hat{\eta}^{\mathcal{A},y}_k A'_{k+1}\right)
    \leq
       \frac{1}{10\tilde{c}_2}
       ,
\\&
     \left(\eta_k^{\mathcal{A}}A_{k+1} + \hat{\eta}^{\mathcal{A},z}_k A'_{k+1}\right)
    \leq
      \frac{1}{8\tilde{c}_2},
\quad
\hat{\eta}^{\prime\mathcal{A},y}_k A'_{k+1}\beta_k
\leq \frac{\mu}{16c_2},
\quad
\hat{\eta}^{\prime\mathcal{A},z}_k A'_{k+1}\gamma_k
\leq \frac{\mu}{10L_z^2},
\\&
    \left(\eta^{\mathcal{B}}_k B_{k+1} + \hat{\eta}^{\mathcal{B},x}_k B'_{k+1} \right)\alpha_k
    \leq
    \frac{1}{32\tilde{c}_2},
\quad
     \left(\eta_k^{\mathcal{B}}B_{k+1} + \hat{\eta}^{\mathcal{B},y}_k B'_{k+1}\right)
     \leq
     \frac{1}{10\tilde{c}_2},\\
&
    \left( \eta_k^{\mathcal{C}}C_{k+1}+\hat{\eta}_k^{\mathcal{C},x}C'_{k+1} \right)\alpha_k
     \leq
     \frac{1}{32\tilde{c}_1},
\quad
 \left( \eta_k^{\mathcal{C}}C_{k+1} +\hat{\eta}_k^{\mathcal{C},y}C'_{k+1} \right)  \leq
    \frac{1}{10\tilde{c}_2},
\\&
     \left( \eta_k^{\mathcal{C}}C_{k+1} +\hat{\eta}_k^{\mathcal{C},z}C'_{k+1} \right)  \leq
    \frac{1}{8\tilde{c}_2},
    \quad
\hat{\eta}^{\prime\mathcal{C},y}_k C'_{k+1}\beta_k
\leq \frac{\mu}{16c_2},
\quad
\hat{\eta}^{\prime\mathcal{C},z}_k C'_{k+1}\gamma_k
\leq \frac{\mu}{10L_z^2},
\\
 &
\rho_k\alpha_kA_{k+1}'' \leq \frac{1}{384c_1},
\quad
\rho_k^2\alpha_kA_{k+1}''\eta_k^{\mathcal{A}}
\leq\frac{1}{32c_1},
\quad
   \rho_k^2A_{k+1}''\eta_k^{\mathcal{A}}
   \leq
   \min
   \left\{
   \frac{1}{10c_1},\,
   \frac{1}{8c_2}
   \right\},
\end{align*}
which guarantee that $\text{Part}_i' \leq 0$ holds for $i = 2, 3, 4, 6, 7$ and $\text{Part}_0' \leq -\frac{\alpha_k}{4}$.
Similar to the proof of Theorem \ref{theorem_biased}, we obtain the final conclusion.
\end{proof}

\vskip 0.2in
\bibliographystyle{plainnat}
\bibliography{sample}

\end{document}